\documentclass[11pt,twoside]{amsart}

\usepackage{graphicx}

\usepackage{amsfonts}
\usepackage{graphicx}
\usepackage{courier}
\usepackage{multirow}
\usepackage{tikz}
\usepackage{footnote}
\usepackage{hyperref}
\usepackage{footnote}
\usepackage{amstext}
\usetikzlibrary{matrix}
\usepackage[margin=1.5in]{geometry}

\usepackage{flipbook}

\usepackage{fancyhdr}
\renewcommand{\headrulewidth}{0pt}
\newcommand{\noheader}{
    \renewcommand{\headrulewidth}{0pt}%
    \fancyhead{}
}
\newcommand{\footer}{%
    \fancyfoot[RO]{\setlength\unitlength{1cm}\begin{picture}(0,0)\put(0,-2.5){\scalebox{-1}[1]{\includegraphics[clip,width=3cm]{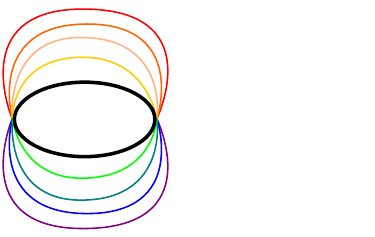}}}\end{picture}}
}

\newcommand{\fancyendfooter}{%
    \fancyfoot[RO]{\setlength\unitlength{1cm}\begin{picture}(0,0)\put(0,-2.5){\scalebox{-1}[1]{\includegraphics[clip,width=3cm]{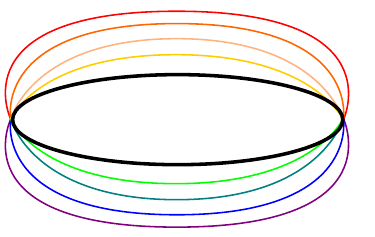}}}\end{picture}}
}

\fancypagestyle{title}{\noheader\footer}
\fancypagestyle{fancyend}{\fancyendfooter}

\makeatletter
\DeclareRobustCommand{\0}{%
  \nfss@text{%
    \sbox0{0}%
    \sbox2{/}%
    \sbox4{%
      \raise\dimexpr((\ht0-\dp0)-(\ht2-\dp2))/2\relax\copy2 %
    }%
    \ooalign{%
      \hfill\copy4 \hfill\cr
      \hfill0\hfill\cr
    }%
    \vphantom{0\copy4 }%
  }%
}
\makeatother

\newcommand{\btwo}{II$_B$ }
\newcommand{\bone}{I$_B$ }

\makesavenoteenv{tabular}
\makesavenoteenv{table}

\newtheorem{theorem}{Theorem}[section]
\newtheorem{lemma}[theorem]{Lemma}
\newtheorem{proposition}[theorem]{Proposition}

\newtheorem{corollary}[theorem]{Corollary}
\newtheorem{question}[theorem]{Question}

\newtheorem*{sliceribbon}{Slice-ribbon conjecture~\cite{fox}}
\newtheorem*{cassongordon}{Theorem 5.1 of~\cite{gordon}}
\newtheorem*{jefftheorem}{Theorem 1.1 of~\cite{jeff}}
\newtheorem*{maintheorem}{Theorem~\ref{maintheorem}}
\newtheorem*{secondtheorem}{Theorem~\ref{secondtheorem}}

\newtheorem*{fibrationthm}{Theorem~\ref{fibrationthm}}
\newtheorem*{mainquestion}{Question~\ref{mainquestion}}
\newtheorem*{suzuki}{Suzuki's unknotting conjecture~\cite{suzuki}}
\newtheorem*{unknotting}{Smooth $4$-dimensional unknotting conjecture {\rm{(see~\cite{kirby}, Problem 1.55(a))}}}

\newcounter{movie}

\theoremstyle{definition}
\newtheorem{definition}[theorem]{Definition}
\newtheorem{move}[movie]{Movie}
\newtheorem{remark}[theorem]{Remark}
\newtheorem*{claim}{Claim}

\newenvironment{innerproof}{\proof}{\endproof}

\newcommand{\id}{\text{id}}

\let\t\relax
\newcommand{\t}{\mathrm}

\newcommand{\boundary}{\partial}

\newcommand{\proj}{\t{proj}}

\newcommand{\Z}{\mathbb{Z}}
\newcommand{\R}{\mathbb{R}}

\newcommand{\pt}{\mathrm{pt}}
\newcommand{\into}{\hookrightarrow}

\newcommand{\F}{\mathcal{F}}
\newcommand{\G}{\mathcal{G}}

\begin{document}
\title[Extending fibrations to ribbon disk complements]{Extending fibrations of knot complements to ribbon disk complements}
\author[Maggie Miller]{Maggie Miller}                    

\address{Department of Mathematics, Princeton University (2018)}
\email{maggie.miller.math@gmail.com}

\maketitle
\thispagestyle{title}
\begin{abstract}
We show that if $K$ is a fibered ribbon knot in $S^3=\boundary B^4$ bounding a ribbon disk $D$, then given an extra transversality condition the fibration on $S^3\setminus\nu(K)$ extends to a fibration of $B^4\setminus\nu(D)$. This partially answers a question of Casson and Gordon. In particular, we show the fibration always extends when $D$ has exactly two local minima. More generally, we construct movies of singular fibrations on $4$-manifolds and describe a sufficient property of a movie to imply the underlying $4$-manifold is fibered over $S^1$.

\end{abstract}
\setcounter{tocdepth}{1}
\setcounter{equation}{0}

\section{Introduction\label{ch:intro}}

The term, ``3.5-dimensional topology," is often used to refer to the study of knots in $S^3$ which bound smooth disks embedded in $B^4$ and related topics. We recall some basic definitions.

\begin{definition}

A knot $K$ in $S^3$ bounding a disk $D$ smoothly embedded in $B^4$ is said to be {\emph{slice}}.

If for some $D$, inclusion causes $\pi_1(S^3\setminus K)$ to surject onto $\pi_1(B^4\setminus D)$, then $K$ is said to be {\emph{homotopy-ribbon}}.

If for some $D$, the radial height function of $B^4$ is Morse on $D$ with no local maxima, then $K$ is said to be {\emph{ribbon}}.
\end{definition}

From definitions, we have the following order of inclusion. \[\{\text{Ribbon knots}\}\subset\{\text{Homotopy-ribbon knots}\}\subset\{\text{Slice knots}\}.\]

Whether any of these inclusions are proper is a long-standing open question.

\begin{sliceribbon}
Every slice knot in $S^3$ is ribbon.
\end{sliceribbon}

One challenge in studying this problem is that it is difficult to show that a knot $K$ is slice without implying that $K$ is ribbon, and conversely that it is difficult to obstruct $K$ from being ribbon without obstructing $K$ from being slice. Casson and Gordon~\cite{gordon} found an obstruction to a {\emph{fibered}} knot $K$ being homotopy-ribbon.

\begin{definition}
Let $K$ be a knot in $S^3$. Then $K$ is {\emph{fibered}} if $S^3\setminus\nu(K)$ is a bundle $\mathring\Sigma_g\times_{\phi} S^1=\mathring{\Sigma}_g\times[0,1]/[(x,1)\sim(\phi(x),0)]$, where $\mathring\Sigma_g$ is a Seifert surface for $K$ and $\phi:\mathring{\Sigma}_g\to\mathring{\Sigma}_g$ is a surface automorphism fixing $\boundary\mathring\Sigma_g$ pointwise.
\end{definition}

The map $\phi$ is said to be the {\emph{monodromy}} of $K$. This surface automorphism is well-defined up to conjugacy in the mapping class group of $\mathring\Sigma_g$. Note $\phi$ is an automorphism of a surface with boundary, but can be extended to a map $\hat{\phi}$ of the closed surface $\Sigma_g=\mathring{\Sigma}_g\cup D^2$. We call $\hat{\phi}$ the {\emph{closed monodromy}} of $K$. We now state the general result of Casson and Gordon.

\begin{cassongordon}
A fibered knot $K$ in a homology $3$-sphere $M$ is homotopy-ribbon if and only if its closed monodromy extends over a handlebody. In particular, if $M=S^3$ and $K$ is fibered and homotopy-ribbon, then there is a homotopy $4$-ball $V$ and slice disk $D\subset V$ so that $(S^3,K)=\boundary(V,D)$ and $V\setminus\nu(D)$ is fibered by handlebodies.
\end{cassongordon}

(Note we did not define what it means for a knot in a homology $3$-sphere to be homotopy-ribbon.)

From Casson and Gordon's proof of the above theorem, it is not obvious if the fibration on $S^3\setminus \nu(K)$ should extend over $B^4\setminus \nu(D)$ for a fixed disk $D$, or even whether $V$ is diffeomorphic to $B^4$. 

\begin{question}\label{otherquestion}
Is $V\cong B^4$?\footnote{In a recent preprint, Meier and Zupan~\cite{mzgensquare} show that the answer to Question~\ref{otherquestion} is, ``yes," in some very special circumstances. However, in general Question~\ref{otherquestion} remains wide open.}
\end{question}

In this paper, we focus on the other natural question stemming from Theorem 5.1 of~\cite{gordon}.

\begin{question}\label{mainquestion}
Fix a homotopy-ribbon disk $D$ in $B^4$ with boundary a fibered knot. Is $B^4\setminus\nu(D)$ fibered over $S^1$? (i.e. ``is $D$ fibered?")
\end{question}

If the answer to Question~\ref{mainquestion} is, ``yes'', then the $3$-dimensional fibers are immediately known to be handlebodies by the following result of Larson and Meier.

\begin{jefftheorem}
Let $D$ be a fibered slice disk in $B^4$ with fiber $H$. Then $D$ is homotopy-ribbon if and only if $H$ is a handlebody.
\end{jefftheorem}
\begin{remark}\label{strongremark}
Larson and Meier use a different (and without the slice-ribbon conjecture, possibly more restrictive) definition of homotopy-ribbon, which we might call {\emph{strongly homotopy-ribbon}}. A disk $E$ is strongly homotopy-ribbon if $B^4\setminus\nu(E)$ admits a handle decomposition without $3$- or $4$-handles. Their proof works for our stated definition of homotopy-ribbon as well (see section~\ref{sec:questions}), 
yielding (disk fibered by handlebodies $\implies$ strongly homotopy-ribbon) and (disk fibered and homotopy-ribbon $\implies$ fibers are handlebodies). We conclude that for the class of fibered disks in $B^4$, homotopy-ribbon and strongly homotopy-ribbon are equivalent conditions.
\end{remark}

\begin{remark}
We remark for the general reader that although the sets of fibered knots and ribbon knots in $S^3$ are very special (and their intersection even more so), the set of prime fibered ribbon knots is infinite. For example, the three-stranded pretzel knots $P(\pm2,n,-n)$ for any odd $|n|>1$ are prime, fibered~\cite{detecting}, and ribbon. In fact, these knots even admit a ribbon disk with only two local minima (the index-$1$ critical point ``cuts'' the $\pm2$-twisted strand).

There are a total of $74$ prime fibered ribbon knots of fewer than thirteen crossings~\cite{livingston}.
\end{remark}

In this paper, we prove the following relevant theorem.

\begin{theorem}\label{maintheorem}
Let $K\subset S^3$ be a fibered knot bounding a ribbon disk $D\subset B^4$. View the index-$1$ critical points of $D$ as fission bands attached to $K$. If the bands can be isotoped to be transverse to the fibration on $S^3\setminus\nu(K)$, then the fibration extends to a fibration by handlebodies on $B^4\setminus\nu(D)$.
\end{theorem}

To parse Theorem~\ref{maintheorem}, we must understand how bands attached to a knot can describe a ribbon disk. Let $K$ be a fibered knot in $S^3$. Fix a ribbon disk $D\subset B^4$ with $\boundary D=K$. View $B^4=S^3\times [0,3]/(S^3\times 0\sim\pt)$, where $S^3=\boundary B^4=S^3\times 3$. We call the projection function $h:B^4\to[0,3]$ the height function on $B^4$. A point $(x,t)\in S^3\times[0,3]/\sim$ is said to be at height $t$.

 Isotope $D$ so that:
\begin{itemize}
\item $D\subset S^3\times[1,3]$.
\item For $t_1>t_2>\ldots>t_n\in(2,3)$, $D\cap (S^3\times t_i)=$ a link with a band attached (referred to as a ribbon band). We will always view these bands as ``fission'' bands, meaning resolving the band increases the number of components of the underlying link. In this setting, this means we view the band at height $t_i$ as attached to a link equivalent to $D\cap (S^3\times(t_i+\epsilon))$.
\item For $s_1>\ldots>s_{n+1}\in(1,2)$, $D\cap (S^3\times s_i)$ is an $(n+1-i)$-component unlink and a disjoint disk (referred to as a minimum disk).
\item For all other $u\in(s_{n+1},3]$, $D\cap (S^3\times u)$ is a nonsingular link. For all $u\in[0,s_{n+1})$, $D\cap (S^3\times u)=\emptyset$.
\end{itemize}

In plain English, we isotope the disk $D$ so that, starting from the boundary $S^3=S^3\times 3$ and moving downward, we first see all the saddle points and then the minima of $D$. 
We project the ribbon bands into $h^{-1}=\boundary B^4=S^3$ to obtain a ribbon diagram $\mathcal{D}$ for $D$. (Assume generically that all ribbon bands are disjoint in $\mathcal{D}$.) Call the bands $b_1,\ldots, b_n$, where $b_i$ corresponds to the saddle at $t_i$.

See Figure~\ref{fig:diskexample1} for an example of a ribbon disk and Figure~\ref{fig:diskexample2} for a schematic of $D$ in $B^4$. Whenever we refer to a disk being determined by bands (or index-1 critical points of a disk corresponding to bands), we are referring to this description of a ribbon disk.

\begin{figure}\begin{centering}
\includegraphics[width=.3\textwidth]{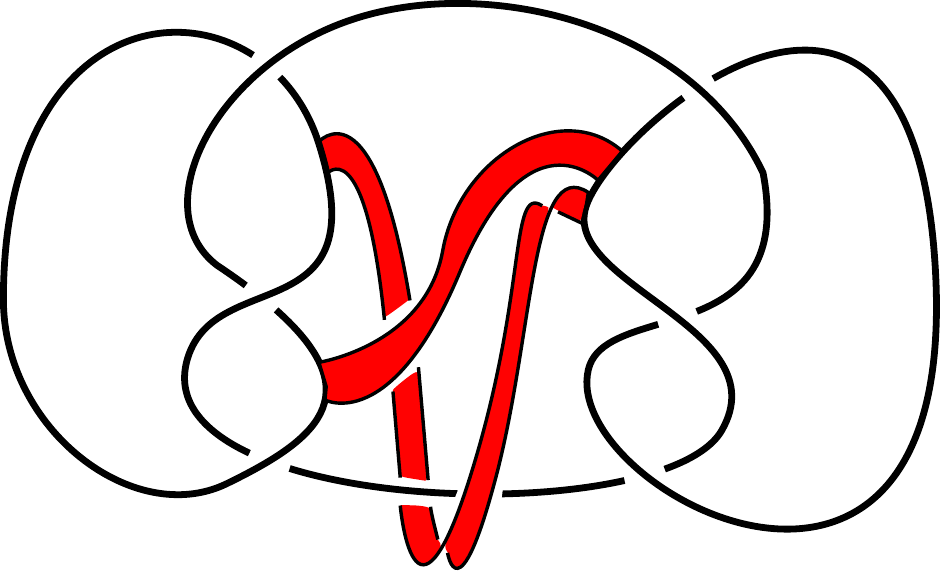}\\\hspace{.5in}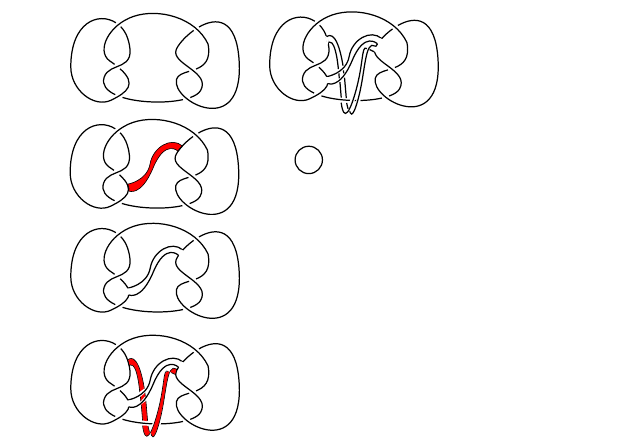
\caption[A diagram of a ribbon disk $D$ with boundary the square knot]{Top: a diagram of a ribbon disk $D$ with boundary the square knot. Below: Some cross-sections of $D\subset B^4$. Recall $h^{-1}(3)$ (that is, $t=3$) is the boundary $S^3$ of $B^4$.}
\label{fig:diskexample1}
\end{centering}\end{figure}

\begin{figure}\begin{centering}
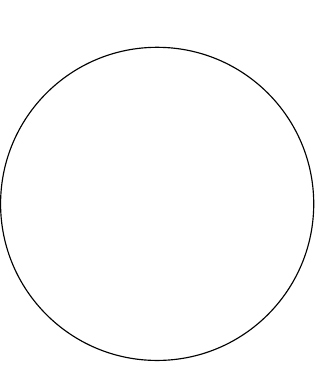
\caption[A schematic of a ribbon disk in $B^4$.]{A schematic of a ribbon disk in $B^4$. In this example, the disk has $3$ minima and $2$ index-$1$ critical points (bands), as in Figure~\ref{fig:diskexample1}.}
\label{fig:diskexample2}
\end{centering}\end{figure}

Now we discuss the transversality condition Theorem~\ref{maintheorem} places on fission bands.

\begin{definition}
Let $b$ be a band attached to a fibered knot $K$, where $S^3\setminus\nu(K)=\mathring{\Sigma}\times I/\sim$. Say $b=\gamma\times[0,\epsilon]$ for some arc $\gamma\in S^3$ with $(\boundary \gamma )\times[0,\epsilon]\subset K$. We say $b$ is transverse to the fibration on $S^3\setminus\nu(K)$ if $\mathring\gamma$ is transverse to every $\mathring{\Sigma}\times t$.  Given a transverse band $b$, we may (and will) always assume $b\cap (\Sigma\times t)=\{$finite set of points$\}\times[0,\epsilon]$.

In Theorem~\ref{maintheorem}, the hypothesis is that {\emph{all}} bands in the description of $D$ are disjoint and  {\emph{simultaneously}} transverse to the fibration on $S^3\setminus\nu(K)$.
\end{definition}

The following theorem follows almost immediately from the proof of Theorem~\ref{maintheorem}.
\begin{theorem}\label{secondtheorem}
Let $K$ be a fibered knot in $S^3$. Let $J\subset S^3$ be a knot so that there is a {\emph{ribbon concordance}} from $K$ to $J$ (i.e.\ there is an annulus $A$ properly embedded in $S^3\times I$ with boundary $(K\times 1)\sqcup (\overline{J}\times 0)$ so that $\proj_I|_{A}$ is Morse with no local maxima). 
Say the index-$1$ critical points of $A$ correspond to fission bands attached to $K$. If the bands can be isotoped to be transverse to the fibration on $S^3\setminus\nu(K)$, then $(S^3\times I)\setminus\nu(A)$ is fibered by compression bodies.
\end{theorem}

We remark that in Theorem~\ref{secondtheorem}, the fact that $J$ is fibered follows immediately from~\cite[Prop. 5]{miyazaki}. Theorem~\ref{secondtheorem} is truly a 4-dimensional theorem about the ribbon concordance between $K$ and $J$.

Both Theorems~\ref{maintheorem} and~\ref{secondtheorem} are applications of our main theorem:

\begin{theorem}\label{fibrationthm}
Let $\F_t$ be a valid movie of singular fibrations on $Z^4$ so that $\F_0$ and $\F_1$ are nonsingular. 
Then $Z^4$ is fibered over $S^1$ by $\{H_\theta\}$, where $H_\theta=\cup_t \F_t^{-1}(\theta)$.
\end{theorem}
We will define a valid movie of singular fibrations in section~\ref{sec:fibrations}.

\begin{remark}
The proof of Theorem~\ref{maintheorem} is constructive. From a minimum-genus surface $S\subset S^3$ for $K$ and the fission bands defining $D$, it is possible to recover the attaching circles on $S$ which define the handlebody to which the fibration extends. See section~\ref{sec:conclusion}.
\end{remark}

In Theorem~\ref{maintheorem}, note that the fission bands are transverse to the open book on $S^3$ induced by the fibration of $S^3\setminus\nu(K)$ in their {\emph{interior}}. Near the binding of the open book on $S^3$, the bands approach $K$ asymptotically tangent to some leaf (see Figure~\ref{fig:bandinsurface}, bottom right).

\begin{remark}\label{remark:infiber}
If a fission band attached to a fibered knot $K$ is contained in a fiber of $S^3\setminus\nu(K)$, we may perturb the band to be transverse to the fibers as in Figure~\ref{fig:bandinsurface}.
\end{remark}

\begin{figure}\begin{centering}
{\center{
\begin{tabular}{cc}
\begingroup%
  \makeatletter%
  \providecommand\color[2][]{%
    \errmessage{(Inkscape) Color is used for the text in Inkscape, but the package 'color.sty' is not loaded}%
    \renewcommand\color[2][]{}%
  }%
  \providecommand\transparent[1]{%
    \errmessage{(Inkscape) Transparency is used (non-zero) for the text in Inkscape, but the package 'transparent.sty' is not loaded}%
    \renewcommand\transparent[1]{}%
  }%
  \providecommand\rotatebox[2]{#2}%
  \newcommand*\fsize{\dimexpr\f@size pt\relax}%
  \newcommand*\lineheight[1]{\fontsize{\fsize}{#1\fsize}\selectfont}%
  \ifx\svgwidth\undefined%
    \setlength{\unitlength}{160.08491216bp}%
    \ifx\svgscale\undefined%
      \relax%
    \else%
      \setlength{\unitlength}{\unitlength * \real{\svgscale}}%
    \fi%
  \else%
    \setlength{\unitlength}{\svgwidth}%
  \fi%
  \global\let\svgwidth\undefined%
  \global\let\svgscale\undefined%
  \makeatother%
  \begin{picture}(1,0.39127791)%
    \lineheight{1}%
    \setlength\tabcolsep{0pt}%
    \put(0,0){\includegraphics[width=\unitlength,page=1]{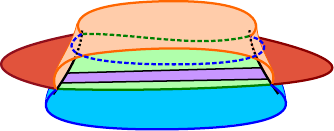}}%
    \put(1.52746567,0.25324179){\color[rgb]{0.50196078,0,0.50196078}\makebox(0,0)[lt]{\lineheight{0}\smash{\begin{tabular}[t]{l}$b$\end{tabular}}}}%
    \put(0,0){\includegraphics[width=\unitlength,page=2]{Fig1a.pdf}}%
    \put(0,0){\includegraphics[width=\unitlength,page=3]{Fig1a.pdf}}%
  \end{picture}%
\endgroup%
&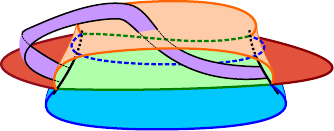\\
\begingroup%
  \makeatletter%
  \providecommand\color[2][]{%
    \errmessage{(Inkscape) Color is used for the text in Inkscape, but the package 'color.sty' is not loaded}%
    \renewcommand\color[2][]{}%
  }%
  \providecommand\transparent[1]{%
    \errmessage{(Inkscape) Transparency is used (non-zero) for the text in Inkscape, but the package 'transparent.sty' is not loaded}%
    \renewcommand\transparent[1]{}%
  }%
  \providecommand\rotatebox[2]{#2}%
  \newcommand*\fsize{\dimexpr\f@size pt\relax}%
  \newcommand*\lineheight[1]{\fontsize{\fsize}{#1\fsize}\selectfont}%
  \ifx\svgwidth\undefined%
    \setlength{\unitlength}{77.02865541bp}%
    \ifx\svgscale\undefined%
      \relax%
    \else%
      \setlength{\unitlength}{\unitlength * \real{\svgscale}}%
    \fi%
  \else%
    \setlength{\unitlength}{\svgwidth}%
  \fi%
  \global\let\svgwidth\undefined%
  \global\let\svgscale\undefined%
  \makeatother%
  \begin{picture}(1,0.66842116)%
    \lineheight{1}%
    \setlength\tabcolsep{0pt}%
    \put(0,0){\includegraphics[width=\unitlength,page=1]{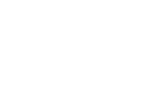}}%
    \put(0.7184217,0.41315788){\color[rgb]{0.50196078,0,0.50196078}\makebox(0,0)[lt]{\lineheight{0}\smash{\begin{tabular}[t]{l}$b$\end{tabular}}}}%
    \put(0,0){\includegraphics[width=\unitlength,page=2]{Fig1c.pdf}}%
    \put(0,0){\includegraphics[width=\unitlength,page=3]{Fig1c.pdf}}%
  \end{picture}%
\endgroup%
&
\begingroup%
  \makeatletter%
  \providecommand\color[2][]{%
    \errmessage{(Inkscape) Color is used for the text in Inkscape, but the package 'color.sty' is not loaded}%
    \renewcommand\color[2][]{}%
  }%
  \providecommand\transparent[1]{%
    \errmessage{(Inkscape) Transparency is used (non-zero) for the text in Inkscape, but the package 'transparent.sty' is not loaded}%
    \renewcommand\transparent[1]{}%
  }%
  \providecommand\rotatebox[2]{#2}%
  \newcommand*\fsize{\dimexpr\f@size pt\relax}%
  \newcommand*\lineheight[1]{\fontsize{\fsize}{#1\fsize}\selectfont}%
  \ifx\svgwidth\undefined%
    \setlength{\unitlength}{88.98834409bp}%
    \ifx\svgscale\undefined%
      \relax%
    \else%
      \setlength{\unitlength}{\unitlength * \real{\svgscale}}%
    \fi%
  \else%
    \setlength{\unitlength}{\svgwidth}%
  \fi%
  \global\let\svgwidth\undefined%
  \global\let\svgscale\undefined%
  \makeatother%
  \begin{picture}(1,0.57858795)%
    \lineheight{1}%
    \setlength\tabcolsep{0pt}%
    \put(0,0){\includegraphics[width=\unitlength,page=1]{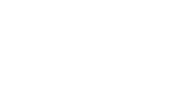}}%
    \put(0,0){\includegraphics[width=\unitlength,page=2]{Fig1d.pdf}}%
    \put(0.62186837,0.35763109){\color[rgb]{0.50196078,0,0.50196078}\makebox(0,0)[lt]{\lineheight{0}\smash{\begin{tabular}[t]{l}$b$\end{tabular}}}}%
    \put(0,0){\includegraphics[width=\unitlength,page=3]{Fig1d.pdf}}%
  \end{picture}%
\endgroup%
\end{tabular}}}
\caption[We perturb a band in one leaf to be transverse to a fibration.]{Top left: a band $b$ lies in a single Seifert surface for a fibered knot $K$. Top right: Perturb the band near one component of $b\cap K$ so that $b$ is transverse to each leaf. (We can think of this perturbation as going halfway around $S^1=\pi(S^3\setminus K)$, where $\pi$ is the projection map from $S^3=\mathring{\Sigma}\times_{\phi} S^1$ to $S^1$.) Bottom left: side view of $b$ in one leaf, near one end of $b\cap K$. Bottom right: side view after perturbing $b$.}
\label{fig:bandinsurface}
\end{centering}
\end{figure}

Now we discuss some corollaries of Theorem~\ref{maintheorem}.

\begin{corollary}\label{diskcor}
Let $D$ be a ribbon disk in $B^4$ with exactly two minima. If $\boundary D$ is a fibered knot, then the fibration extends to a fibration of $B^4\setminus\nu(D)$ by handlebodies.
\end{corollary}
\begin{proof}
For a link $J$ in $S^3$, let $\chi(J)=\max\{\chi(S)\mid S\subset S^3$ is an oriented surface bounded by $J$ with no closed components$\}$. Suppose $J'$ is obtained from $J$ by orientation-preserving band surgery along a band $b$. Then if $\chi(J')>\chi(J)$, the band $b$ can be isotoped to lie in a Seifert surface $S$ for $J$ with $\chi(S)=\chi(J)$~\cite[Prop 3.1]{thompson}.

Say $K=\boundary D$. Let $b$ be a fission band for $K$ corresponding to the single saddle of $D$. Then surgering $K$ along $b$ yields the $2$-component unlink $U\sqcup U$. Note $\chi(K)\le 1<2=\chi(U\sqcup U)$, so $b$ may be isotoped to lie in a Seifert surface $S$ for $K$ with $\chi(S)=\chi(K)$. Then $S$ is a fiber surface for $K$ (see e.g.~\cite[Subsection 5B]{burde}), so the claim follows from Theorem~\ref{maintheorem} and Remark~\ref{remark:infiber}.

\end{proof}

By gluing together two ribbon disks in separate $4$-balls, we may obtain a $2$-knot.
\begin{definition}
An {\emph{$n$-knot}} is a copy of $S^n$ smoothly embedded into $S^{n+2}$.
\end{definition}
A $1$-knot is a classical knot in $S^3$, while a $2$-knot is a knotted $2$-sphere in $S^4$.

\begin{definition}
Let $D_1,D_2$ be ribbon disks  with $\boundary D_1=\boundary D_2$.
We say $\boundary D_i$ is an {\emph{equator}} of the $2$-knot $D_1\cup\overline{D_2}$ in $S^4=B^4\cup\overline{B^4}$.
\end{definition}

Note that an equator of a 2-knot $G$ is not uniquely defined; for example, $K\#-K$ is an equator of the unknotted $2$-sphere for any $1$-knot $K$, as $K\#-K$ is naturally an equator of the $1$-twist spin of $K$, which is unknotted~\cite{zeeman}. 

If $G=D_1\cup\overline{D_2}$ where $D_1,D_2$ have $m_1,m_2$ minima respectively, then in $S^4$ the $2$-knot $G$ can be taken to have $m_1$ minima, $m_2$ maxima, and $m_1+m_2-2$ saddle points with respect to the standard height function. (Of course, ambient isotopy of $S^4$ may change the number of critical points induced by the standard height function on $G$.)

\begin{corollary}\label{2knotcor}
Let $G$ be a $2$-knot embedded in $S^4$ with exactly two minima, two maxima, and two saddle points with respect to the standard height function. If the equator of $G$ (with respect to this decomposition) is a fibered $1$-knot of genus-$g$, then $G$ is fibered and the closure of the $3$-dimensional fiber for $G$ admits a genus-$g$ Heegaard splitting.
\end{corollary}

\begin{proof}
We have $G=D_1\cup\overline{D_2}$, where each $D_i$ has two minima and $K:=\boundary D_i$ is fibered. By Corollary~\ref{diskcor}, the fibration on $S^3\setminus\nu(K)$ extends to fibrations on $B^4\setminus\nu(D_1)$ and $B^4\setminus\nu(D_2)$ by genus-$g$ handlebodies. The claim follows.
\end{proof}

Corollary~\ref{diskcor} 
can be viewed as an extension of the following fact due to Scharlemann~\cite{scharlemann}: If $D\subset B^4$ is a ribbon disk with two minima so that $\boundary D$ is the unknot, then $D$ is boundary parallel 
(and hence $B^4\setminus\nu(D)\cong S^1\times B^3$ is fibered by $3$-balls). Note that to prove Corollary~\ref{diskcor}, we used the fact that the saddle band lies in a minimum-genus surface of the equator. This fact also proves Scharlemann's theorem, so we have extended this result but do not claim a simpler proof.

Finally, we remark that there is no analogue of Theorem~\ref{maintheorem} for links. In fact, for $r>1$, a fibered $r$-component link cannot be slice (in the sense of bounding $r$ disjoint slice disks in $B^4$). This follows from the fact that for $r>1$, a slice $r$-component link $L$ has Alexander polynomial\footnote{We are referring to Fox's definition of the Alexander polynomial of a link, as in~\cite{foxtrip}.} $\Delta(t_1,\ldots, t_r)=0$~\cite{foxtrip}. Meanwhile, if $L$ is fibered, then $\Delta(t,\ldots, t)$ must be monic, as $t^k (t-1)\Delta(t,\ldots, t)$ is the characteristic polynomial of the monodromy of $L$ for some $k\in\Z$~\cite[Lemma 10.1]{milnor}.

More fundamentally, this result cannot be extended to links because for $r>1$, $\chi(B^4\setminus\nu(\sqcup_r D^2))<0$, so the complement of multiple ribbon disks in $B^4$ cannot be fibered.\footnote{More generally, a multi-component link of $S^{k}$s embedded in $S^{k+2}$ is {\emph{never}} fibered for $k>1$~\cite{notfibered}. For $k$ even, this follows easily from the Euler characteristic argument.} Thus, Theorem~\ref{maintheorem} passes a basic sanity check for links.

\subsection*{Outline}
\begin{itemize}
\item
Section~\ref{sec:fibrations}: We give basic definitions and prove Theorem~\ref{fibrationthm}.
\item Section~\ref{sec:blocks}: We construct a library of basic movies of singular fibrations used to construct fibrations on larger $4$-manifolds.
\item Section~\ref{construction}: We prove Theorems~\ref{maintheorem} and~\ref{secondtheorem}. We mainly focus on Theorem~\ref{maintheorem}, and then quickly prove Theorem~\ref{secondtheorem} using the same argument.
\item Section~\ref{sec:conclusion}: We discuss how to find the $2$-handle attaching circles on a fiber in $S^3\setminus\nu(K)$ which define the handlebody fiber in $B^4\setminus\nu(D)$ built during the construction of Theorem~\ref{maintheorem}. 
\item Section~\ref{sec:examples}: We provide examples of fibered ribbon disks. In particular, we produce a fibered ribbon disk for each prime fibered slice knot of fewer than 13 crossings.
\item Section~\ref{sec:questions}: We discuss relevant open questions.
\end{itemize}

\subsection*{Acknowledgements}
 The author thanks her graduate advisor, David Gabai, for many long, helpful conversations. Thanks to David Gay for interesting points about singularities that helped to explain movies of singular fibrations. Thanks also to Clayton McDonald, Jeff Meier, Ian Zemke and Alex Zupan for insight on ribbon disks and Cole Hugelmeyer for useful expositional comments. Finally, thanks to an anonymous referee (who clearly read this paper very carefully) for providing many helpful comments.

During the time of this project, the author was a fellow in the National Science Foundation Graduate Research Fellowship program, under Grant No. DGE-1656466, at Princeton University. This paper was the basis of her doctoral dissertation submitted to Princeton in April, 2020.

\section{Movies of singular fibrations on $4$-manifolds with boundary\label{sec:fibrations}}

\subsection{Basic definitions}

In this subsection, we essentially give alternate definitions of circular Morse functions and their critical points specifically in dimension three. These descriptions will be useful in explicit constructions.

\begin{definition}
A {\emph{singular fibration}} $\F$ on a compact $3$-manifold $M^3$ is a smooth map $\F:M^3\to S^1$ with the following properties:
\begin{itemize}
\item Each $\F^{-1}(\theta)$ is a compact, properly embedded surface away from a finite number of {\emph{cone}} or {\emph{dot}} singularities in its interior and {\emph{half-cone}}, {\emph{half-dot}}, and {\emph{bowl}} singularities in its boundary (see below).
\item For all but finitely many $\theta$, $\F^{-1}(\theta)$ is a nonsingular surface.
\item There are finitely many singularities in the foliation $\F$ induces on $\boundary M^3$. If $\F^{-1}(\theta)\cap\boundary M$ contains such a singular point, then near that point $\F^{-1}(\theta)$ is a half-cone, a half-dot, or a bowl. See Figure~\ref{fig:singularities}.
 \end{itemize}

See Figure~\ref{fig:singularities} for an illustration of cone and dot singularities, as well as singularities in the induced foliation on  $\boundary M^3$. We believe this figure and the names of the singularities to be sufficient to understand what they are, but we include the following more precise descriptions for completeness. (These descriptions can be safely ignored if the reader is comfortable with Figure~\ref{fig:singularities}).

\begin{figure}\begin{centering}
\includegraphics[width=.8\textwidth]{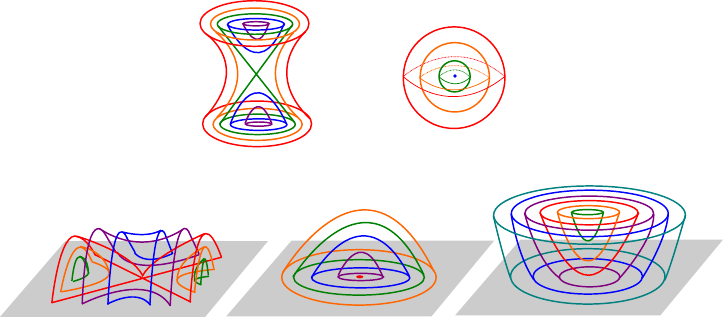}
\caption[Neighborhoods of singularities of singular fibrations.]{Top: neighborhoods of a cone and dot singularity in $\mathring M^3$. Bottom: Leaves of $\F$ near a half-cone, half--dot, and bowl singular point in the foliation induced on $\boundary M^3$.}
\label{fig:singularities}
\end{centering}\end{figure}

For $p\in \mathring{M}^3$, We say that $p\in \F^{-1}(\theta)$ is a {\emph{cone}} singularity if for some small $3$-ball neighborhood of $p$ in $M^3$, with coordinates $x,y,z$, we have $\F^{-1}(\theta)=\{x^2+y^2=z^2\}$, and for $\theta'\neq\theta$, $\F^{-1}(\theta')=\{x^2+y^2=z^2+c_{\theta'}\}$ for some constant $c_{\theta'}\neq 0$. In words, near $p$ the set $\F^{-1}(\theta)$ looks like a double cone, while each nearby $\F^{-1}(\theta')$ looks like a hyperboloid of one or two sheets.

For $p\in\mathring{M}^3$, we say $p\in\F^{-1}(\theta)$ is a {\emph{dot}} singularity if for some small $3$-ball neighborhood of $p$ in $M^3$, with coordinates $x,y,z$, we have $\F^{-1}(\theta)=\{(0,0,0)\}$ and for $\theta'\neq\theta$, $\F^{-1}(\theta')=\{x^2+y^2+z^2=c_{\theta'}\}$ for some constant $c_{\theta'}\neq 0$. In words, near $p$ the set $\F^{-1}(\theta)$ looks like an isolated point, while nearby $\F^{-1}(\theta')$ are nested spherical shells (or the empty set).
 
 For $p\in\boundary M^3$, we say $p\in\F^{-1}(\theta)$ is a {\emph{half-cone}} singularity if for some small neighborhood of $p$ in $M^3$, with coordinates $x,y,z$, we have $\F^{-1}(\theta)=\{x^2+z^2=y^2, z\ge0\}$, and for $\theta'\neq\theta$, $\F^{-1}(\theta')=\{x^2+z^2=y^2+c_{\theta'},z\ge 0\}$ for some constant $c_{\theta'}\neq 0$. In words, near $p$ the set $\F^{-1}(\theta)$ looks like half of a cone (half of each side), while each nearby $\F^{-1}(\theta')$ looks like half of a hyperboloid of one or two sheets.

For $p\in\boundary M^3$, we say $p\in\F^{-1}(\theta)$ is a {\emph{half-dot}} singularity if for some small neighborhood of $p$ in $M^3$, with coordinates $x,y,z$, we have $\F^{-1}(\theta)=\{(0,0,0)\}$ and for $\theta'\neq\theta$, $\F^{-1}(\theta')=\{x^2+y^2+z^2=c_{\theta'}, z\ge 0\}$ for some constant $c_{\theta'}\neq 0$. In words, near $p$ the set $\F^{-1}(\theta)$ looks like an isolated point, while nearby $\F^{-1}(\theta')$ are nested half-spherical shells (or the empty set).

 For $p\in\boundary M^3$, we say $p\in\F^{-1}(\theta)$ is a {\emph{bowl}} singularity if for some small neighborhood of $p$ in $M^3$, with coordinates $x,y,z$, we have $\F^{-1}(\theta)=\{x^2+y^2=z\}$ and for $\theta'\neq\theta$, $\F^{-1}(\theta')=\{x^2+y^2=z+c_{\theta'}, z\ge 0\}$ for some constant $c_{\theta'}\neq 0$. In words, near $p$ the set $\F^{-1}(\theta)$ looks like a bowl (i.e. a circular paraboloid), while nearby $\F^{-1}(\theta')$ are nested bowls or cylinders.

We may think of $\F$ as being a circle-valued Morse function on $M^3$, so that the cone and dot singularities correspond to critical points in the Morse function $\F$. A cone corresponds to a critical point of index $1$ or $2$, while a dot to a critical point of index $0$ or $3$ (depending on the orientation of each $\F^{-1}(\theta)$ induced by the positive orientation of $S^1$).
 
 We will refer to each $\F^{-1}(\theta)$ as a {\emph{leaf}} of $\F$.

 \end{definition}

 We may extend the definition of singular fibrations to some singular $3$-manifolds. Specifically, we consider those singular 3-manifolds which arise as level sets of Morse functions on closed 4-manifolds.
 
 \begin{definition}\label{singdef}
 Let $M^3$ be a closed, compact, singular $3$-manifold with finitely many singularities $S=\{p_1,\ldots, p_n,q_1,\ldots, q_m\mid p_i\in\boundary M^3, q_j\in\mathring{M^3}\}$. For $p_i\in S\cap\boundary M^3$, we require $p_i$ to have a neighborhood homeomorphic to one of: $\{x^2+y^2\ge z^2\}$, $\{x^2+y^2\le z^2\}$, $\{x^2+y^2+z^2\ge0\mid (0,0,0)$ is artificially declared to be a boundary point$\}$ or $\{(0,0,0)\}$. 
 
 For $q_i\in S\cap\mathring{M^3}$, we require $q_i$ to have a neighborhood of the form $P\cup_{\nu(q_i)\cap\boundary(P)}\overline{P}$, where $P$ is a singular manifold with boundary and $q_i$ is a singularity as above. (That is, we first understand a neighborhood $P$ of a singularity in $\boundary M^3$. Then an interior singularity has a neighborhood which is a double of $P$.)
 
 This means $\boundary M^3$ is the disjoint union of surfaces with a finite number of cone singularities and isolated points $\{p_1,\ldots, p_n\}$.
 
Let $\G$ be a map $\G:(M^3\setminus S)\to S^1$ which is a singular fibration on $M^3$ away from singularities of $M^3$. 

We extend $\G$ to $\F:M^3\to S^1$ by saying $\F(x)=\theta$ if there is a smooth path $\gamma:[0,1]\to M^3$ with $\gamma(1)=x$ and $\lim_{t\to 1}\G(\gamma(t))=\theta$. If $x$ is isolated in $M^3$, we choose $\F(x)$ arbitarily. We require that $\G$ be defined so that near a singular point $p_i$ in $\boundary M^3$, $\F$ is as in one of the diagrams in Figure~\ref{fig:singularmfd}. Near a singular point $q_i\in\mathring{M^3}$, $\F$ must be as in two mirror copies of one diagram of Figure~\ref{fig:singularmfd} glued along the boundary near $q_i$. In particular, we require $\F$ be a well-defined function. 

We call $\F$ a singular fibration of $M^3$.
\end{definition}

\begin{figure}\begin{centering}
\includegraphics[width=.6\textwidth]{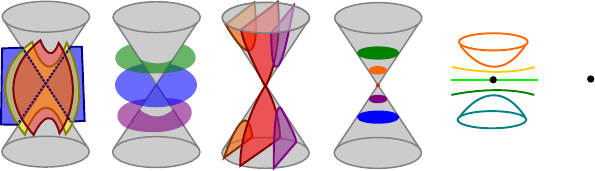}
\caption[A singular fibration $\F$ on a singular $3$-manifold $M^3$.]{A singular fibration $\F$ on a singular $3$-manifold $M^3$, as in Def.~\ref{singdef}. We require $\F$ agree with one of these six configurations in a neighborhood of a singular point in $\boundary M^3$. Near a singular point in $\mathring{M^3}$, $\F$ should agree with two copies of one of these diagrams, glued together along boundary near the singularity. {\bf{We first describe each figure as if the singularity is in $\boundary M^3$.}}
First four: the gray cone is contained in $\boundary M^3$. At the cone point, a leaf of $\F$ is either a wedge of two disks along a boundary point, a disk centered at the cone, a wedge of two disks along a boundary point, or a point (respectively). A nearby leaf is a disk, annulus, two disjoint disks, or disk (respectively). 
Last two: the black dot is an isolated point of $\boundary M^3$. We draw leaves of $\F$ near this point, in the case that the point is not isolated in $M^3$, and then finally the case that the point is isolated in $M^3$.
{\bf{We now describe each figure as if the singularity is in $\mathring{M^3}$.}}
First four: at the singularity, a leaf of $\F$ is either a cone, a cone, a cone, or a point (respectively). A nearby leaf is a cylinder, cylinder, two disjoint disks, or sphere (respectively).
Last two: the black dot is an isolated point of $\mathring{M^3}$. The leaf of $\F$ containing the singularity is either a cone or a point, respectively.}
\label{fig:singularmfd}
\end{centering}\end{figure}

\begin{definition}
A {\emph{movie of singular fibrations}} on $X^4$ with Morse function $h:X^4\to I$ is a family of smooth maps $\F_t:h^{-1}(t)\to S^1$  so that each $\F_t$ is a singular fibration on $h^{-1}(t)$ and the $\F_t$ vary smoothly with $t$ (i.e.\ $x\mapsto \F_{h(x)}(x)$ is a smooth map from $X^4$ to $S^1$).
\end{definition}

We give several basic building block movies of singular fibrations in section~\ref{sec:blocks}. First, we discuss how movies of singular fibrations may describe smooth fibrations of a $4$-manifold.

\subsection{Singularity charts: how to fiber a $4$-manifold}

Suppose we have a movie $\F_t$ of singular fibrations on $Z^4$ and 
we hope to find a fibration $\{H_\theta\}_{\theta\in S^1}$ of $Z^4$.

First, we consider two natural (and failing) attempts to create a fibration on $Z^4$.

\begin{itemize}
\item The ``vertical fiber'' strategy: Let $H_\theta=\cup_t \F_t^{-1}(\theta)$. This fails because $\F_t^{-1}(\theta)$ may be singular for an interval of $t$ values. Then the given union will not generally be a 3-manifold.
\item The ``tilted fiber'' strategy: Let $H_\theta=\cup_t \F_t^{-1}(\theta+c t)$ for a constant $c$. For some choice of $c$, singular cross-sections of $H_\theta$ are isolated. 
However this strategy fails (in part) because we may change the direction in which the $3$-dimensional fibers flow through a singular leaf in $\F_t$, leading to some $3$-dimensional fiber being singular.
\end{itemize}

It is very important that the reader unstand the failing of the tilted fiber strategy. Attempt to take $H_\theta:\cup_t \F_t^{-1}(\theta+\epsilon t)$. 
Suppose $p$ is a cone point of $\F_t$ for $t\in[s',s'']$. For some $\theta',\theta''$, suppose that $H_{\theta'}$ and $H_{\theta''}$ include $p$ at heights $s'$ and $s''$ respectively. Suppose moreover that near $p$, $H_{\theta'}\cap h^{-1}(s'-\epsilon)$ and $H_{\theta''}\cap h^{-1}(s''+\epsilon)$ are one-sheeted hyperboloids.

By the intermediate value theorem, there must be some $H_\theta$ so that $H_\theta\cap h^{-1}(s)$ includes $p$, and for small epsilon $H_\theta\cap h^{-1}(t_0-\epsilon)$ and $H_\theta\cap h^{-1}(t_0+\epsilon)$ are both one-sheeted hyperboloids near $p$. Then $H_\theta$ is a singular $3$-manifold. We illustrate this situation in Figure~\ref{fig:singularfiber}. (We hope that what just took a lot of words to say becomes obvious in Figure~\ref{fig:singularfiber}.) We see now that it is {\emph{essential}} that the $3$-dimensional fibers of $Z^4$ ``resolve'' the same way at cone $p$.

\begin{figure}\begin{centering}
\begingroup%
  \makeatletter%
  \providecommand\color[2][]{%
    \errmessage{(Inkscape) Color is used for the text in Inkscape, but the package 'color.sty' is not loaded}%
    \renewcommand\color[2][]{}%
  }%
  \providecommand\transparent[1]{%
    \errmessage{(Inkscape) Transparency is used (non-zero) for the text in Inkscape, but the package 'transparent.sty' is not loaded}%
    \renewcommand\transparent[1]{}%
  }%
  \providecommand\rotatebox[2]{#2}%
  \newcommand*\fsize{\dimexpr\f@size pt\relax}%
  \newcommand*\lineheight[1]{\fontsize{\fsize}{#1\fsize}\selectfont}%
  \ifx\svgwidth\undefined%
    \setlength{\unitlength}{295.12601357bp}%
    \ifx\svgscale\undefined%
      \relax%
    \else%
      \setlength{\unitlength}{\unitlength * \real{\svgscale}}%
    \fi%
  \else%
    \setlength{\unitlength}{\svgwidth}%
  \fi%
  \global\let\svgwidth\undefined%
  \global\let\svgscale\undefined%
  \makeatother%
  \begin{picture}(1,0.60442814)%
    \lineheight{1}%
    \setlength\tabcolsep{0pt}%
    \put(0,0){\includegraphics[width=\unitlength,page=1]{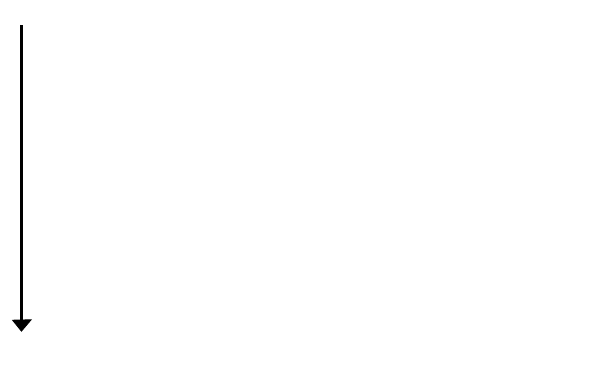}}%
    \put(0.0203035,0.23971736){\color[rgb]{0,0,0}\rotatebox{90}{\makebox(0,0)[lt]{\lineheight{0}\smash{\begin{tabular}[t]{l}decreasing $t$\end{tabular}}}}}%
    \put(0,0){\includegraphics[width=\unitlength,page=2]{Fig4.pdf}}%
  \end{picture}%
\endgroup%

\caption[Smooth vs. singular fibers.]{Left and right: cross-sections of a smooth $3$-manifold. Each has one singular cross-section (containing a cone). The nearby cross-sections are {\emph{different}} resolutions of that cone. Middle: cross-sections of a {\emph{singular}} $3$-manifold. Even though there is only an isolated singular cross-section (a single cone), the resolutions of the cone are {\emph{the same}} above and below the singular cross-section.}
\label{fig:singularfiber}
\end{centering}\end{figure}

We make this more precise:

\begin{definition}\label{typecone}
 Let $\F_t$ be a movie of singular fibrations on a $4$-manifold $Z^4$ with Morse function $h:Z^4\to I$. Let $t_0$ be a regular value of $h$. Let $x_{t_0}\in h^{-1}(t_0)$ be a critical point of $\F_{t_0}$. Let $B\subset Z^4$ be a small neighborhood of $x_{t_0}$. Take $B$ small so that for some small $\epsilon>0$, there is a unique critical point $x_t$ of $\F_t$ in the interior of $B$ for $t\in(t_0-\epsilon, t_0+\epsilon)$. 
 
 Let \[T_{\F,t_0}(x)=\frac{d}{ds}\bigg|_{s={t_0}}\F_s(x_s).\]

Let $p\in\F_{t_0}^{-1}(\theta_p)$ be a cone or half-cone singularity.
Let $q\in\F_{t_0}^{-1}(\theta_q)$ be a dot or half-dot singularity. Let $b\in\boundary\F_{t_0}^{-1}(\theta_b)$ be a bowl point. Assume $T_{\F,t_0}(\theta_p)$, $T_{\F,t_0}(\theta_q)$, $T_{\F,t_0}(\theta_b)\neq 0$.

 We define the type of $p$ or $q$ or $b$ with respect to $\F_t$ in Table~\ref{typetable}.
 
\begin{table}\begin{centering}
\begin{tabular}{c|cc}
&\multicolumn{2}{c}{index of $p$}\\
 &$1$&$2$\\
 \hline
 $T_{\F,t_0}(p)>0$&type I&type II\\
 $T_{\F,t_0}(p)<0$&type II&type I\\
 \hline\multicolumn{3}{c}{}\\
 &\multicolumn{2}{c}{index of $q$}\\
 &$0$&$3$\\
 \hline\\
 $T_{\F,t_0}(q)>0$&type \0&type III\\
 $T_{\F,t_0}(q)<0$&type III&type \0\\
 \hline\multicolumn{3}{c}{}\\
 &\multicolumn{2}{c}{index of $b$ in $\F_t$ restricted to $\boundary h^{-1}(t)$}\\
 &$0$&$2$\\
 \hline
 $T_{\F,t_0}(b)>0$&type \bone&type \btwo\\
 $T_{\F,t_0}(b)<0$&type \btwo&type \bone
 \end{tabular}\caption[The type of a singularity in a movie of singular fibrations.]{We define the type of a cone or half-cone $p$, a dot or half-dot $q$, or a bowl $b$ in a movie of singular fibrations.}\label{typetable}
 \end{centering}\end{table}

See Figure~\ref{fig:types}. 
In words, if as $t$ decreases the leaves of $\F_t$ near a cone $p$ change from hyperboloids of two sheets to hyperboloids of one sheet (the leaves ``fuse together''), we say $p$ is type I. If the leaves change from hyperboloids of one sheet to hyperboloids of two sheets (the leaves ``split apart''), we say $p$ is type II. Similarly for half-cones. 

If as $t$ decreases the leaves of $\F_t$  near a dot $q$ shrink, then we say $q$ is of type III. 
If the spheres expand, we say $q$ is of type \0. Similarly for half-dots.

If as $t$ decreases the leaves of $\F_t$ near a bowl 
point $b$ move away from $\boundary h^{-1}(t)$ near $b$, then 
we say $b$ is type \btwo\hspace{-3.3pt}. 
If as $t$ decreases, the leaves of $\F_t$ move toward $\boundary h^{-1}(t)$ near $b$, then we say $b$ is type \bone\hspace{-3.3pt}.

\begin{figure}\begin{centering}
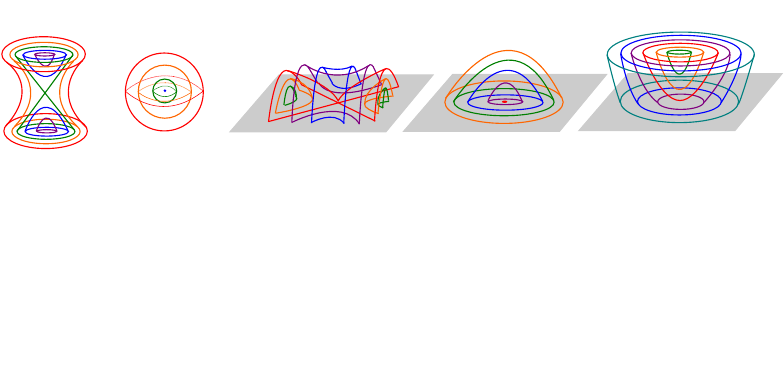
\caption[The type of a singularity in a movie of singular fibrations.]{The type of a singularity $p$ of $\F_t$. 
The arrow indicates the direction in which a singular cross-section of $H_\theta$ is resolved as $t$ decreases.}
\label{fig:types}
\end{centering}\end{figure}

  \end{definition}

Thus, for $H_\theta=\cup_t\F_t^{-1}(\theta)$ to be nonsingular, it must be the case that a singularity of $\F_t$ does not change type as $t$ varies (away from critical points of $h$). With this motivation, in this paper we will use the vertical fiber strategy while carefully ensuring that the types of singularities never change with $t$.

\begin{remark}
The name of a type of an interior cone or dot refers to a relative handle in $H_\theta=\cup_t\F_t^{-1}(\theta)$ (in a handle decomposition relative to $\F_1^{-1}(\theta)$). A cone of type I or II contributes a $1$- or $2$-handle, respectively. A dot of type \0 or III contributes a $0$- or $3$-handle, respectively.

A half-cone or half-dot contributes a handle with half the usual attaching region. That is, a copy of $B^k\times B^{3-k}$ attached along $($one hemisphere of $S^{k-1})\times B^{3-k}$ (where $S^0$ is its own hemisphere). Therefore, a half-cone or half-dot of type I or type \0 contributes a $1$- or $0$-handle respectively. A half-cone or half-dot of type II or III do not contribute handles to $H_\theta$.

A bowl point of type \btwo contributes a $2$-handle to $H_\theta$. A bowl point of type \bone does not contribute a handle to $H_\theta$.
\end{remark}

\begin{definition}\label{def:chart}
Let $\F_t$ be a movie of singular fibrations. We will keep track of the types and indices of singularities as $t$ decreases via a diagram we call a {\emph{singularity chart}} of $\F_t$. The vertical axis is $t$ (with $t=1$ at the top and $t=0$ at the bottom). The horizontal axis is $S^1=\F_t(h^{-1}(t))$.
We plot the singular values of $\F_t$, so that as $t$ varies these points trace out arcs with some endpoints where a singularity is born or dies. We label each arc with the type of the corresponding singularity. Singularities in $\boundary h^{-1}(t)$ are drawn with dashed lines. (Note this is essentially a Cerf chart.)
\end{definition}

\begin{definition}\label{def:valid}
A singularity chart is said to be {\emph{valid}} if it satisfies the following conditions.
\begin{itemize}
\item All arcs are smooth in their interiors and are never vertical.
\item Arcs intersect transversely in their interiors.
\item At each endpoint of an arc, there are two or three arcs which share that endpoint. Near the endpoint, the chart looks like one of the subcharts in Figure~\ref{fig:chartjoins}.
\end{itemize}

If a singularity chart is not valid, we say it is {\emph{invalid}}. See Figure~\ref{fig:examplechart} for an example of a valid singularity chart.
\end{definition}

\begin{figure}\begin{centering}
\scalebox{0.8}{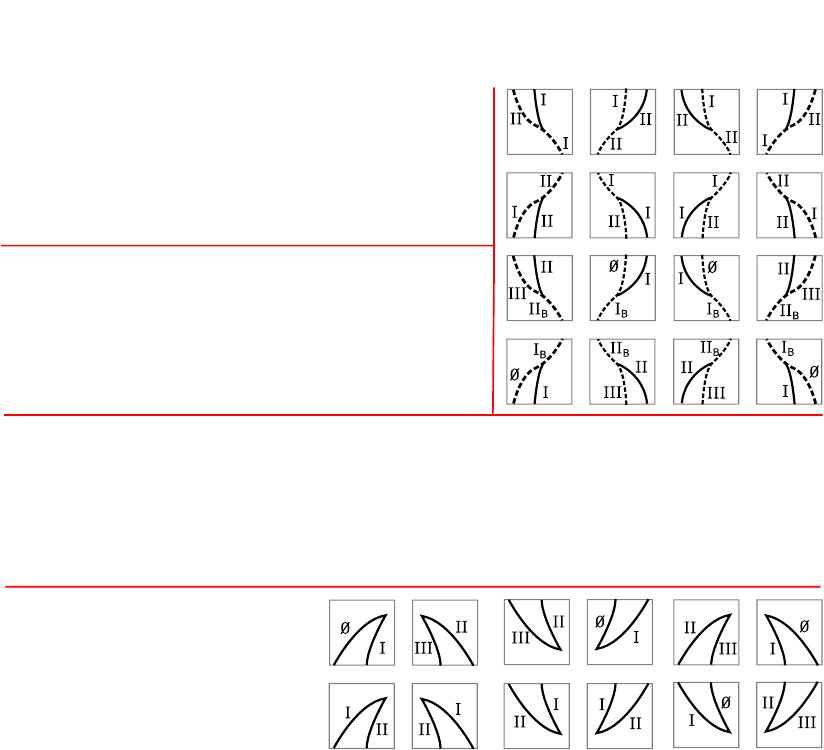}
\caption[Possible configurations of two or three arcs meeting at an endpoint in a valid singularity chart.]{Possible configurations of two or three arcs meeting at an endpoint in a valid singularity chart.}
\label{fig:chartjoins}
\end{centering}\end{figure}

\begin{figure}\begin{centering}
\scalebox{1.2}{
\begingroup%
  \makeatletter%
  \providecommand\color[2][]{%
    \errmessage{(Inkscape) Color is used for the text in Inkscape, but the package 'color.sty' is not loaded}%
    \renewcommand\color[2][]{}%
  }%
  \providecommand\transparent[1]{%
    \errmessage{(Inkscape) Transparency is used (non-zero) for the text in Inkscape, but the package 'transparent.sty' is not loaded}%
    \renewcommand\transparent[1]{}%
  }%
  \providecommand\rotatebox[2]{#2}%
  \newcommand*\fsize{\dimexpr\f@size pt\relax}%
  \newcommand*\lineheight[1]{\fontsize{\fsize}{#1\fsize}\selectfont}%
  \ifx\svgwidth\undefined%
    \setlength{\unitlength}{125.80002071bp}%
    \ifx\svgscale\undefined%
      \relax%
    \else%
      \setlength{\unitlength}{\unitlength * \real{\svgscale}}%
    \fi%
  \else%
    \setlength{\unitlength}{\svgwidth}%
  \fi%
  \global\let\svgwidth\undefined%
  \global\let\svgscale\undefined%
  \makeatother%
  \begin{picture}(1,1.88258461)%
    \lineheight{1}%
    \setlength\tabcolsep{0pt}%
    \put(0,0){\includegraphics[width=\unitlength,page=1]{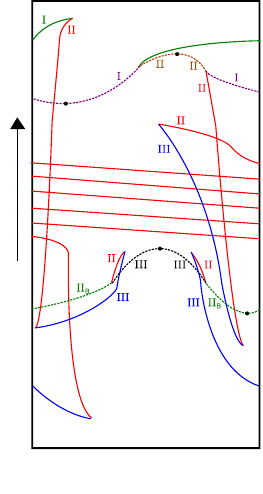}}%
    \put(-0.00902832,1.12724921){\color[rgb]{0,0,0}\makebox(0,0)[lt]{\lineheight{0}\smash{\begin{tabular}[t]{l}$t$\end{tabular}}}}%
    \put(0,0){\includegraphics[width=\unitlength,page=2]{Fig7.pdf}}%
    \put(0.54576445,0.01598554){\color[rgb]{0,0,0}\makebox(0,0)[lt]{\lineheight{0}\smash{\begin{tabular}[t]{l}$\theta$\end{tabular}}}}%
  \end{picture}%
\endgroup%
}
\caption[An example of a valid singularity chart.]{An example of a valid singularity chart. Note that no arc is ever vertical. Each endpoint locally looks like one of the images in Figure~\ref{fig:chartjoins}. Colors/shades of the arcs are only meant to help the reader view the image; they contain no mathematical meaning.}
\label{fig:examplechart}
\end{centering}\end{figure}

Note that the index of a singularity and the slope of its arc in a singularity chart determine its type. If on the arc, $\theta$ decreases as $t$ decreases (i.e. the arc has positive slope), then the index and type agree (e.g. an index-$2$ cone is type II). If $\theta$ increases as $t$ decreases (i.e. the arc has negative slope), then the index and type do not agree (e.g. an index-$2$ cone is type I). Therefore, labelling each arc of a valid singularity chart with only the type of the singularity it corresponds to also determines the indices of each singularity.

\begin{remark}
In a valid singularity chart, singularity arcs which cross in their interior (i.e. correspond to singular points whose heights interchange) may have the same or different types. Crossing data does not affect the validity of the chart.
\end{remark}

\begin{definition}
We say that a movie of singular fibrations $\F_t$ is {\emph{valid}} if the singularity chart for $\F_t$ is valid after potentially perturbing $\F_t$ away from critical points of $h$.

That is, $\F_t$ is valid if it is valid near critical points of $h$, and each arc in the singularity chart for $\F_t$ attains one type (sign of slope), but may sometimes have undefined type (vertical slope). Singularity arcs may be tangent or coincide in some interval, but a small perturbation of $\F_t$ should make them intersect transversely.

In the singularity chart for a valid movie, we label each arc with the unique type it ever attains.
\end{definition}

From now on, when $\F_t$ is a valid movie, we will implicitly perturb $\F_t$ to have valid singularity chart.

Theorem~\ref{fibrationthm} essentially follows from the definition of a valid singularity chart.

\begin{fibrationthm}
Let $\F_t$ be a movie of singular fibrations on $Z^4$ with a valid singularity chart so that $\F_0$ and $\F_1$ are nonsingular. 
Then $Z^4$ is fibered over $S^1$ by $\{H_\theta\}$, where $H_\theta=\cup_t \F_t^{-1}(\theta)$.
\end{fibrationthm}
\begin{proof}
Because the types of an arc in a singularity chart $\mathcal{C}$ for $\F_t$ never change, each $H_\theta$ is nonsingular away from points in $Z^4$ corresponding to endpoints of arcs in $\mathcal{C}$.

Where $\mathcal{C}$ has a cusp (as in a standard Cerf diagram), the $H_\theta$ are smooth. Near the endpoint, the arcs are contributing geometrically cancelling handles to each $H_\theta$ (Or, in the case of a type II-,III- or \bone-,II- boundary pair, no handles at all.)

Where three arcs in $\mathcal{C}$ have a common endpoint (i.e. at a cusp of an interior and boundary singularity, with another boundary singularity on the other side), the $H_{\theta}$ are similarly smooth. Near the endpoint, the interior singularity arc contributes a $k$-handle, where $k=\{0,1,2,3\}$ if the singularity is type $\{$\0, I, II, III$\}$. The boundary singularity cancelled by the interior contributes a $D^j\times D^{3-j}$ attached along (hemisphere of $S^{j-1}$)$\times D^{3-j}$, where $j=\{0,1,2,3\}$ if the singularity is type $\{$\0, I or \bone, II or \btwo, III$\}$.

Suppose $k=1$ and $j=2$. The handle $D^1\times D^{2}$ and half-handle $D^2\times D^{1}$ glue along (hemisphere of $S^1$)$\times D^1$ to form a $D^3$ attached along (hemisphere of $S^0$)$\times D^2$, exactly the handle contribution of a type I half-cone. (Recall that $S^0$ is its own hemisphere.) Similarly for other cancelling $k$ and $j$. These points in the singularity chart should be thought of as a standard critical point birth/death, intersected with a half-ball whose boundary meets the critical point.

Finally, we consider a common endpoint of two arcs in $\mathcal{C}$ which is not a cusp. This endpoint is at a local maximum or minimum of arcs in $\mathcal{C}$, so can be thought of as a birth or death of two arcs. The two arcs are of the same type (opposite index). 
Recall that in the definition of a singular fibration on a singular $3$-manifold $M^3=h^{-1}(t)$, we specified six possible configurations of $\F^{-1}(t)$ near a singular point in $\boundary M^3$ or $\mathring{M^3}$. (This was Figure~\ref{fig:singularmfd}). Say $x\in h^{-1}(t_0)$ is a critical point of $h$. Then $\F^{-1}(t_0)$ and the type of $x$ determines the topology of $\F^{-1}(t_0\pm\epsilon)$ near $x$. This determines the topology of $H_\theta$ near $x$. In particular, since $\F$ is valid, the $H_\theta$ containing $x$ is nonsingular near $x$. See Figure~\ref{fig:boundarychangeokay}. (In Section~\ref{thm110}, we will see that this pair of births/deaths corresponds to a critical point of $h$, with index determined by the types of the singularities.)

Thus, $\{H_\theta\}$ is a fibration of $Z^4$ over $S^1$.

\begin{figure}\begin{centering}
\begingroup%
  \makeatletter%
  \providecommand\color[2][]{%
    \errmessage{(Inkscape) Color is used for the text in Inkscape, but the package 'color.sty' is not loaded}%
    \renewcommand\color[2][]{}%
  }%
  \providecommand\transparent[1]{%
    \errmessage{(Inkscape) Transparency is used (non-zero) for the text in Inkscape, but the package 'transparent.sty' is not loaded}%
    \renewcommand\transparent[1]{}%
  }%
  \providecommand\rotatebox[2]{#2}%
  \newcommand*\fsize{\dimexpr\f@size pt\relax}%
  \newcommand*\lineheight[1]{\fontsize{\fsize}{#1\fsize}\selectfont}%
  \ifx\svgwidth\undefined%
    \setlength{\unitlength}{309.9537935bp}%
    \ifx\svgscale\undefined%
      \relax%
    \else%
      \setlength{\unitlength}{\unitlength * \real{\svgscale}}%
    \fi%
  \else%
    \setlength{\unitlength}{\svgwidth}%
  \fi%
  \global\let\svgwidth\undefined%
  \global\let\svgscale\undefined%
  \makeatother%
  \begin{picture}(1,1.49214484)%
    \lineheight{1}%
    \setlength\tabcolsep{0pt}%
    \put(0,0){\includegraphics[width=\unitlength,page=1]{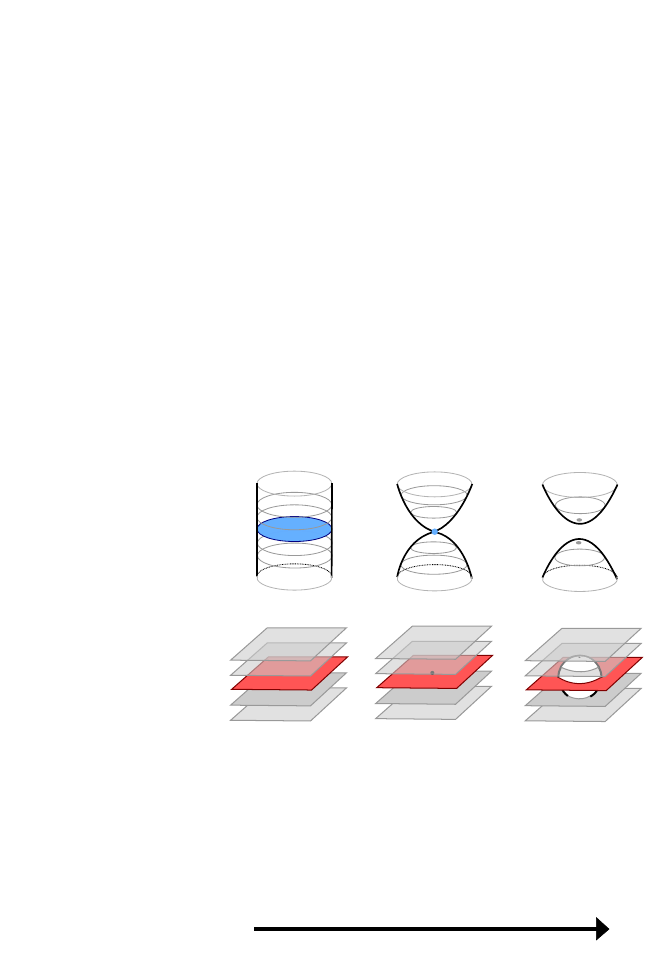}}%
    \put(0.55759459,0.00764723){\color[rgb]{0,0,0}\makebox(0,0)[lt]{\lineheight{0}\smash{\begin{tabular}[t]{l}decreasing $t$\end{tabular}}}}%
    \put(0,0){\includegraphics[width=\unitlength,page=2]{Fig8.pdf}}%
    \put(-0.00301272,1.46785552){\color[rgb]{0,0,0}\makebox(0,0)[lt]{\lineheight{0}\smash{\begin{tabular}[t]{l}interior\end{tabular}}}}%
    \put(0,0){\includegraphics[width=\unitlength,page=3]{Fig8.pdf}}%
    \put(0.16520559,1.46785552){\color[rgb]{0,0,0}\makebox(0,0)[lt]{\lineheight{0}\smash{\begin{tabular}[t]{l}boundary\end{tabular}}}}%
  \end{picture}%
\endgroup%

\caption[A movie of singular fibrations near a critical point of the height function.]{If $\F_t$ has a valid movie of singular fibrations, then near a critical point of $h$, $\F_t$ looks like one of these movies or their inverses ($t\mapsto 1-t$). We shade the $H_\theta$ that includes the critical point of $h$, and find that $H_\theta$ is smooth near the critical point. We have drawn each picture as if the singularity is in $\boundary M^3$; double the picture to obtain a movie containing an interior critical point of $h$. 
}\label{fig:boundarychangeokay}
\end{centering}\end{figure}
\end{proof}

\section{Building blocks: simple movies of fibrations\label{sec:blocks}}
In this section, we construct several explicit movies of singular fibrations. 

Every movie $\F_t:Z^4\to S^1$ (with height function $h:Z^4\to I$) from this section has the following property:  Up to isotopy $\F_t$ agrees with $\F_1$ on $(\boundary h^{-1}(1)) \cap(\pi\boundary h^{-1}(t))$,  under a natural projection $\pi$.

Let $\G$ be an arbitrary singular fibration on a $3$-manifold $M^3$. Suppose that $\G$ agrees with $\F_1$ on some copy $h^{-1}(1)\subset M^3$. Then we can extend $\G$ to a movie $\G_t$ on $[(M^3\setminus h^{-1}(1))\times I]\cup Z^4$ where $\G_t=\G$ outside $Z^4$ and $\G_t=\F_t$ inside $Z^4$.

We will say that we build $\G_t$ by ``playing'' Movie $\F_t$, where the identification $h^{-1}(1)\into M^3$ is clear. We might refer to $\G_t$ as agreeing with the ``top'' of $\F_t$ in $h^{-1}(1)\subset M^3$.

At the end of each subsection, we produce a valid singularity chart for each movie. 

To make the definition of these movies simpler, we will consider a special singular fibration on $B^3$.

\begin{definition}
A singular fibration $\F$ on $B^3=D^2\times I/(D^2\times 0\sim\pt, D^2\times 1\sim\pt)$ is {\emph{standard}} if each component of $\F^{-1}(\theta)$ (for all $\theta)$ is $D^2\times\pt$. A standard singular fibration of $B^3$ has no interior singularities, and has two half-dot singularities in $\boundary B^3$.
\end{definition}

\subsection{Adding or cancelling singular points}

We define movies corresponding to adding or cancelling a pair of singularities. 

\begin{move}[Interior stabilization and destabilization]\label{addconedot}
See Figure~\ref{fig:addcones}.

Let $Z^4=B^3\times I$, where $h:Z^4\to I$ is projection onto the second factor. Let $\F_1$ be a standard singular fibration of $B^3\times 1$. View $\F_1: B^3\to S^1$ as a circular Morse function. As $t$ decreases to $1/2$, obtain $\F_t$ from $\F_1$ by perturbing to add a cancelling-index pair. Then $\F_{1/2}$ agrees with $\F_1$ in a neighborhood of $\boundary B^3$, but $\F_{1/2}$ has exactly two interior singularities: one cone and one dot if the cancelling pair are index-$0$,-$1$ or -$2$,-$3$; two cones if the cancelling pair are index-$1$,-$2$. Extend $\F_t$ vertically to $t=0$. We call $\F_t$ an {\emph{interior stabilization movie}} (we generally include the indices of the birthed critical points but may drop the word ``interior'').

Let $\bar{h}:Z^4\to I$ be given by $\bar{h}(z)=1-h(z)$. Let $\G_t=\F_{1-t}$. With respect to $\bar{h}$, we call $\G_t$ an {\emph{interior destabilization movie}}.

We note that in the description of $\F_t$, it is not clear what are the types of the two singularities. In fact, we may parametrize $\F_t$ so that the types of the singularities either both agree or both disagree with their indices (i.e. an index-$0$,-$1$ pair may be types \0 and I or types III and II). See Figure~\ref{fig:stabtype}. 
\end{move}

\begin{figure}\begin{centering}
\begingroup%
  \makeatletter%
  \providecommand\color[2][]{%
    \errmessage{(Inkscape) Color is used for the text in Inkscape, but the package 'color.sty' is not loaded}%
    \renewcommand\color[2][]{}%
  }%
  \providecommand\transparent[1]{%
    \errmessage{(Inkscape) Transparency is used (non-zero) for the text in Inkscape, but the package 'transparent.sty' is not loaded}%
    \renewcommand\transparent[1]{}%
  }%
  \providecommand\rotatebox[2]{#2}%
  \newcommand*\fsize{\dimexpr\f@size pt\relax}%
  \newcommand*\lineheight[1]{\fontsize{\fsize}{#1\fsize}\selectfont}%
  \ifx\svgwidth\undefined%
    \setlength{\unitlength}{249.99337745bp}%
    \ifx\svgscale\undefined%
      \relax%
    \else%
      \setlength{\unitlength}{\unitlength * \real{\svgscale}}%
    \fi%
  \else%
    \setlength{\unitlength}{\svgwidth}%
  \fi%
  \global\let\svgwidth\undefined%
  \global\let\svgscale\undefined%
  \makeatother%
  \begin{picture}(1,0.82563786)%
    \lineheight{1}%
    \setlength\tabcolsep{0pt}%
    \put(0,0){\includegraphics[width=\unitlength,page=1]{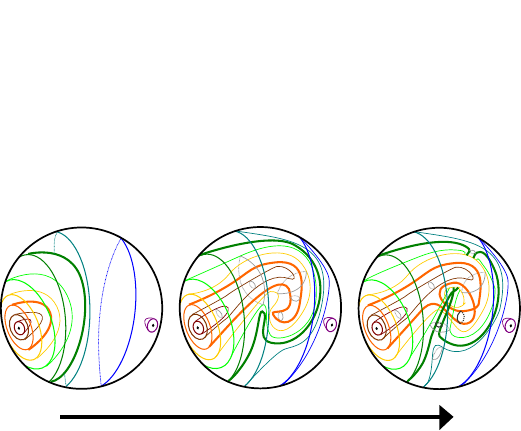}}%
    \put(0.41529764,0.03390663){\color[rgb]{0,0,0}\makebox(0,0)[lt]{\lineheight{0}\smash{\begin{tabular}[t]{l}decreasing $t$\end{tabular}}}}%
    \put(0,0){\includegraphics[width=\unitlength,page=2]{Fig9.pdf}}%
    \put(0.41649973,0.46681118){\color[rgb]{0,0,0}\makebox(0,0)[lt]{\lineheight{0}\smash{\begin{tabular}[t]{l}decreasing $t$\end{tabular}}}}%
    \put(0,0){\includegraphics[width=\unitlength,page=3]{Fig9.pdf}}%
  \end{picture}%
\endgroup%

\caption[A diagram of an interior stabilization movie.]{A diagram of an interior stabilization movie. As $t$ decreases, we perturb $\F_t$ to add a cancelling pair of critical points. Top: an index-$0$,-$1$ or -$2$,-$3$ pair, depending on the orientation of $\F_t$. Bottom: an index-$1$,-$2$ pair (the order depends on the orientation of $\F_t$).}
\label{fig:addcones}
\end{centering}\end{figure}

\begin{figure}\begin{centering}
\begingroup%
  \makeatletter%
  \providecommand\color[2][]{%
    \errmessage{(Inkscape) Color is used for the text in Inkscape, but the package 'color.sty' is not loaded}%
    \renewcommand\color[2][]{}%
  }%
  \providecommand\transparent[1]{%
    \errmessage{(Inkscape) Transparency is used (non-zero) for the text in Inkscape, but the package 'transparent.sty' is not loaded}%
    \renewcommand\transparent[1]{}%
  }%
  \providecommand\rotatebox[2]{#2}%
  \newcommand*\fsize{\dimexpr\f@size pt\relax}%
  \newcommand*\lineheight[1]{\fontsize{\fsize}{#1\fsize}\selectfont}%
  \ifx\svgwidth\undefined%
    \setlength{\unitlength}{302.86214033bp}%
    \ifx\svgscale\undefined%
      \relax%
    \else%
      \setlength{\unitlength}{\unitlength * \real{\svgscale}}%
    \fi%
  \else%
    \setlength{\unitlength}{\svgwidth}%
  \fi%
  \global\let\svgwidth\undefined%
  \global\let\svgscale\undefined%
  \makeatother%
  \begin{picture}(1,0.64883074)%
    \lineheight{1}%
    \setlength\tabcolsep{0pt}%
    \put(0,0){\includegraphics[width=\unitlength,page=1]{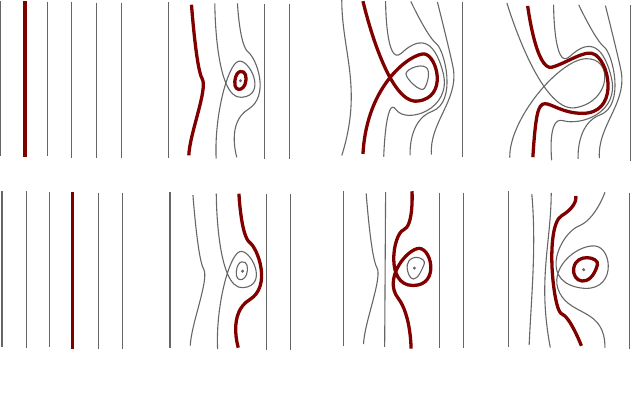}}%
    \put(0.37413083,0.00764628){\color[rgb]{0,0,0}\makebox(0,0)[lt]{\lineheight{0}\smash{\begin{tabular}[t]{l}decreasing $t$\end{tabular}}}}%
    \put(0,0){\includegraphics[width=\unitlength,page=2]{Fig10.pdf}}%
  \end{picture}%
\endgroup%

\caption[Schema of interior stabilization movies.]{Each row is a schematic of an interior stabilization movie (either index-$0$,-$1$ or -$2$,-$3$). In the top movie, the cone is type I and the dot is type {\0}. In the bottom, the cone is type II and the dot is type III.}
\label{fig:stabtype}
\end{centering}\end{figure}

\begin{move}[Boundary stabilization and destabilization]
See Figure~\ref{fig:boundarystab}.

Let $Z^4=B^3\times I$, where $h:Z^4\to I$ is projection onto the second factor. Let $\F_1$ be a standard singular fibration of $B^3\times 1$. View $\F_1: h^{-1}(1)\to S^1$ as a circular Morse function. 
Fix $x\in \boundary h^{-1}(1)$ so that there are no boundary singularities of $\F_1$ in a small neighborhood $\nu(x)$. 
As $t$ decreases to $1/2$, obtain $\F_t$ from $\F_1$ by perturbing to add a cancelling pair of critical points at $x$. Then $\F_{1/2}$ agrees with $\F_1$ outside $\nu(x)$, but has two boundary singularities in $\nu(x)$. 

If the stabilization is $0$-,$1$- or $2$-,$3$- stabilization, then the new boundary singularities are a half-cone and half-dot. If the stabilization is a $1$-,$2$- stabilization, then the new boundary singularities are a half-cone and a bowl. We call $\F_t$ a {\emph{boundary stabilization movie}} (and specify the indices of the birthed singularities). 

Let $\bar{h}:Z^4\to I$ be given by $\bar{h}(z)=1-h(z)$. Let $\G_t=\F_{1-t}$. With respect to $\bar{h}$, we call $\G_t$ a {\emph{boundary destabilization movie}}.

As in the interior stabilization movie, the types of the singularities of $\F_t$ may be taken to agree with their indices, or to be opposite.
\end{move}

\begin{figure}\begin{centering}
\includegraphics[width=.65\textwidth]{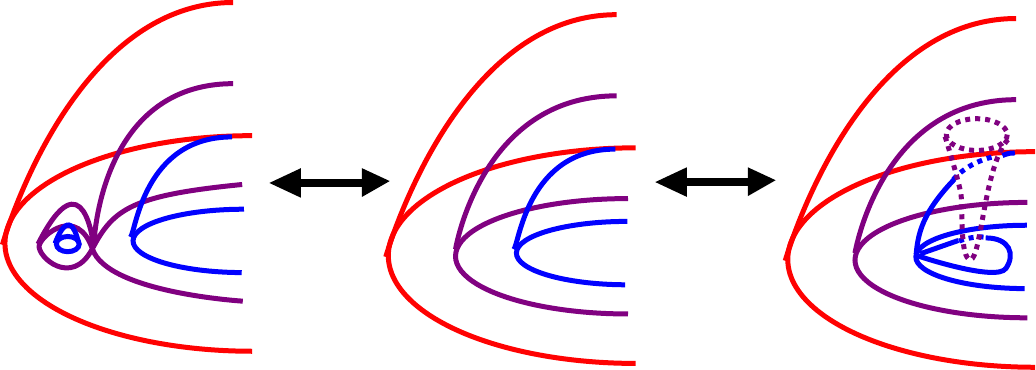}
\caption[As $t$ decreases, we stabilize or destabilize $\F_t$ at a point in $\boundary h^{-1}(t)$.]{As $t$ decreases, we stabilize or destabilize $\F_t$ at a point in $\boundary h^{-1}(t)$.}
\label{fig:boundarystab}
\end{centering}\end{figure}

In Figure~\ref{fig:interiorstab}, we give a valid singularity chart for each interior stabilization and destabilization movie. We give two charts for each, which describe different parametrizations of a (de)stabilization movie. We omit the boundary singularities, which persist through the whole movie.

\begin{figure}\begin{centering}
\includegraphics[width=.5\textwidth]{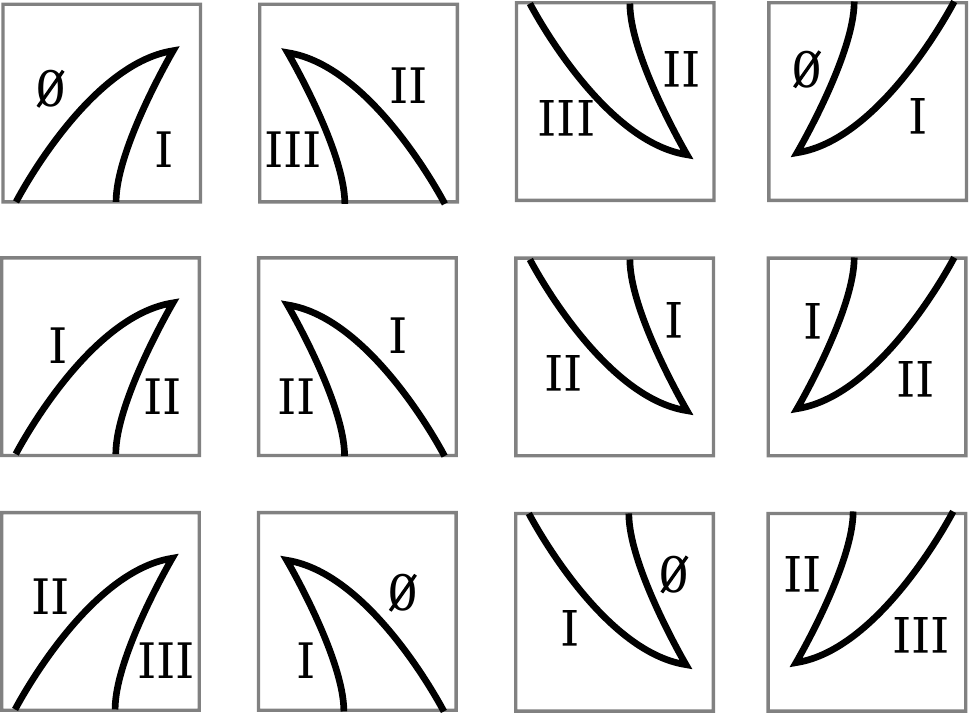}
\caption[Valid singularity charts for each interior stabilization/destabilization.]{Valid singularity charts for each interior stabilization/destabilization. The left half are interior stabilization, while the right half are destabilization. Top row: 0-,~\hspace{-3.4mm}~1-(de)stabilization. Middle row: $1$-,$2$-(de)stabilization. Note that either of the cones may be taken to be the type I cone. Bottom row: $2$-,$3$-(de)stabilization.}
\label{fig:interiorstab}
\end{centering}\end{figure}

In Figure~\ref{fig:boundarystabcharts}, we give a valid singularity chart for each boundary stabilization and destabilization movie. We give two charts for each, which describe different parametrizations of a (de)stabilization movie. We omit other singularities which persist through the whole movie.

\begin{figure}\begin{centering}
\includegraphics[width=.5\textwidth]{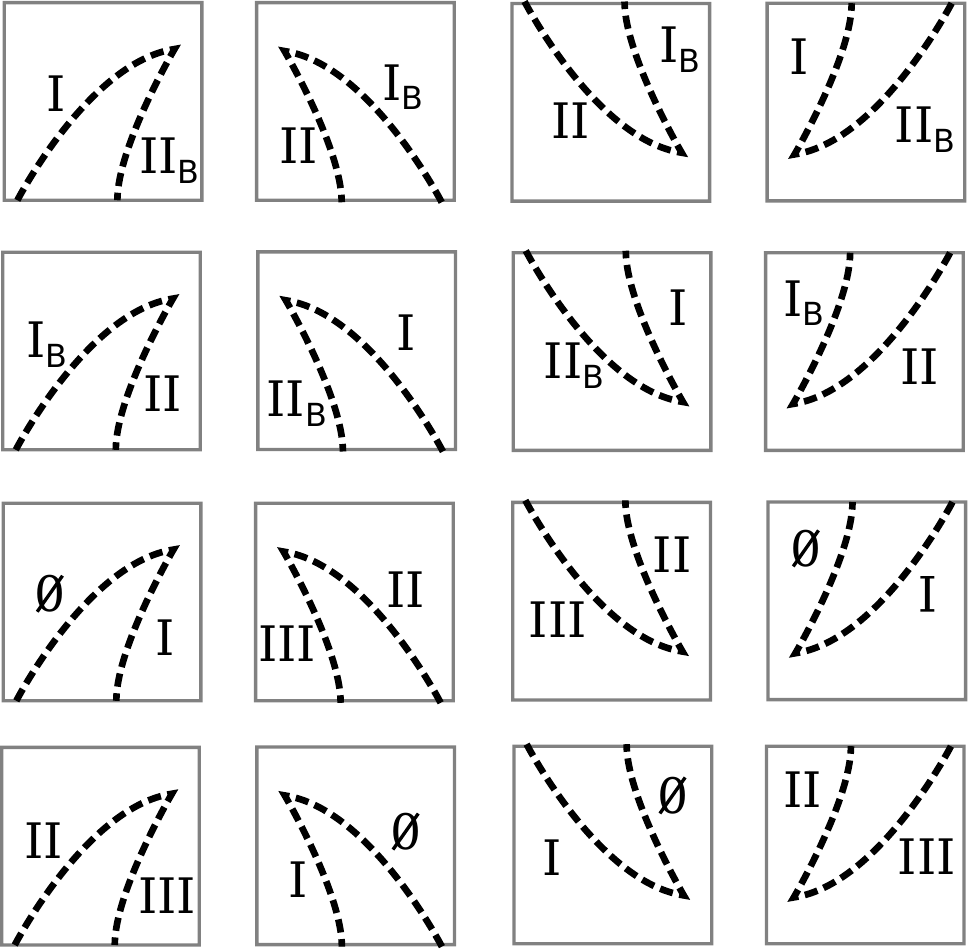}
\caption[Valid singularity charts for each boundary stabilization/destabilization.]{Valid singularity charts for each boundary stabilization/destabilization. The left half are boundary stabilization, while the right half are boundary destabilization. Top and second row: $1$-,~$2$- (de)stabilization. Third row: $0$-,$1$-(de)stabilization. Bottom row: $2$-,$3$-(de)stabilization.}
\label{fig:boundarystabcharts}
\end{centering}\end{figure}

\subsection{Movies on $Z^4$ where $h:Z^4\to I$ has one critical point}\label{thm110}
In the proof of Theorem~\ref{fibrationthm}, we implicitly described movies of singular fibrations on $Z^4$ where $h:Z^4\to I$ has one critical point. Now we make this more explicit. Each of these movies involves the birth or death of one pair of singularities of the same type. At the end of this subsection, we provide a table relating the types of these singularities with the index of the unique critical point of $h$.

In short, a birth of a pair of interior type $\0$, I, II, or III singularities corresponds to a critical point of $h$ of index $4,3,2$, or $1$, respectively. The apparent inversion in those orderings is due to the fact that we are working with decreasing $t$. It may be helpful to think, ``a birth of type $n$ singularities corresponds to a relative $n$-handle in a handle decomposition of $Z^4$"

A death of a pair of interior type $\0$, I, II, or III singularities corresponds to a critical point of $h$ of index $3,2,1$, or $0$, respectively. It may be helpful to think, ``a death of type $n$ singularities corresponds to cancelling a relative $n$-handle, i.e. attaching a relative $(n+1)$-handle in a handle decomposition of $Z^4$."

Similarly, a birth/death of boundary singularities corresponds to attachment of a relative half-handle, in the language of~\cite{BNR}. Attaching a {\emph{right half $n$-handle}} to a manifold $M$ with boundary yields a manifold diffeomorphic to the result of attaching an $n$-handle to $M$. On the other hand, attaching a {\emph{left half $n$-handle} to $M$ yields a manifold diffeomorphic to $M$.

A birth of a pair of boundary type \0, I, or \btwo singularities corresponds to a attaching a relative right half $0,1,$ or $2$-handle, respectively.
A death of a pair of boundary type \0, I, or \btwo singularities corresponds to a attaching a relative left half $1,2,$ or $3$-handle, respectively.

A birth of a pair of boundary type I$_{\text{B}}$, II, or III singularities corresponds to attaching a relative left half $1,2,$ or $3$-handle, respectively.
A death of a pair of boundary type I$_{\text{B}}$, II, or III singularities corresponds to attaching a relative right half $2,3,$ or $4$-handle, respectively.

This discussion is summarized in a table at the end of this subsection.

We start with the case that the critical point of $h$ is contained in $\boundary Z^4$.

\begin{move}[Boundary fusion/compression movie 1]
Let $Z^4\cong B^4$ and $h:Z^4\to I$ so that $h^{-1}(1)=B^3$, $h^{-1}(0)\cong S^1\times D^2$, and $h$ has one critical point. 
Say the critical value is $1/2$.

Let $\F_t$ be the movie of Figure~\ref{fig:boundaryfusion1}. That is, let $\F_t$ be a standard fibration on $B^3$ for $t>1/2$. One leaf of $\F_{1/2}$ is a wedge of two disks along a boundary point. For $t<1/2$, $\F_{t}$ has two half-cone singularities of opposite index. $\F_t$ has a singularity chart containing a birth of two type II half-cones, of opposite index.

We call $\F_t$ the {\emph{first boundary fusion movie}} (or boundary fusion movie 1).

Let $\bar{h}:Z^4\to I$ be given by $\bar{h}(z)=1-h(z)$. Let $\G_t=\F_{1-t}$. With respect to $\bar{h}$, $\G_t$ has a singularity chart containing a death of two type I half-cones, of opposite index.

We call $\G_t$ the {\emph{first boundary compression movie}} (or boundary compression movie 1).

\begin{figure}\begin{centering}
\begingroup%
  \makeatletter%
  \providecommand\color[2][]{%
    \errmessage{(Inkscape) Color is used for the text in Inkscape, but the package 'color.sty' is not loaded}%
    \renewcommand\color[2][]{}%
  }%
  \providecommand\transparent[1]{%
    \errmessage{(Inkscape) Transparency is used (non-zero) for the text in Inkscape, but the package 'transparent.sty' is not loaded}%
    \renewcommand\transparent[1]{}%
  }%
  \providecommand\rotatebox[2]{#2}%
  \newcommand*\fsize{\dimexpr\f@size pt\relax}%
  \newcommand*\lineheight[1]{\fontsize{\fsize}{#1\fsize}\selectfont}%
  \ifx\svgwidth\undefined%
    \setlength{\unitlength}{330.57342289bp}%
    \ifx\svgscale\undefined%
      \relax%
    \else%
      \setlength{\unitlength}{\unitlength * \real{\svgscale}}%
    \fi%
  \else%
    \setlength{\unitlength}{\svgwidth}%
  \fi%
  \global\let\svgwidth\undefined%
  \global\let\svgscale\undefined%
  \makeatother%
  \begin{picture}(1,0.38031432)%
    \lineheight{1}%
    \setlength\tabcolsep{0pt}%
    \put(0,0){\includegraphics[width=\unitlength,page=1]{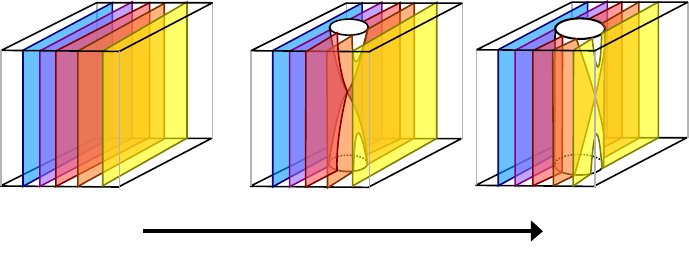}}%
    \put(0.40779068,0.00639541){\color[rgb]{0,0,0}\makebox(0,0)[lt]{\lineheight{0}\smash{\begin{tabular}[t]{l}decreasing $t$\end{tabular}}}}%
  \end{picture}%
\endgroup%

\caption[The first boundary fusion movie.]{The first boundary fusion movie.} 
\label{fig:boundaryfusion1}
\end{centering}\end{figure}
\end{move}

\begin{move}[Boundary fusion/compression movie 2]
Let $Z^4\cong B^4$ and $h:Z^4\to I$ so that $h^{-1}(1)=B^3$, $h^{-1}(0)\cong S^1\times D^2$, and $h$ has one critical point. 
Say the critical value is $1/2$.

Let $\F_t$ be the movie of Figure~\ref{fig:boundaryfusion2}. $\F_1$ is a fibration with no interior singularities and with two bowls (of opposite indices), two half-cones (opposite indices), and two half-dots (opposite indices). (We don't draw the boundary singularities.) 
In particular, there is an arc between the two bowl points which is transverse to the leaves of $\F_1$ -- in Figure~\ref{fig:boundaryfusion2}, this arc is parallel to the $y$-axis. In $\F_{1/2}$, one leaf $\F_{1/2}^{-1}(\theta_0)$ is a disk and contains the critical point of $h$. For $t<1/2$, for $\theta$ near $\theta_0$, $\F_t^{-1}(\theta)$ is an annulus. 

$\F_t$ has a singularity chart containing a death of two type \bone bowls, of opposite index.

We call $\F_t$ the {\emph{second boundary fusion movie}} (or boundary fusion movie 2).

Let $\bar{h}:Z^4\to I$ be given by $\bar{h}(z)=1-h(z)$. Let $\G_t=\F_{1-t}$. With respect to $\bar{h}$, $\G_t$ has a singularity chart containing a birth of two type \btwo bowls, of opposite index.

We call $\G_t$ the {\emph{second boundary compression movie}} (or boundary compression movie 2).

\begin{figure}\begin{centering}
\begingroup%
  \makeatletter%
  \providecommand\color[2][]{%
    \errmessage{(Inkscape) Color is used for the text in Inkscape, but the package 'color.sty' is not loaded}%
    \renewcommand\color[2][]{}%
  }%
  \providecommand\transparent[1]{%
    \errmessage{(Inkscape) Transparency is used (non-zero) for the text in Inkscape, but the package 'transparent.sty' is not loaded}%
    \renewcommand\transparent[1]{}%
  }%
  \providecommand\rotatebox[2]{#2}%
  \newcommand*\fsize{\dimexpr\f@size pt\relax}%
  \newcommand*\lineheight[1]{\fontsize{\fsize}{#1\fsize}\selectfont}%
  \ifx\svgwidth\undefined%
    \setlength{\unitlength}{340.36003533bp}%
    \ifx\svgscale\undefined%
      \relax%
    \else%
      \setlength{\unitlength}{\unitlength * \real{\svgscale}}%
    \fi%
  \else%
    \setlength{\unitlength}{\svgwidth}%
  \fi%
  \global\let\svgwidth\undefined%
  \global\let\svgscale\undefined%
  \makeatother%
  \begin{picture}(1,0.35201235)%
    \lineheight{1}%
    \setlength\tabcolsep{0pt}%
    \put(0,0){\includegraphics[width=\unitlength,page=1]{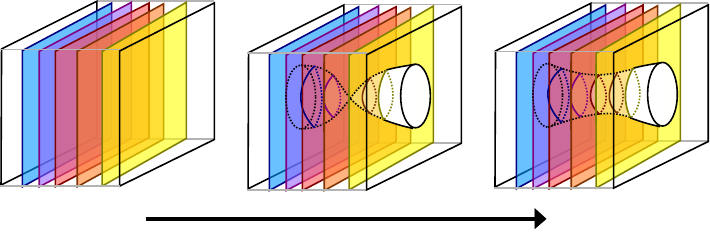}}%
    \put(0.40124909,0.00621121){\color[rgb]{0,0,0}\makebox(0,0)[lt]{\lineheight{0}\smash{\begin{tabular}[t]{l}decreasing $t$\end{tabular}}}}%
  \end{picture}%
\endgroup%

\caption[The second boundary fusion movie.]{The second boundary fusion movie. We only draw a neighborhood of the arc between the two bowls, ignoring the other boundary singularities.} 
\label{fig:boundaryfusion2}
\end{centering}\end{figure}

\end{move}

\begin{move}[Boundary fusion/compression movie 3]
Let $Z^4\cong B^4$ and $h:Z^4\to I$ so that $h^{-1}(1)=B^3\sqcup B^3$, $h^{-1}(0)=B^3$, and $h$ has one critical point. 
Say the critical value is $1/2$.

Let $\F_t$ be the movie of Figure~\ref{fig:boundaryfusion3}.

That is, let $\F_t$ be a standard fibration on each component of $B^3\sqcup B^3$ for $t>1/2$. The leaf of $\F_{1/2}$ containing the critical point of $h$ is a wedge of two disks. For $t<1/2$, $\F_{t}$ has 4 half-dots and two half-cones.

$\F_t$ has a singularity chart containing a birth of two type I half-cones, of opposite index.

We call $\F_t$ the {\emph{third boundary fusion movie}} (or boundary fusion movie 3).

Let $\bar{h}:Z^4\to I$ be given by $\bar{h}(z)=1-h(z)$. Let $\G_t=\F_{1-t}$. With respect to $\bar{h}$, $\G_t$ has a singularity chart containing a death of two type II half-cones, of opposite index.

We call $\G_t$ the {\emph{third boundary compression movie}} (or boundary compression movie 3).

\begin{figure}\begin{centering}
\begingroup%
  \makeatletter%
  \providecommand\color[2][]{%
    \errmessage{(Inkscape) Color is used for the text in Inkscape, but the package 'color.sty' is not loaded}%
    \renewcommand\color[2][]{}%
  }%
  \providecommand\transparent[1]{%
    \errmessage{(Inkscape) Transparency is used (non-zero) for the text in Inkscape, but the package 'transparent.sty' is not loaded}%
    \renewcommand\transparent[1]{}%
  }%
  \providecommand\rotatebox[2]{#2}%
  \newcommand*\fsize{\dimexpr\f@size pt\relax}%
  \newcommand*\lineheight[1]{\fontsize{\fsize}{#1\fsize}\selectfont}%
  \ifx\svgwidth\undefined%
    \setlength{\unitlength}{231.86212975bp}%
    \ifx\svgscale\undefined%
      \relax%
    \else%
      \setlength{\unitlength}{\unitlength * \real{\svgscale}}%
    \fi%
  \else%
    \setlength{\unitlength}{\svgwidth}%
  \fi%
  \global\let\svgwidth\undefined%
  \global\let\svgscale\undefined%
  \makeatother%
  \begin{picture}(1,0.79495567)%
    \lineheight{1}%
    \setlength\tabcolsep{0pt}%
    \put(0,0){\includegraphics[width=\unitlength,page=1]{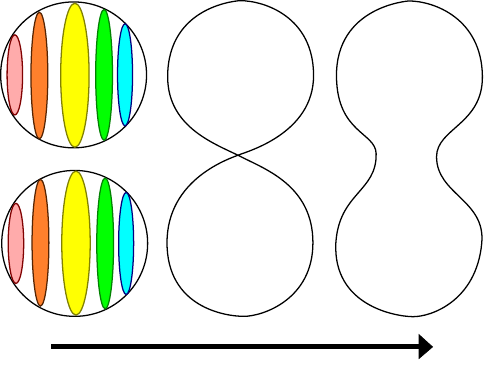}}%
    \put(0.3404778,0.01100929){\color[rgb]{0,0,0}\makebox(0,0)[lt]{\lineheight{0}\smash{\begin{tabular}[t]{l}decreasing $t$\end{tabular}}}}%
    \put(0,0){\includegraphics[width=\unitlength,page=2]{Fig16.pdf}}%
  \end{picture}%
\endgroup%

\caption[The third boundary fusion movie.]{The third boundary fusion movie.} 
\label{fig:boundaryfusion3}
\end{centering}\end{figure}

\end{move}

\begin{move}[Boundary fusion/compression movie 4]
Let $Z^4\cong B^4$ and $h:Z^4\to I$ so that $h^{-1}(1)=B^3\sqcup B^3$, $h^{-1}(0)=B^3$, and $h$ has one critical point. 
Say the critical value is $1/2$.

Let $\F_t$ be the movie of Figure~\ref{fig:boundaryfusion4}.

That is, let $\F_t$ be a standard fibration on each component of $B^3\sqcup B^3$ for $t>1/2$. One leaf of $\F_{1/2}$ is the critical point of $h$. For $t<1/2$, $\F_{t}$ is a standard fibration.

$\F_t$ has a singularity chart containing a death of two type \0 half-dots, of opposite index.

We call $\F_t$ the {\emph{fourth boundary fusion movie}} (or boundary fusion movie 4).

Let $\bar{h}:Z^4\to I$ be given by $\bar{h}(z)=1-h(z)$. Let $\G_t=\F_{1-t}$. With respect to $\bar{h}$, $\G_t$ has a singularity chart containing a birth of two type III half-dots, of opposite index.

We call $\G_t$ the {\emph{fourth boundary compression movie}} (or boundary compression movie 4).

\begin{figure}\begin{centering}
\begingroup%
  \makeatletter%
  \providecommand\color[2][]{%
    \errmessage{(Inkscape) Color is used for the text in Inkscape, but the package 'color.sty' is not loaded}%
    \renewcommand\color[2][]{}%
  }%
  \providecommand\transparent[1]{%
    \errmessage{(Inkscape) Transparency is used (non-zero) for the text in Inkscape, but the package 'transparent.sty' is not loaded}%
    \renewcommand\transparent[1]{}%
  }%
  \providecommand\rotatebox[2]{#2}%
  \newcommand*\fsize{\dimexpr\f@size pt\relax}%
  \newcommand*\lineheight[1]{\fontsize{\fsize}{#1\fsize}\selectfont}%
  \ifx\svgwidth\undefined%
    \setlength{\unitlength}{231.86329759bp}%
    \ifx\svgscale\undefined%
      \relax%
    \else%
      \setlength{\unitlength}{\unitlength * \real{\svgscale}}%
    \fi%
  \else%
    \setlength{\unitlength}{\svgwidth}%
  \fi%
  \global\let\svgwidth\undefined%
  \global\let\svgscale\undefined%
  \makeatother%
  \begin{picture}(1,0.79495409)%
    \lineheight{1}%
    \setlength\tabcolsep{0pt}%
    \put(0,0){\includegraphics[width=\unitlength,page=1]{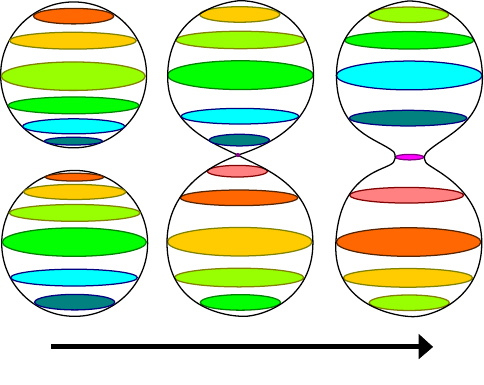}}%
    \put(0.3404906,0.01100919){\color[rgb]{0,0,0}\makebox(0,0)[lt]{\lineheight{0}\smash{\begin{tabular}[t]{l}decreasing $t$\end{tabular}}}}%
  \end{picture}%
\endgroup%

\caption[The fourth boundary fusion movie.]{The fourth boundary fusion movie.} 
\label{fig:boundaryfusion4}
\end{centering}\end{figure}

\end{move}

\begin{move}[Boundary birth/death movie 1]

Let $Z^4\cong B^4$ and $h:Z^4\to I$ so that $h^{-1}(1)=B^3$, $h^{-1}(0)=B^3\setminus\nu($point$)\cong S^2\times I$, and $h$ has one critical point. 
Say the critical value is $1/2$.

Let $\F_t$ be the movie of Figure~\ref{fig:boundarybirth}.

That is, let $\F_t$ be a standard fibration on each component of $B^3$ for $t>1/2$. 
For $t<1/2$, $\F_{t}$ agrees with a standard fibration in a neighborhood of $\boundary B^3$. $\F_t$ has two bowl singularities on the other boundary component of $h^{-1}(t)$.

$\F_t$ has a singularity chart containing a birth of two type \bone bowls, of opposite index.

We call $\F_t$ the {\emph{first boundary birth movie}} (or boundary birth movie 1).

Let $\bar{h}:Z^4\to I$ be given by $\bar{h}(z)=1-h(z)$. Let $\G_t=\F_{1-t}$. With respect to $\bar{h}$, $\G_t$ has a singularity chart containing a death of two type \btwo bowls, of opposite index.

We call $\G_t$ the {\emph{first boundary death movie}} (or boundary death movie 1).

\begin{figure}\begin{centering}
\begingroup%
  \makeatletter%
  \providecommand\color[2][]{%
    \errmessage{(Inkscape) Color is used for the text in Inkscape, but the package 'color.sty' is not loaded}%
    \renewcommand\color[2][]{}%
  }%
  \providecommand\transparent[1]{%
    \errmessage{(Inkscape) Transparency is used (non-zero) for the text in Inkscape, but the package 'transparent.sty' is not loaded}%
    \renewcommand\transparent[1]{}%
  }%
  \providecommand\rotatebox[2]{#2}%
  \newcommand*\fsize{\dimexpr\f@size pt\relax}%
  \newcommand*\lineheight[1]{\fontsize{\fsize}{#1\fsize}\selectfont}%
  \ifx\svgwidth\undefined%
    \setlength{\unitlength}{301.3872474bp}%
    \ifx\svgscale\undefined%
      \relax%
    \else%
      \setlength{\unitlength}{\unitlength * \real{\svgscale}}%
    \fi%
  \else%
    \setlength{\unitlength}{\svgwidth}%
  \fi%
  \global\let\svgwidth\undefined%
  \global\let\svgscale\undefined%
  \makeatother%
  \begin{picture}(1,0.33402592)%
    \lineheight{1}%
    \setlength\tabcolsep{0pt}%
    \put(0,0){\includegraphics[width=\unitlength,page=1]{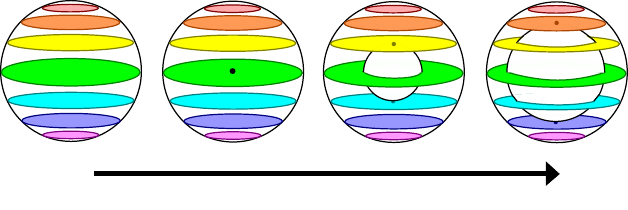}}%
    \put(0.40538415,0.00817182){\color[rgb]{0,0,0}\makebox(0,0)[lt]{\lineheight{0}\smash{\begin{tabular}[t]{l}decreasing $t$\end{tabular}}}}%
  \end{picture}%
\endgroup%

\caption[The first boundary birth movie.]{The first boundary birth movie.} 
\label{fig:boundarybirth}
\end{centering}\end{figure}

\end{move}

\begin{move}[Boundary birth/death movie 2]

Let $Z^4\cong B^4$ and $h:Z^4\to I$ so that $h^{-1}(1)=\emptyset$, $h^{-1}(0)=B^3$, and $h$ has one critical point. 
Say the critical value is $1/2$.

Let $\F_t$ be the movie of Figure~\ref{fig:boundarybirth2}.

That is, let $\F_t$ be a standard fibration on $B^3$ for $t<1/2$. 
$\F_{1/2}$ maps $h^{-1}(1/2)=\pt$ to some value of $\theta$. 

$\F_t$ has a singularity chart containing a birth of two type \0 half-dots, of opposite index.

We call $\F_t$ the {\emph{second boundary birth movie}} (or boundary birth movie 2).

Let $\bar{h}:Z^4\to I$ be given by $\bar{h}(z)=1-h(z)$. Let $\G_t=\F_{1-t}$. With respect to $\bar{h}$, $\G_t$ has a singularity chart containing a death of two type III half-dots, of opposite index.

We call $\G_t$ the {\emph{second boundary death movie}} (or boundary death movie 2).

\begin{figure}\begin{centering}
\begingroup%
  \makeatletter%
  \providecommand\color[2][]{%
    \errmessage{(Inkscape) Color is used for the text in Inkscape, but the package 'color.sty' is not loaded}%
    \renewcommand\color[2][]{}%
  }%
  \providecommand\transparent[1]{%
    \errmessage{(Inkscape) Transparency is used (non-zero) for the text in Inkscape, but the package 'transparent.sty' is not loaded}%
    \renewcommand\transparent[1]{}%
  }%
  \providecommand\rotatebox[2]{#2}%
  \newcommand*\fsize{\dimexpr\f@size pt\relax}%
  \newcommand*\lineheight[1]{\fontsize{\fsize}{#1\fsize}\selectfont}%
  \ifx\svgwidth\undefined%
    \setlength{\unitlength}{223.68468998bp}%
    \ifx\svgscale\undefined%
      \relax%
    \else%
      \setlength{\unitlength}{\unitlength * \real{\svgscale}}%
    \fi%
  \else%
    \setlength{\unitlength}{\svgwidth}%
  \fi%
  \global\let\svgwidth\undefined%
  \global\let\svgscale\undefined%
  \makeatother%
  \begin{picture}(1,0.44167613)%
    \lineheight{1}%
    \setlength\tabcolsep{0pt}%
    \put(0,0){\includegraphics[width=\unitlength,page=1]{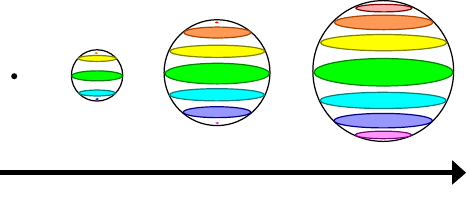}}%
    \put(0.34537134,0.00975632){\color[rgb]{0,0,0}\makebox(0,0)[lt]{\lineheight{0}\smash{\begin{tabular}[t]{l}decreasing $t$\end{tabular}}}}%
  \end{picture}%
\endgroup%

\caption[The second boundary birth movie.]{The second boundary birth movie.} 
\label{fig:boundarybirth2}
\end{centering}\end{figure}

\end{move}

\begin{move}
Let $h:Z^4\to I$ be Morse and $\F_t$ on $Z^4$ be a boundary fusion, compression, birth, or death movie. Let $x\in\boundary Z^4$ be the critical point of $h$. Let \[X^4=Z^4\cup_{\boundary Z^4\cap\nu(x)}\overline{Z^4}\] be a double of $Z^4$ along $x$, and $\tilde{h}:X^4\to I$ be given by $\tilde{h}=h\cup\overline{h}$.

Let $\G_t$ be a movie of singular fibrations on $X^4$ given by $\G_t=\F_t\cup\overline{\F_t}$ (that is, $\G_t$ agrees with $\F_t$ on $Z^4$, and agrees with $\overline{\F_t}$ on $\overline{Z^4}$).

We call $\G_t$ an {\emph{interior singularity movie}}. Let $\mathcal{C}_{\F}$ and $\mathcal{C}_{\G}$ be valid charts for $\F$ and $\G$, respectively. We give a table with the name of $\G_t$ and a description of part of $\mathcal{C}_{\F}$ and $\mathcal{C}_{\G}$.

\begin{center}
{\scriptsize
\begin{tabular}{ccccc}$\F_t$&$\G_t$&in $\mathcal{C}_{\F}$ (boundary)&in $\mathcal{C}_{\G}$ (interior)&index of c.p. of $h$
\\Boundary fusion 1&Index-2 birth&II birth&II birth&2
\\Boundary fusion  2&Index-1 death&\bone death&I death&2
\\Boundary fusion  3&First index-1 birth&I birth&I birth&3
\\Boundary fusion  4&Index-0 death&\0 death&\0 death&3
\\Boundary birth 1&Second Index-2 birth&\bone birth&I birth&2
\\Boundary birth 2&Index-0 birth&\0 birth&\0 birth&4
\\Boundary compression 1&Index-1 death&I death&I death&2
\\Boundary compression 2&Index-2 birth&\btwo birth&II birth&2
\\Boundary compression 3&First Index-2 death&II death&II death&1
\\Boundary compression 4&Index-3 birth&III birth&III birth&1
\\Boundary death 1&Second Index-1 death&\btwo death&II death&2
\\Boundary death 2&Index-3 death&III death&III death&0
\end{tabular}}
\end{center}

The names of these movies may seem unwieldy, but the logic is for the name of $\G_t$ to simply describe the critical points introduced in $\G_t$ as $t$ decreases from $1$ to $0$. An index-$k$ birth corresponds to an index-$(4-k)$ critical point of the ambient height function, while an index-$k$ death corresponds to an index-$(3-k)$ critical point of the ambient height function. (Again, there is some confusion as we generally work from top to bottom, but this viewpoint is convenient for this paper.)

Recall from the beginning of this section that a birth of interior type $n$ singularities corresponds to attachment of a relative $n$-handle. A death of interior type $n$ singularities corresponds to attachment of a relative $(n+1)$-handle.

\end{move}

In Figure~\ref{fig:boundarychart}, we give valid singularity charts for the boundary fusion, compression, birth, death movies and the interior singularity movies, omitting the persisting singularities.

\begin{figure}\begin{centering}
\includegraphics[width=.6\textwidth]{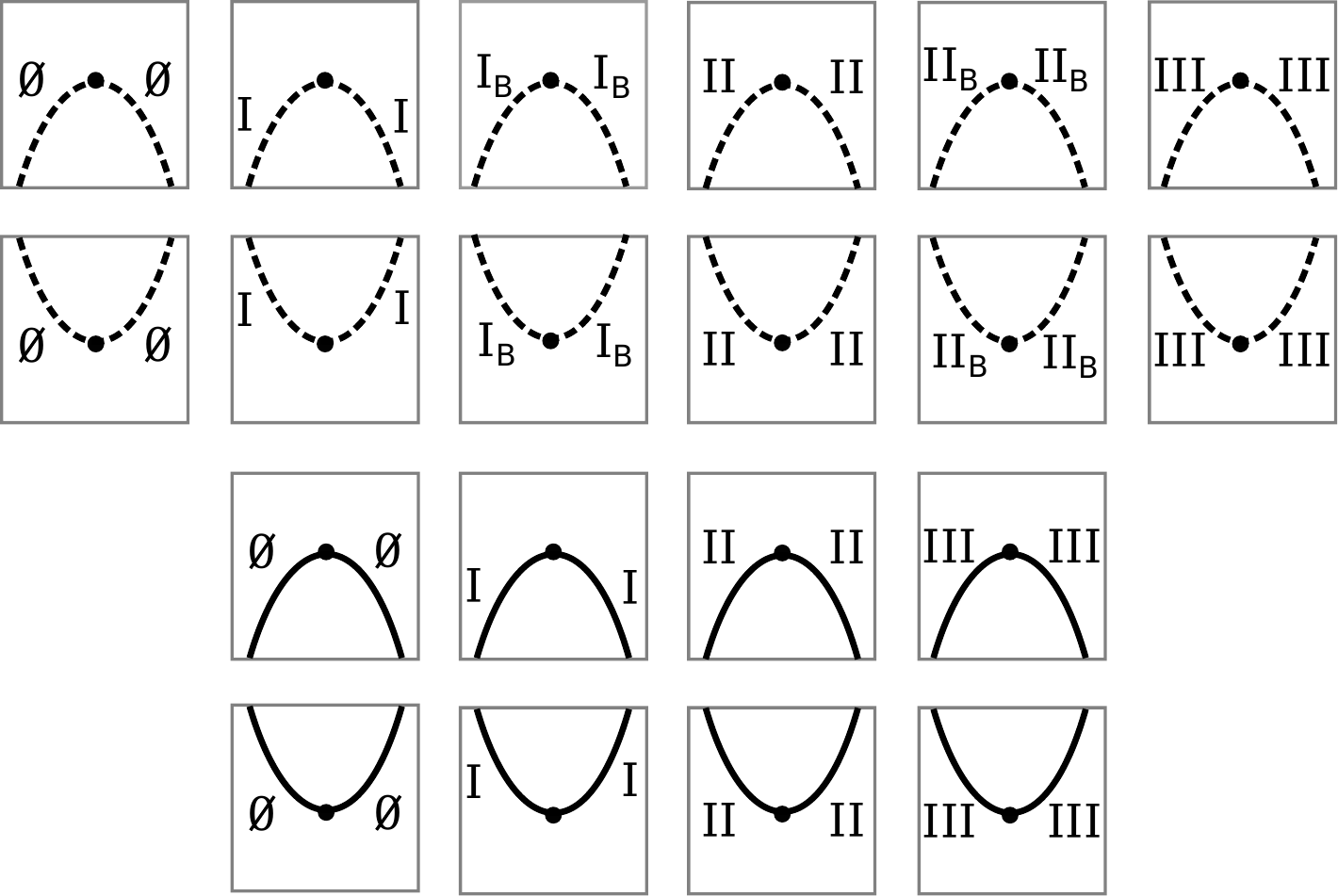}
\caption[Valid singularity charts for the boundary fusion, compression, birth, death and interior singularity movies.]{Valid singularity charts for the boundary fusion, compression, birth, death and interior singularity movies.}
\label{fig:boundarychart}
\end{centering}\end{figure}

\subsection{Positioning singularities in a fibration}
In this subsection, we define movies $\F_t$ on $M^3\times I$ which allow us to ``position'' the singularities of $\F_1$, in the sense of choosing the singular values of $\F_0$ (with some constraints). Imprecisely, both of these movies essentially ``flow'' some choice of leaves of $\F_1$ through a singularity. These movies exchange the heights of critical points of $\F_t$ as $t$ decreases.

\begin{move}[Positioning a dot or half-dot along an arc]\label{flowdot}
Let $\F_1$ be a singular fibration on $M^3$. Suppose $q$ is a cone point, and $p$ is a dot or half-dot of $\F_1$. Assume there exists an injective $\gamma:[0,1]\to M^3$ transverse to the leaves of $\F_1$ so that $\gamma(0)=q$, $\gamma(1)=p$.

If $p\in\mathring{M}^3$, then assume $\gamma(I)\subset\mathring{M}^3$ and the leaf component of $\F_1$ containing $\gamma(t)$ is a $2$-sphere for $0<t<1$.

If $p\in\boundary M^3$, then assume $\gamma(I)\cap\boundary M^3=p$ and the leaf component of $\F_1$ containing $\gamma(t)$ is a $2$-hemisphere for $0<t<1$.

Fix $\epsilon>0$ small, and let $\eta:[0,1]\to M^3$ be given by $\eta(s)=\gamma((1-\epsilon)s+\epsilon)$.

We define a movie $\F_t$ on $M^3\times I$ as follows:
\begin{itemize}
\item If $x\in M^3$ is not in a leaf component of $\F_1$ meeting $\eta([0,1])$, then $\F_t(x)=\F_1(x)$ for all $t$.
\item Let $B$ be the set on which we have not defined $\F_t$. If $p$ is a dot singularity (in the interior of $h^{-1}(t)$, write $\overline{B}=S^2\times I/(S^2\times 1\sim\pt)$ so $p=S^2\times 1/\sim$, and $\F_1(S^2\times s)=\F_1(\eta(s))$.
\item If $p$ is a half-dot,  write $B=D^2\times I/(D^2\times 1\sim\pt)$ so $p=D^2\times 1/\sim$, and $\F_1(D^2\times s)=\F_1(\gamma(s))$.
\item Define $\F_t$ by \[\F_t(S^2\times s)=\F_1(\eta(ts+(1-t)\epsilon s).\] 
\end{itemize}

As $t$ decreases, we are shrinking the range of $\F_t|_B$ while fixing $\F_t|_{M^3\setminus B}$. See Figure~\ref{fig:flowdot} for an illustration.

We say that $\mathcal{F}_t$ {\emph{positions $p$}}. We will use the phrase, ``position the dot or half-dot $p$ along $\gamma$," to mean, ``play $\mathcal{F}_t$ in a neighborhood of $\gamma$, where $p$ is an endpoint of $\gamma$."

\begin{figure}\begin{centering}
\scalebox{0.8}{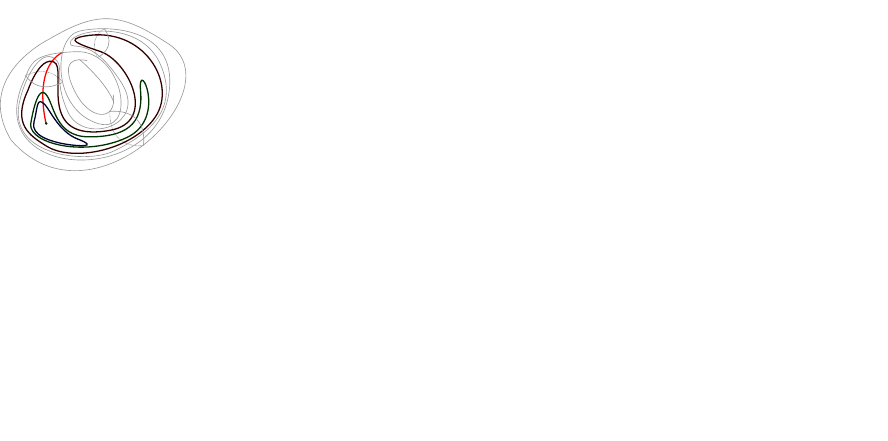}
\caption[Diagram of positioning a dot or half-dot.]{Diagram of positioning a dot or half-dot. We draw a neighborhood of $B$, which might be a solid torus or a ball. On the light gray leaf components (and all components not illustrated), $\F_t$ is fixed for all $t$. As $t$ decreases, the range of $\F_t(B)$ contracts.}
\label{fig:flowdot}
\end{centering}\end{figure}

\end{move}

Note that when positioning a dot or half-dot, the (components of) the level sets of $\F_t$ are identical for all $t$. In the beginning of this subsection, we said this movie would ``position'' a singularity. The singularities $p,q\in M^3$ of $\F_1$ are also singularities of $\F_0$; we have not moved the singular points within $M^3$. However, in $M^3\times 0$ the arc $\gamma$ is very short in the sense that $\F_0(\gamma(I))$ is very short in $S^1$. 

If $\epsilon>0$ is taken to be sufficiently small, then $\gamma((0,1))$ does not intersect any singular leaf of $\F_0$.

\begin{move}[Positioning a cone along an arc]\label{flowcone}
Consult Figure~\ref{fig:flowcone}. This image conveys the idea of this movie better than the ensuing text.

 Let $\F_1$ be a singular fibration on $M^3$. Suppose $q$ is a cone point of $\F_1$. Let $\gamma:[-1,1]\to \mathring{M}^3$ be injective so that $\gamma(0)=q$, $\F_1(\gamma(s))=\F_1(\gamma(-s))$ for all $s$, and $\gamma(I)$ is transverse to the leaves of $\F_1$ and does not meet any singularities except $q$. Let $B$ be a small ball containing $\gamma(I)$. Choose $\gamma$ so that near $q$, $\gamma(\pm\epsilon)$ meets the two-sheeted hyperboloid leaves of $\F_1|_B$ (rather than the one-sheeted leaves). 
 
 We define the movie $\F_t$ by:
 \begin{itemize}
 \item If $x$ is not in $\mathring B$, then $\F_t(x)=\F_1(x)$ for all $t$.
 \item The leaves of $\F_t|_B$ are (up to isotopy) hyperboloids, except for one cone.
 \item At time $t$, the cone of $\F_t|_{B}$ lies in $\F_t^{-1}\F_1(\gamma(1-t))$. The leaves of $\F_1|_{B}$ meeting $\gamma(s)$ for $|s|>1-t$ are still two-sheeted hyperboloid leaves of $\F_t|_B$, while the leaves of $\F_1$ meeting $\gamma(s)$ for $|s|<1-t$ become one-sheeted hyperboloid leaves of $\F_t|_B$.
 \end{itemize}
 
 The purpose of this movie is that in $\F_0$, the cone singularity $q$ now lies in $\F_0^{-1}(\F_1(\gamma(\pm 1)))$.

\begin{figure}\begin{centering}
\scalebox{0.9}{
\begingroup%
  \makeatletter%
  \providecommand\color[2][]{%
    \errmessage{(Inkscape) Color is used for the text in Inkscape, but the package 'color.sty' is not loaded}%
    \renewcommand\color[2][]{}%
  }%
  \providecommand\transparent[1]{%
    \errmessage{(Inkscape) Transparency is used (non-zero) for the text in Inkscape, but the package 'transparent.sty' is not loaded}%
    \renewcommand\transparent[1]{}%
  }%
  \providecommand\rotatebox[2]{#2}%
  \newcommand*\fsize{\dimexpr\f@size pt\relax}%
  \newcommand*\lineheight[1]{\fontsize{\fsize}{#1\fsize}\selectfont}%
  \ifx\svgwidth\undefined%
    \setlength{\unitlength}{370.87045264bp}%
    \ifx\svgscale\undefined%
      \relax%
    \else%
      \setlength{\unitlength}{\unitlength * \real{\svgscale}}%
    \fi%
  \else%
    \setlength{\unitlength}{\svgwidth}%
  \fi%
  \global\let\svgwidth\undefined%
  \global\let\svgscale\undefined%
  \makeatother%
  \begin{picture}(1,0.56337137)%
    \lineheight{1}%
    \setlength\tabcolsep{0pt}%
    \put(0,0){\includegraphics[width=\unitlength,page=1]{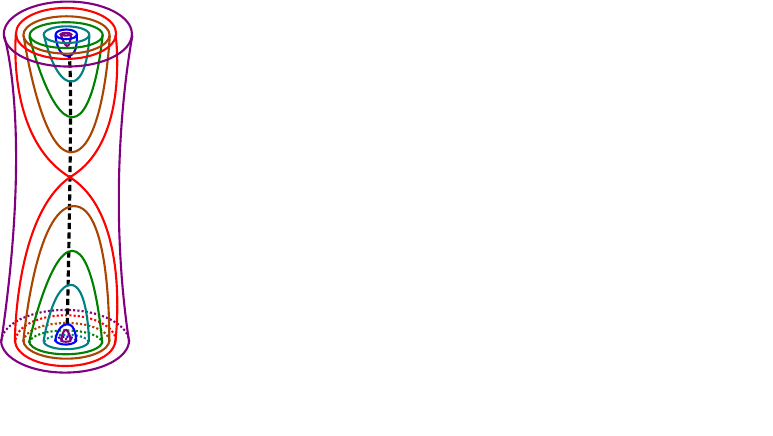}}%
    \put(0.11290685,0.33003697){\color[rgb]{0,0,0}\makebox(0,0)[lt]{\lineheight{0}\smash{\begin{tabular}[t]{l}$\gamma$\end{tabular}}}}%
    \put(0,0){\includegraphics[width=\unitlength,page=2]{Fig22.pdf}}%
    \put(0.39952649,0.02223116){\color[rgb]{0,0,0}\makebox(0,0)[lt]{\lineheight{0}\smash{\begin{tabular}[t]{l}decreasing $t$\end{tabular}}}}%
  \end{picture}%
\endgroup%
}
\caption[Diagram of positioning a cone along an arc $\gamma$.]{Diagram of positioning a cone along an arc $\gamma$. We draw a neighborhoorhood of $\gamma$. As $t$ decreases, the value of $\F_t($cone point$)$ changes from $\F_1(\gamma(0))$ to $\F_1(\gamma(\pm1))$.}
\label{fig:flowcone}
\end{centering}\end{figure}

As described, a chart for $\F_t$ has a cone of type I. We will now reparametrize $\F_t$ so that the cone is type II. Let $B$ be a small neighborhood of $\gamma([-1,1])$ so that $q$ is the only singularity of $\F_1$ contained in $B$. Let $f:S^1\to[0,m]$ be a narrow bump function on $\F_1(q)$. If $q$ is index-$1$, let $f:=-f$.

Let $g_t:S^1\to S^1$ be a smooth family of automorphisms with $g_1=\id$ and $d g_t/dt=-f$. Obtain a movie of singular fibrations $\tilde{\G}_t$ on $B\times I$ defined by $\tilde{G}_t(x)=g_t\circ\F_t(x)$ for $x\in h^{-1}(t)$. If $m$ is taken to be sufficiently large, then the repositioned cone is type II in $\tilde{\G}_t$. Now isotope $\tilde{\G}_t$ near $\boundary B$ so that $\tilde{\G}_t|_{\nu(\boundary B)}$ agrees with $\F_t$. Let $\G_t$ be a movie on $M^3\times I$ given by $\G_t(x)=\begin{cases}\tilde{\G}_t(x)\mid x\in B\\\F_t(x)\mid x\not\in B\end{cases}$.  The movies $\G_t$ and $\F_t$ agree near every singularity of $\F_t$ except $q$, so their charts are identical except for the arc of $q$, which has opposite slope (and hence opposite type) in the other chart. See Figure~\ref{fig:positionreparam}.

\begin{figure}\begin{centering}
\scalebox{0.9}{
\begingroup%
  \makeatletter%
  \providecommand\color[2][]{%
    \errmessage{(Inkscape) Color is used for the text in Inkscape, but the package 'color.sty' is not loaded}%
    \renewcommand\color[2][]{}%
  }%
  \providecommand\transparent[1]{%
    \errmessage{(Inkscape) Transparency is used (non-zero) for the text in Inkscape, but the package 'transparent.sty' is not loaded}%
    \renewcommand\transparent[1]{}%
  }%
  \providecommand\rotatebox[2]{#2}%
  \newcommand*\fsize{\dimexpr\f@size pt\relax}%
  \newcommand*\lineheight[1]{\fontsize{\fsize}{#1\fsize}\selectfont}%
  \ifx\svgwidth\undefined%
    \setlength{\unitlength}{371.31609098bp}%
    \ifx\svgscale\undefined%
      \relax%
    \else%
      \setlength{\unitlength}{\unitlength * \real{\svgscale}}%
    \fi%
  \else%
    \setlength{\unitlength}{\svgwidth}%
  \fi%
  \global\let\svgwidth\undefined%
  \global\let\svgscale\undefined%
  \makeatother%
  \begin{picture}(1,0.56322926)%
    \lineheight{1}%
    \setlength\tabcolsep{0pt}%
    \put(0,0){\includegraphics[width=\unitlength,page=1]{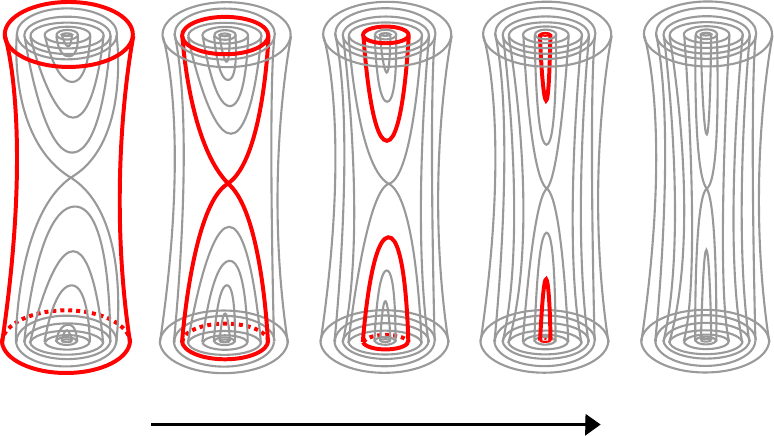}}%
    \put(0.40024727,0.02220435){\color[rgb]{0,0,0}\makebox(0,0)[lt]{\lineheight{0}\smash{\begin{tabular}[t]{l}decreasing $t$\end{tabular}}}}%
  \end{picture}%
\endgroup%
}
\caption[Ensuring a positioned cone is type II.]{When positioning a cone along an arc $\gamma$, suppose $\gamma$ meets no other singular leaf. By reparametrizing $\F_t$, we find a movie $\G_t$ in which the positioned cone is type II. In bold: $\G_t^{-1}(\theta_0)$ near the positioned cone, as $t$ decreases.}
\label{fig:positionreparam}
\end{centering}

\end{figure}

We say that $\mathcal{F}_t$ and $\mathcal{G}_t$ {\emph{position $q$ along an arc}}. We will use the phrase, ``position the cone or half-cone $q$ along $\gamma$," to mean, ``play $\mathcal{F}_t$ or $\mathcal{G}_t$ in a neighborhood of $\gamma$, where $q$ is on $\gamma$." If $q$ already has a type, then implicitly we choose $\mathcal{F}_t$ or $\mathcal{G}_t$ to maintain the type of $q$. 
\end{move}

\begin{move}[Positioning a cone along a disk]\label{flowcone}
This movie is the inverse of positioning a cone along an arc. However, we wish to write this movie out explicitly.

Let $\F_1$ be a singular fibration on $M^3$. Suppose $q$ is a cone point of $\F_1$. Let $D$ be an improperly embedded disk in $\mathring{M}^3$ intersecting $q$, with $\boundary D$ contained in one leaf of $\F_1$. Parameterize $D$ so $D=S^1\times I/(S^1\times 1\sim\pt)$, where $S^1\times 0=\boundary D$ and $S^1\times 1/\sim=q$. Assume that each $S^1\times r$ lives in a single leaf of $\F_1$, which intersects $D$ transversely for $0<r<1$. Let $B$ be an open neighborhood of $D$.

 We define the movie $\F_t$ by:
 \begin{itemize}
 \item If $x$ is not in $B$, then $\F_t(x)=\F_1(x)$ for all $t$.
 \item The leaves of $\F_t|_B$ are (up to isotopy) hyperboloids, except one cone.
 \item At time $t$, the cone of $\F_t|_B$ lies in $\F_t^{-1}\F_1(S^1\times t)$. The leaves of $\F_1|_B$ meeting $D$ in $S^1\times r$ are still one-sheeted hyperboloid leafs of $\F_t|_B$ when $r<t$. When $r>1$, they become two-sheeted hyperboloid leafs of $\F_t|_B$.
 \end{itemize}

 The purpose of this movie is that in $\F_0$, the cone singularity $q$ now lies in $\F_0^{-1}(\F_1(\boundary D))$.

As described, the repositioned cone is type II in $\F_t$. 

Similarly to the ``positioning a cone along an arc'' movie, we may reparameterize this movie to find movie $\mathcal{G}_t$, in which the repositioned cone is type I (by first reparametrizing in a small neighborhood of $D\times I$ and then extending to $M^3\times I$).

We say that $\mathcal{F}_t$ and $\mathcal{G}_t$ {\emph{position $q$ along a disk}}. We will use the phrase, ``position the cone or half-cone $q$ along $D$," to mean, ``play $\mathcal{F}_t$ or $\mathcal{G}_t$ in a neighborhood of $D$, where $q$ is the center of disk $D$." If $q$ already has a type, then implicitly we choose $\mathcal{F}_t$ or $\mathcal{G}_t$ to maintain the type of $q$.

\end{move}

In Figure~\ref{fig:positionchart}, we give singularity charts for the singularity positioning movies.
\begin{figure}\begin{centering}
\begingroup%
  \makeatletter%
  \providecommand\color[2][]{%
    \errmessage{(Inkscape) Color is used for the text in Inkscape, but the package 'color.sty' is not loaded}%
    \renewcommand\color[2][]{}%
  }%
  \providecommand\transparent[1]{%
    \errmessage{(Inkscape) Transparency is used (non-zero) for the text in Inkscape, but the package 'transparent.sty' is not loaded}%
    \renewcommand\transparent[1]{}%
  }%
  \providecommand\rotatebox[2]{#2}%
  \newcommand*\fsize{\dimexpr\f@size pt\relax}%
  \newcommand*\lineheight[1]{\fontsize{\fsize}{#1\fsize}\selectfont}%
  \ifx\svgwidth\undefined%
    \setlength{\unitlength}{228.97782934bp}%
    \ifx\svgscale\undefined%
      \relax%
    \else%
      \setlength{\unitlength}{\unitlength * \real{\svgscale}}%
    \fi%
  \else%
    \setlength{\unitlength}{\svgwidth}%
  \fi%
  \global\let\svgwidth\undefined%
  \global\let\svgscale\undefined%
  \makeatother%
  \begin{picture}(1,0.74751284)%
    \lineheight{1}%
    \setlength\tabcolsep{0pt}%
    \put(0,0){\includegraphics[width=\unitlength,page=1]{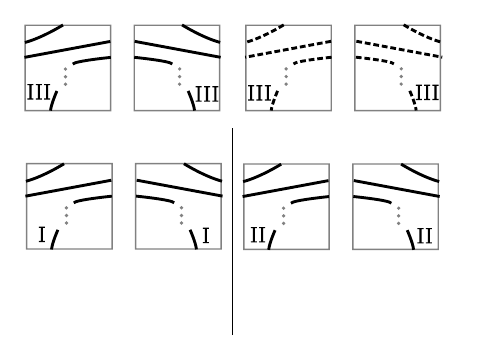}}%
    \put(0.20808937,0.43229036){\color[rgb]{0,0,0}\makebox(0,0)[t]{\lineheight{0}\smash{\begin{tabular}[t]{c}Position cone along arc\end{tabular}}}}%
    \put(0.78008346,0.42979937){\color[rgb]{0,0,0}\makebox(0,0)[t]{\lineheight{0}\smash{\begin{tabular}[t]{c}Position cone along disk\end{tabular}}}}%
    \put(0.27446746,0.71975575){\color[rgb]{0,0,0}\makebox(0,0)[lt]{\lineheight{0}\smash{\begin{tabular}[t]{l}Position dot or half-dot\end{tabular}}}}%
    \put(0,0){\includegraphics[width=\unitlength,page=2]{Fig24.pdf}}%
  \end{picture}%
\endgroup%

\caption[Valid singularity charts for the positioning movies.]{Valid singularity charts for the positioning movies. Each chart contains one arc (we omit the arcs of the singularities we are not positioning).} 
\label{fig:positionchart}
\end{centering}\end{figure}

\subsection{Cancelling interior and boundary singularities}

\begin{move}[Cancelling a cone with a half-cone]
See Figure~\ref{fig:intboundarycancel} (top).

Let $\F_1$ be a singular fibration on $B^3$ which has two half-cone singularities $p_1,p_2$ on $\boundary B^3$ and one cone $q$ in the interior of $B^3$. (Necessarily, the half-cones are both of index $i$ while the cone is index $3-i$.) 

Let $D=S^1\times I/(S^1\times1\sim\pt)$ be a disk improperly embedded in $B^3$ so that $D\cap p_2=\emptyset$, $D\cap\boundary B^3=p_1\in S^1\times 0$, $q=S^1\times 1$, $D$ is transverse to the leaves of $\F$ away from $p_1$ and $q$ and $S^1\times s$ is contained in a single leaf for each $s$. We proceed as if positioning $q$ along $D$. That is, obtain $\F_t$ by perturbing $\F_1$ in a neighborhood of $D$ so that the cone lies in $\F_t^{-1}(\F_t(S^1\times t))$. As $t$ grows small, take the cone point to be very close to $p_1$. At $t=0$, the cone meets $p_1$ and becomes a half-cone of index $3-i$.

As written, the interior cone of $\F_t$ is type II while the cancelled boundary half-cone is type I. Now we describe a movie $\G_t$ with the same level sets, in which the types are reversed.

Let $\G_1=\F_1$. 
Let $D'=[0,\pi]\times I/([0,\pi]\times1\sim\pt)$ be a disk improperly embedded in $B^3$ so that $D'\cap p_2=\emptyset$, $p_1= [0,\pi]\times 1$, $q\in[0,\pi]\times 0$, $D'\cap\boundary B^3=\{0,\pi\}\times I$, $D'$ is transverse to the leaves of $\F$ away from $p_1$ and $q$ and $[0,\pi]\times s$ is contained in a single leaf for each $s$. We proceed as if positioning $p_1$ along $D'$. That is, obtain $\G_t$ by perturbing $\G_1$ in a neighborhood of $D'$ so that the half-cone lies in $\G_t^{-1}(\F_t([0,\pi]\times t))$. As $t$ grows small, take the cone point to be very close to $p_1$. At $t=0$, the cone meets the boundary and becomes a half-cone of index $3-i$.

We say that $\F_t$ and $\G_t$ are movies of a {\emph{cone, half-cone death}}. Turning $\F_t$ or $\G_t$ upside-down yields a movie of a {\emph{cone, half-cone birth}}.

\end{move}

\begin{figure}\begin{centering}
\scalebox{0.85}{
\begingroup%
  \makeatletter%
  \providecommand\color[2][]{%
    \errmessage{(Inkscape) Color is used for the text in Inkscape, but the package 'color.sty' is not loaded}%
    \renewcommand\color[2][]{}%
  }%
  \providecommand\transparent[1]{%
    \errmessage{(Inkscape) Transparency is used (non-zero) for the text in Inkscape, but the package 'transparent.sty' is not loaded}%
    \renewcommand\transparent[1]{}%
  }%
  \providecommand\rotatebox[2]{#2}%
  \newcommand*\fsize{\dimexpr\f@size pt\relax}%
  \newcommand*\lineheight[1]{\fontsize{\fsize}{#1\fsize}\selectfont}%
  \ifx\svgwidth\undefined%
    \setlength{\unitlength}{382.15431994bp}%
    \ifx\svgscale\undefined%
      \relax%
    \else%
      \setlength{\unitlength}{\unitlength * \real{\svgscale}}%
    \fi%
  \else%
    \setlength{\unitlength}{\svgwidth}%
  \fi%
  \global\let\svgwidth\undefined%
  \global\let\svgscale\undefined%
  \makeatother%
  \begin{picture}(1,0.64633749)%
    \lineheight{1}%
    \setlength\tabcolsep{0pt}%
    \put(0,0){\includegraphics[width=\unitlength,page=1]{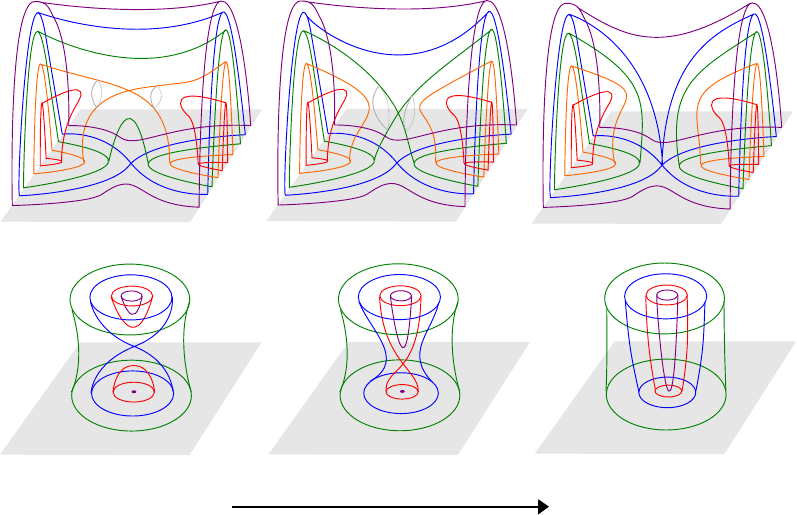}}%
    \put(0.43182736,0.01999474){\color[rgb]{0,0,0}\makebox(0,0)[lt]{\lineheight{0}\smash{\begin{tabular}[t]{l}decreasing $t$\end{tabular}}}}%
  \end{picture}%
\endgroup%
}
\caption[Schema of the cone, half-cone death and cone, half-dot death.]{Top: a schematic of the cone, half-cone death. Bottom: a cone, half-dot death.}
\label{fig:intboundarycancel}
\end{centering}\end{figure}

\begin{move}[Cancelling a cone with a half-dot]

See Figure~\ref{fig:intboundarycancel} (bottom).

 Let $\F_1$ be a singular fibration on $M^3$. Suppose $q$ is a cone point of $\F_1$. Let $\gamma:[0,1]\to M^3$ so that $\gamma(0)=q$, $\gamma(1)$ is a half-dot $p$, and the interior of $\gamma$ is disjoint from $\boundary M^3$. Assume $\gamma((0,1))$ is transverse to the leaves of $\F_1$ and does not meet any singular leaf components of $\F_1$. 
Assume that near $q$, $\gamma(\epsilon)$ meets the two-sheeted hyperboloid leaves of $\F_1|_{\nu(q)}$ (rather than the one-sheeted leaves). 
 
 From $t=1$ to $1/2$, position the half-dot along $\gamma$ so that 
 $\gamma((0,1))$ does not meet any singular leaves of $\F_{1/2}$. Extend $\gamma$ to an interval $[-\delta,1]$ so that $\F_{1/2}(\gamma(s))=\F_{1/2}(\gamma(-s))$, as in the ``positioning a cone along an arc'' movie.

 Let $B$ be a small neighborhood of $\gamma((-\delta,1))$.
 
 We define the movie $\F_t$ by:
 \begin{itemize}
 \item If $x$ is not in $\mathring B$, then $\F_t(x)=\F_1(x)$ for all $t$.
 \item The leaves of $\F_t|_B$ are (up to isotopy) hyperboloids, except one cone.
 \item At time $t$, the cone of $\F_t\mid B$ lies in $\F_t^{-1}\F_1(\gamma(1-t))$. The leaves of $\F_1\mid B$ meeting $\gamma(s)$ for $|s|<t$ are still two-sheeted hyperboloid leaves of $\F_t|_B$, while the leaves of $\F_1$ meeting $\gamma(s)$ for $|s|>t>0$ become one-sheeted hyperboloid leaves of $\F_t|_B$.
 \end{itemize}
 
 As $t$ grows small, take the cone to be very close to $p$. At $t=0$, the cone meets $p$ and becomes a bowl singularity.

As described, the cone is type I and the half-dot is type \0. Because $\gamma((-\delta,1))$ does not meet any singular leaves of $\F_{1/2}$, we can reparametrize $\F_t$ to find a movie $\G_t$ with the same level sets in which the positioned cone is type II (and the half-dot is type III) as follows:
 \begin{itemize}
 \item Let $f:S^1\to[0,m]$ be a narrow bump function on $\F_{1/2}(\gamma([-\delta,1]))$. If $q$ is index-$1$, let $f:=-f$.
\item Let $g_t:S^1\to S^1$ be a smooth family of automorphisms with $g_1=\id$ and $d g_t/dt=-f$.
\item Let $\G_t(x)=g_t\circ\F_t(x)$ for $x\in h^{-1}(t)$.
\end{itemize}
If $m$ is taken to be sufficiently large, then the repositioned cone is type II in $\G_t$, and the half-dot is type III.

We say that $\F_t$ is a movie of a {\emph{cone, half-dot death}}. Turning $\F_t$ upside-down yields a movie of a {\emph{cone, half-dot birth}}.

\end{move}

In Figure~\ref{fig:intbdychart}, we give a valid singularity chart for the interior and boundary singularity cancellation movies.
\begin{figure}\begin{centering}
\scalebox{0.8}{
\begingroup%
  \makeatletter%
  \providecommand\color[2][]{%
    \errmessage{(Inkscape) Color is used for the text in Inkscape, but the package 'color.sty' is not loaded}%
    \renewcommand\color[2][]{}%
  }%
  \providecommand\transparent[1]{%
    \errmessage{(Inkscape) Transparency is used (non-zero) for the text in Inkscape, but the package 'transparent.sty' is not loaded}%
    \renewcommand\transparent[1]{}%
  }%
  \providecommand\rotatebox[2]{#2}%
  \newcommand*\fsize{\dimexpr\f@size pt\relax}%
  \newcommand*\lineheight[1]{\fontsize{\fsize}{#1\fsize}\selectfont}%
  \ifx\svgwidth\undefined%
    \setlength{\unitlength}{360.38132074bp}%
    \ifx\svgscale\undefined%
      \relax%
    \else%
      \setlength{\unitlength}{\unitlength * \real{\svgscale}}%
    \fi%
  \else%
    \setlength{\unitlength}{\svgwidth}%
  \fi%
  \global\let\svgwidth\undefined%
  \global\let\svgscale\undefined%
  \makeatother%
  \begin{picture}(1,0.29725418)%
    \lineheight{1}%
    \setlength\tabcolsep{0pt}%
    \put(0,0){\includegraphics[width=\unitlength,page=1]{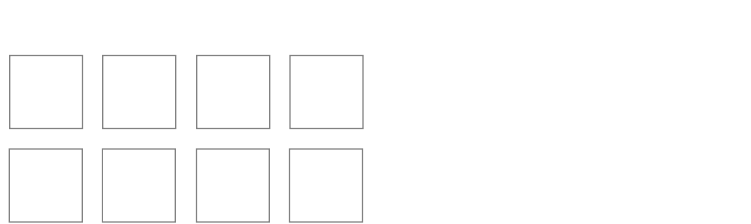}}%
    \put(0.12357666,0.27235126){\color[rgb]{0,0,0}\makebox(0,0)[lt]{\lineheight{0}\smash{\begin{tabular}[t]{l}cone, half-cone\end{tabular}}}}%
    \put(0,0){\includegraphics[width=\unitlength,page=2]{Fig26.pdf}}%
    \put(0.63227565,0.27158393){\color[rgb]{0,0,0}\makebox(0,0)[lt]{\lineheight{0}\smash{\begin{tabular}[t]{l}cone, half-dot\end{tabular}}}}%
    \put(0,0){\includegraphics[width=\unitlength,page=3]{Fig26.pdf}}%
    \put(0.01315727,0.17102218){\color[rgb]{0,0,0}\makebox(0,0)[lt]{\lineheight{0}\smash{\begin{tabular}[t]{l}II\end{tabular}}}}%
    \put(0,0){\includegraphics[width=\unitlength,page=4]{Fig26.pdf}}%
    \put(0.01892499,0.03165347){\color[rgb]{0,0,0}\makebox(0,0)[lt]{\lineheight{0}\smash{\begin{tabular}[t]{l}I\end{tabular}}}}%
    \put(0,0){\includegraphics[width=\unitlength,page=5]{Fig26.pdf}}%
  \end{picture}%
\endgroup%
}
\caption[Valid singularity charts for the cone, half-cone death/birth and cone, half-dot death/birth.]{Top row: valid singularity charts for the cone, half-cone and cone, half-dot deaths. Bottom row: valid singularity charts for the cone, half-cone and cone, half-dot births. We omit the arc of the half-cone or half-dot which is not cancelled or born during the movie.}
\label{fig:intbdychart}
\end{centering}\end{figure}

\subsection{Movies on $B^3\times I\setminus\nu($saddle, minimum, or maximum disk$)$}

Now we consider slightly more interesting $4$-manifolds. These movies are defined on $Z^4\cong (B^3\times I)\setminus($surface with a single singular cross-section$)$. Most of the interesting information of these singular movies is determined by the topology of $Z^4$, while (very informally) $\F_t$ does not change much with $t$.

\begin{move}[Saddle movie]\label{saddlemove}
\begin{figure}\begin{centering}
\scalebox{0.9}{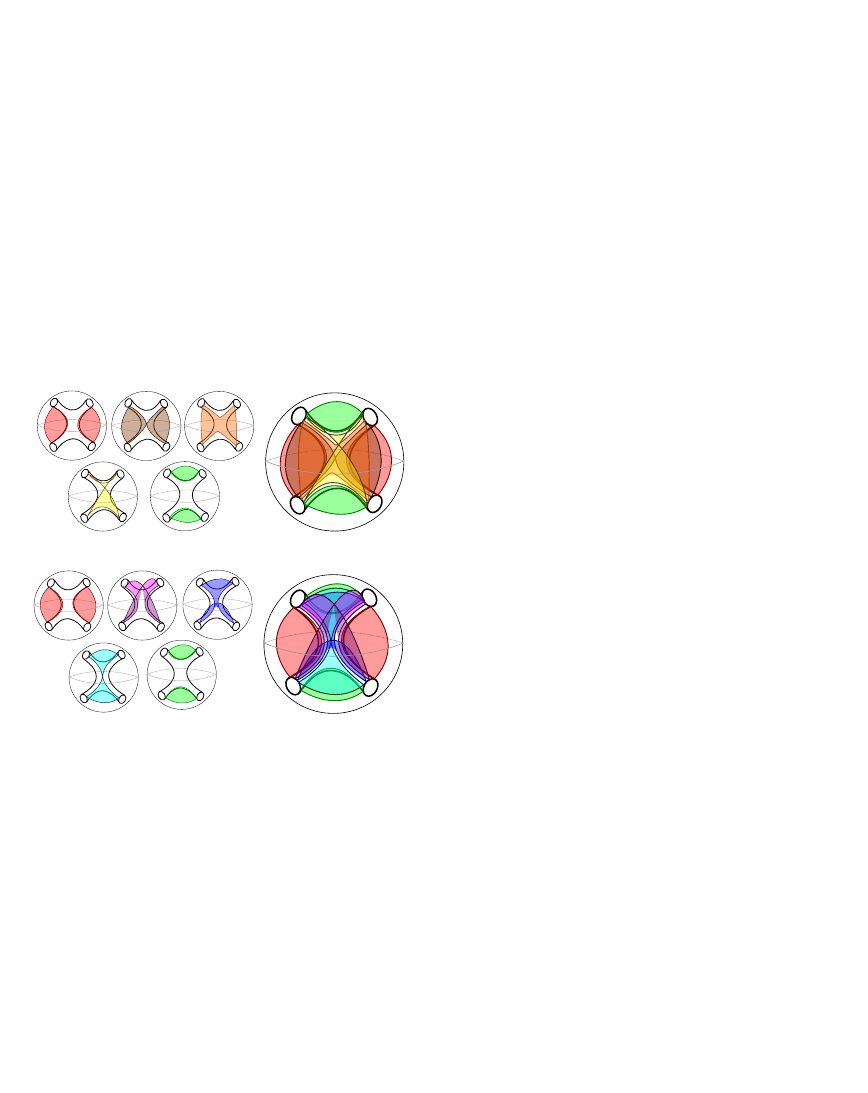}
\caption[Diagram of a saddle movie.]{Diagram of a saddle movie. This is a movie of singular fibrations on $(B^3\times I)\setminus\nu($saddle$)$.}
\label{fig:saddlemove}
\end{centering}\end{figure}

Consult Figure~\ref{fig:saddlemove}.
Let $h:B^3\times I\to I$ be projection onto the second factor. Let $R\subset B^3\times I$ be a saddle. That is, $R$ is a properly embedded disk with $R\cap(B^3\times 0)\cong R\cap(B^3\times 1)\cong\{2$ arcs$\}$ and $h|_R$ is Morse with single index-$1$ critical point and no index-$0$ or -$2$ critical points.

Let $Z^4=(B^3\times I)\setminus\nu(R)$ and restrict $h$ to $Z^4$. There are two values of $t$ for which $h^{-1}(t)$ is a singular $3$-manifold; say these values are $t=3/4$ and $t=1/4$. We define a movie of singular fibrations $\F_t:h^{-1}(t)\to S^1$ on $Z^4$ as follows:
\begin{itemize}
\item Figure~\ref{fig:saddlemove} (top left): Let $\F_1$ be as indicated. There are two half-cone singularities of $\F_1$ in $\boundary h^{-1}(1)$, which are both far from $\boundary\nu(R)$. The nonsingular leaves of $\F_1$ are either: a single disk, a disjoint union of two disks, or an annulus. There is one cone singularity of $\F_1$ in the interior of $h^{-1}(1)$ corresponding to the compression of the annular leaves. There are no dot singularities.
\item  Figure~\ref{fig:saddlemove} (top right): At $t=3/4$, $h^{-1}(3/4)$ is a singular $3$-manifold. 
\item Figure~\ref{fig:saddlemove} (middle left): For $1/4<t<3/4$, $h^{-1}(t)$ is a solid of genus-$3$. $\F_t$ has no singularities in the interior of $h^{-1}(t)$. (The cone leaf at $t>3/4$ becomes a half-cone meeting $\boundary\nu(R)$.) There are four half-cone singularities of $\F_t$ in $\boundary h^{-1}(t)$, two of which lie on $\boundary\nu(R)$.
\item Figure~\ref{fig:saddlemove} (middle right): Again, $h^{-1}(1/4)$ is a singular $3$-manifold. Note $h^{-1}(1/4)\cong h^{-1}(3/4)$. Under a rotation, the singular fibration $\F_{1/4}$ is identical to the singular fibration $\F_{3/4}$.
\item Figure~\ref{fig:saddlemove} (bottom left): As $t$ decreases to $0$, we take $\F_t=\F_{1-t}$ under a rotation so that the movies $\F_t$ ($t$ decreasing from $1/4$ to $0$) and $\F_t$ ($t$ increasing from $3/4$ to $1$) are identical (under a rotation). This results in a singular fibration $\F_0$ with two half-cones on $\boundary h^{-1}(0)$ and one cone in the interior of $h^{-1}(0)$.
\end{itemize}

We call $\mathcal{F}_t$ a {\emph{saddle movie}}.
\end{move}

Note that the cone singularities in $\F_0$ and $\F_1$ are contained in different $\F_t^{-1}(\theta)$'s (different values of $\theta$).

\begin{move}[Minimum/maximum movie]\label{minmove}
Consult Figure~\ref{fig:minmove}.
Let $h:B^3\times I\to I$ be projection onto the second factor. Let $R\subset B^3\times I$ be a disk with unknotted boundary in $B^3\times 1$, so that ${h|_R}$ is Morse with one critical point (necessarily of index $0$). Note $R\cap(B^3\times 0)=\emptyset$. 

Let $Z^4=(B^3\times I)\setminus\nu(R)$ and restrict $h$ to $Z^4$. There are two values of $t$ for which $h^{-1}(t)$ is a singular $3$-manifold; say these values are $t=5/6$ and $t=1/6$. We define a movie of singular fibrations $\F_t:h^{-1}(t)\to S^1$ on $Z^4$ as follows:
\begin{itemize}
\item Figure~\ref{fig:minmove} (top left): Let $\F_1$ be as indicated. There are two dot singularities in the foliation $\F_1$ induces on $\boundary h^{-1}(1)$, which both lie in $\boundary B^3$. The nonsingular leaves of $\F_1$ are disks and annuli. $\F_1$ has no singularities in the interior of $h^{-1}(1)$.
\item  Figure~\ref{fig:minmove} (top right): As $t$ decreases to $5/6$, play the fourth boundary compression movie.
\item Figure~\ref{fig:minmove} (left, second row): There are two half-dot singularities of $\F_{2/3}$ on the now spherical boundary component.
\item Figure~\ref{fig:minmove} (right, second row): As $t$ decreases to $1/2$, Play an interior $0$-,$1$- stabilization and a $2$-,$3$- stabilization near the half-dots. (Do the $0$-,$1$- stabilization near the index-$0$ half-dot.)
\item Figure~\ref{fig:minmove} (left, third row): As $t$ decreases to $1/3$, play a cone, half-dot death movie near each half-dot. There are two bowl singularities in $\F_{1/3}$.
\item Figure~\ref{fig:minmove} (right, third row): As $t$ decreases, play the first boundary death movie, with $h^{-1}(1/6)$ being singular.
\item Figure~\ref{fig:minmove} (bottom left): For $1/6>t>0$, we have $h^{-1}(t)= B^3$. There are two dot singularities of $\F_t$ in the interior of $B^3$.
\end{itemize}

We call $\F_t$ a {\emph{minimum movie}}
\begin{figure}\begin{centering}
\scalebox{0.9}{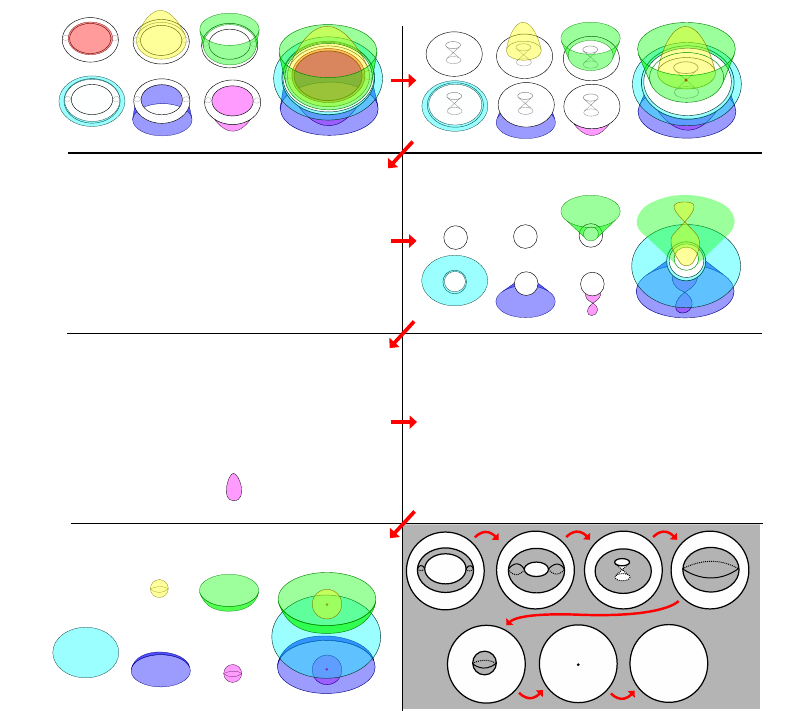}
\caption[Diagram of a minimum movie.]{Diagram of a minimum movie. This is a movie of singular fibrations on $(B^3\times I)\setminus\nu($disk with one minimum$)$.}
\label{fig:minmove}
\end{centering}\end{figure}

Let $\bar{h}:Z^4\to I$ be given by $\bar{h}(z)=1-h(z)$. Let $\G_t=\F_{1-t}$. Now $Z^4=B^3\times I\setminus\nu($disk with one maximum with respect to $\bar{h})$. With respect to $\bar{h}$, we call $\G_t$ a {\emph{maximum movie}}.

\end{move}

In Figure~\ref{fig:complementchart}, we give a valid singularity chart for the saddle, minimum, and maximum movies.
\begin{figure}\begin{centering}
\begingroup%
  \makeatletter%
  \providecommand\color[2][]{%
    \errmessage{(Inkscape) Color is used for the text in Inkscape, but the package 'color.sty' is not loaded}%
    \renewcommand\color[2][]{}%
  }%
  \providecommand\transparent[1]{%
    \errmessage{(Inkscape) Transparency is used (non-zero) for the text in Inkscape, but the package 'transparent.sty' is not loaded}%
    \renewcommand\transparent[1]{}%
  }%
  \providecommand\rotatebox[2]{#2}%
  \newcommand*\fsize{\dimexpr\f@size pt\relax}%
  \newcommand*\lineheight[1]{\fontsize{\fsize}{#1\fsize}\selectfont}%
  \ifx\svgwidth\undefined%
    \setlength{\unitlength}{146.83094535bp}%
    \ifx\svgscale\undefined%
      \relax%
    \else%
      \setlength{\unitlength}{\unitlength * \real{\svgscale}}%
    \fi%
  \else%
    \setlength{\unitlength}{\svgwidth}%
  \fi%
  \global\let\svgwidth\undefined%
  \global\let\svgscale\undefined%
  \makeatother%
  \begin{picture}(1,1.25922663)%
    \lineheight{1}%
    \setlength\tabcolsep{0pt}%
    \put(0,0){\includegraphics[width=\unitlength,page=1]{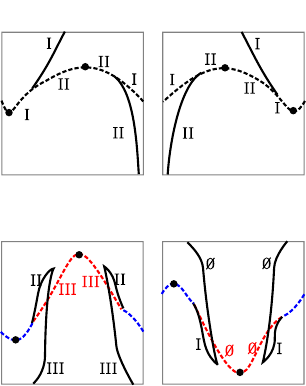}}%
    \put(0.37295971,1.20676206){\color[rgb]{0,0,0}\makebox(0,0)[lt]{\lineheight{0}\smash{\begin{tabular}[t]{l}Saddle\end{tabular}}}}%
    \put(0.06603896,0.50572712){\color[rgb]{0,0,0}\makebox(0,0)[lt]{\lineheight{0}\smash{\begin{tabular}[t]{l}Minimum\end{tabular}}}}%
    \put(0.58222381,0.50572712){\color[rgb]{0,0,0}\makebox(0,0)[lt]{\lineheight{0}\smash{\begin{tabular}[t]{l}Maximum\end{tabular}}}}%
    \put(0,0){\includegraphics[width=\unitlength,page=2]{Fig29.pdf}}%
  \end{picture}%
\endgroup%

\caption[Singularity charts for the saddle, minimum, and maximum movies.]{Top row: two singularity charts for the saddle movies (the difference is the indices of the singular points). Bottom left: a chart for the minimum movie. Bottom right: a chart for the maximum movie. The colors are just meant to help the reader distinguish intersecting arcs. We omit arcs corresponding to singularities which persist from $t=1$ to $t=0$.}
\label{fig:complementchart}
\end{centering}\end{figure}

\subsection{Composite movies}\label{composite}

In this subsection, we will combine previous movies to define more complicated movies of singular fibrations. 

\begin{move}[Simple cancellation movie]\label{canceldot}

Here, we combine the positioning movies with the interior destabilization movie to cancel a cone and dot singularity, when the cone is the nearest singularity of $\F_1$ to the dot.

Let $Z^4=M^3\times I$, and $h:M^3\times I$ be projection onto the second factor. Let $\F_1$ be a singular fibration of $M^3\times 1$ including a cone singularity $q$ and a dot singularity $p$, of cancelling indices.

Suppose there exists an arc $\gamma:[0,1]\to M^3$ which is transverse to the nonsingular leaves of $\F_1$ so that $\gamma(0)=q$ and $\gamma(1)=p$. Let $L$ be a nonsingular leaf near $q$ so that $\nu(q)\cap L$ has two components. Assume that these two components are in {\emph{distinct}} components of $L$. {\bf{Assume moreover that $\gamma((0,1))$ does not meet any singular leaf {\emph{components}} of $\F_1$}}. 

From $t=1$ to $t=1/2$, position the dot $p$ along $\gamma$ so that $\gamma((0,1))$ intersects no singular leaves of $\F_{1/2}$. Now from $t=1/2$ to $t=0$, cancel the singularities (this is an interior destabilization movie). We call the composite $\mathcal{F}_t$ a {\emph{simple cancellation movie}}.

In order to cancel the singularities, it must be the case that the two components of $L\cap\nu(P)$ are not in the same component of $L$. See Figure~\ref{fig:reeb} for an illustration.

\begin{figure}\begin{centering}
\begingroup%
  \makeatletter%
  \providecommand\color[2][]{%
    \errmessage{(Inkscape) Color is used for the text in Inkscape, but the package 'color.sty' is not loaded}%
    \renewcommand\color[2][]{}%
  }%
  \providecommand\transparent[1]{%
    \errmessage{(Inkscape) Transparency is used (non-zero) for the text in Inkscape, but the package 'transparent.sty' is not loaded}%
    \renewcommand\transparent[1]{}%
  }%
  \providecommand\rotatebox[2]{#2}%
  \newcommand*\fsize{\dimexpr\f@size pt\relax}%
  \newcommand*\lineheight[1]{\fontsize{\fsize}{#1\fsize}\selectfont}%
  \ifx\svgwidth\undefined%
    \setlength{\unitlength}{282.90720542bp}%
    \ifx\svgscale\undefined%
      \relax%
    \else%
      \setlength{\unitlength}{\unitlength * \real{\svgscale}}%
    \fi%
  \else%
    \setlength{\unitlength}{\svgwidth}%
  \fi%
  \global\let\svgwidth\undefined%
  \global\let\svgscale\undefined%
  \makeatother%
  \begin{picture}(1,0.35619009)%
    \lineheight{1}%
    \setlength\tabcolsep{0pt}%
    \put(0,0){\includegraphics[width=\unitlength,page=1]{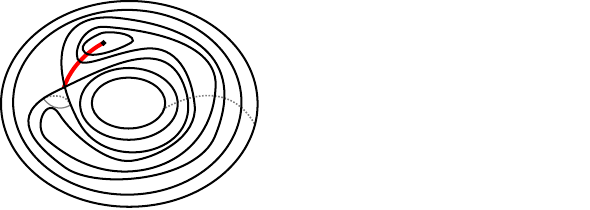}}%
    \put(0.06543344,0.23545585){\color[rgb]{1,0,0}\makebox(0,0)[lt]{\lineheight{0}\smash{\begin{tabular}[t]{l}$\gamma$\end{tabular}}}}%
    \put(0,0){\includegraphics[width=\unitlength,page=2]{Fig30.pdf}}%
  \end{picture}%
\endgroup%

\caption[Uncancellable cones and dots.]{If we attempt to cancel a dot with a cone where the two sheets of $\F_t^{-1}(\theta)$ near the cone are in the same component of $\F_t^{-1}(\theta)$, then we cannot actually destabilize the circular Morse function $\F_t$. 
Left: a local picture of a dot and cone singularity; the two sheets of a leaf near the cone are in the same component of the leaf. Right: We attempt to cancel the dot and cone singularities 
and create a Reeb component, which is not allowed in the definition of a singular fibration.}
\label{fig:reeb}
\end{centering}\end{figure}

\end{move}

\begin{move}[Generalized cancellation movie]

Here, we combine the positioning movies with the interior destabilization movie to cancel a cone and dot singularity which may initially be far apart in $\F_1$.

Let $Z^4=M^3\times I$, and $h:M^3\times I$ be projection onto the second factor. Let $\F_1$ be a singular fibration of $M^3\times 1$ including a cone singularity $q$ and a dot singularity $p$, of cancelling indices.

Suppose there exists an arc $\gamma:[0,1]\to M^3$ which is transverse to the nonsingular leaves of $\F_1$ so that $\gamma(0)=q$ and $\gamma(1)=p$. Let $L$ be a nonsingular leaf near $q$ so that $\nu(q)\cap L$ has two components. Assume that these two components are in {\emph{distinct}} components of $L$. {\bf{We allow $\gamma((0,1))$ to meet other singular leaf components of $\F_1$, but take $\gamma((0,1))$ to be far from the actual singularities.}}

Note that if $M^3$ is a connected manifold with nonempty boundary and $p$ is a dot singularity of $\F_1$, then there {\emph{must}} be a cone $q$ and arc $\gamma$ satisfying the above setup.

Extend $\gamma$ to $\gamma:[-1,1]\to M^3$ so that $\F_1(\gamma(s))=\F_1(\gamma(-s))$ for all $s\in[0,1]$ and $\gamma$ is still transverse to all the leaves of $M^3$ (do this by extending $\gamma$ perpendicularly to the leaves of $\F_1$. If this terminates before extending all the way to $-1$, then the extension terminates at a dot $p'=\gamma(-s_0)$. Exchange the roles of $p$ and $p'$, and let $\gamma(s):=\gamma(-s/s_0)$.)

From $t=1$ to $t=2/3$, position the cone $q$ along $\gamma|_{(-1+\epsilon,1-\epsilon)}$ so that $q$ lies on $\gamma(1-\epsilon)$, so $p$ and $q$ are very close in $h^{-1}(2/3)$. From $t=2/3$ to $1/3$, position the dot $p$ along $\gamma((1-\epsilon,1))$ so that $\F_{1/3}(\gamma(1-\epsilon,1))$ is a very short arc. From $t=1/3$ to $t=0$, cancel the singularities (this is an interior destabilization movie). We call the composite $\mathcal{F}_t$ a {\emph{generalized cancellation movie}}.

\end{move}

\begin{remark}
If $M^3$ is a connected $3$-manifold with nonempty boundary and $\F_1$ is a singular fibration on $M^3$ with $n\ge 1$ dot singularities and $r$ cone singularities, then the generalized cancellation movie allows us to extend $\F_1$ to a valid movie $\F_t$ on $M^3\times I$ so that $\F_0$ has $n-1$ dots and $r-1$ cones. The types of the singularities in $\F_t$ may be chosen freely, except that the cancelled pair of singularities must have cancelling types (\0-,I- or II-,III-).

In particular, if $\F_1$ is a singular fibration on a connected $3$-manifold $M^3$ with nonempty boundary and $\F_1$ has $n$ dot and $n$ cone singularities that are all of types II or III (or all of types \0 or I), then $\F_1$ may be extended to a valid movie $\F_t$ on $M^3\times I$ so that $\F_0$ is nonsingular.
\end{remark}

Now we define our last two movies: the band and disk movies. These two movies are more complicated than previous ones, but viewing them as composites allows us to better understand them.

\begin{move}[Band movie]\label{bandmove}
As in the saddle movie, let $h:B^3\times I\to I$ be projection onto the second factor. Let $R\subset B^3\times I$ be a saddle. That is, $R$ is a properly embedded disk with $R\cap(B^3\times 0)\cong R\cap(B^3\times 1)\cong(2$ arcs$)$ and $h|_R$ is Morse with single index-$1$ critical point and no index-$0$ or -$2$ critical points. Let $b\cong I\times [-\epsilon,\epsilon]$ be a band projected to $B^3\times 1$ with $S^0\times[-\epsilon,\epsilon]\subset\boundary b$ in $\boundary (B^3\setminus\nu(R))$ so that the band $b$ describes the saddle of $R$. (That is, $b$ is obtained by perturbing $R$ so the saddle is degenerate and then projecting a neighborhood of the degeneracy to $B^3\times1$.) Let $\eta=I\times \pt\subset b$ be a core arc of $b$. 

Restrict $h$ to $(B^3\times I)\setminus\nu(R)$. Let $\G_1$ be any singular fibration on $h^{-1}(1)$ with the following properties.
\begin{itemize}
\item Each leaf of $\G_1$ intersects $(\boundary\nu(R))\setminus\boundary B^3$ in arcs.
\item Each leaf of $\G_1$ intersects $b=I\times[-\epsilon,\epsilon]$ in copies of $\pt\times[-\epsilon,\epsilon]$.
\item Each intersection of a leaf of $\G_1$ with the interior of $b$ is transverse.
\end{itemize}

Now consult Figure~\ref{fig:bandmove}. We describe each frame, the arrows indicating decreasing $t$:
\begin{itemize}
\item Top left: A neighborhood $B$ of $\eta(1)$ in $h^{-1}(1)$. 
Play an interior $1$-,$2$-stabilization movie in $B$, during the time span $t=1$ to $1-\epsilon$. Let $q$ be the cone point of the cone oriented horizontally (opening parallel to the band $b$); take $q$ to be type I.
\item Top right: 
Take $\gamma:[-1,1]\to B^3$ to be an arc in $\mathring{h^{-1}(\epsilon)}$ with $\gamma(-1)$ close to $\eta(0)$, $\gamma(1)$ close to $\eta(1)$,and $\gamma(0)=q$. The arc $\gamma$ is parallel to the band away from $\eta(1)$, and then near $\eta(1)$ spirals around the end of the band so that we may take $\gamma(s)=\gamma(-s)$ for all $s$. 
\item Middle left: Position the cone $q$ along the arc $\gamma$. This is a picture of a singular fibration partway through this positioning movie.
\item Middle right: Isotope $(B^3\times I)\setminus\nu(R)$ to shrink the band $b$. We could do this isotopy strictly after the positioning movie, but we believe it is easier to visualize at this step.
\item Bottom left: We have finished positioning the cone. Say $t=\epsilon$. There is a set $B'$ containing $\eta(0)$ in $h^{-1}(\epsilon)$ so that in $B'$, $\G_\epsilon$ agrees with $\F_1$ of the saddle movie. Play the saddle movie in the time span $t=\epsilon$ to $0$.
\item Bottom right: We obtain a singular fibration $\G_0$ of $h^{-1}(0)$. Note $\G_0$ has exactly two singularities interior to $h^{-1}(0)$ (both are type II cones).
\end{itemize}

\begin{figure}\begin{centering}
\scalebox{0.9}{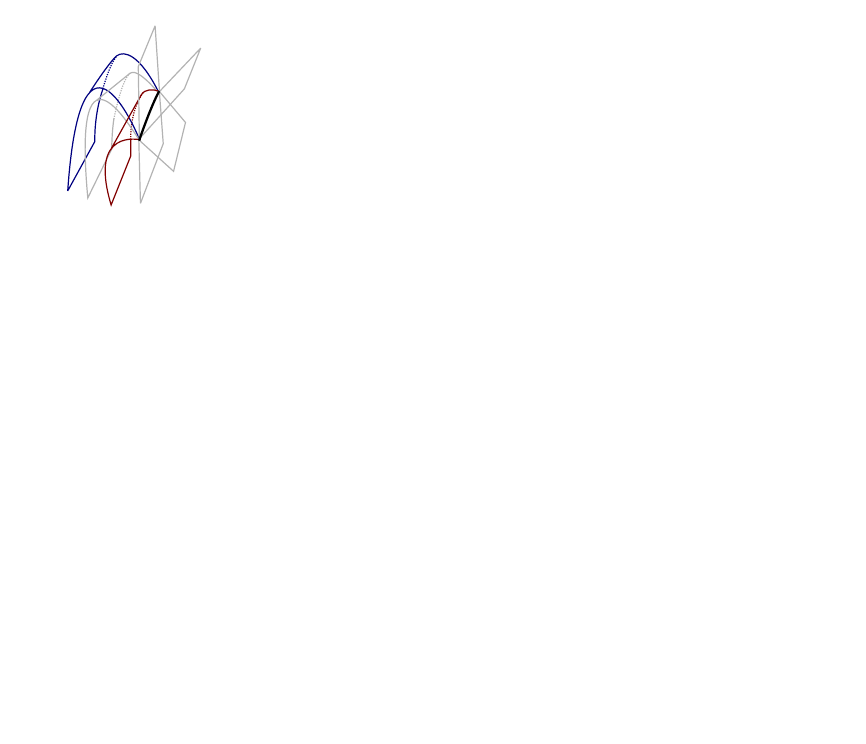}
\caption[Some leaves in a band movie.]{Some leaves of $\G_t$ (a band movie) near a neighborhood of band $b$. Outside this neighborhood, $\G_t=\G_1$ for all $t$.}
\label{fig:bandmove}
\end{centering}\end{figure}

We $\mathcal{G}_t$ a {\emph{band movie}}. This movie is comparitively complicated, so we include some extra cartoons of leaves in $\G_t$ for decreasing $t$. See Figures~\ref{fig:bandmoveoneleaf} and~\ref{fig:bandmovethreeleaves}. 
In Figure~\ref{fig:bandmovesingularity}, we give a more detailed image of the cone singularities in $\G_0$.

\begin{figure}\begin{centering}
\scalebox{0.95}{\includegraphics[width=\textwidth]{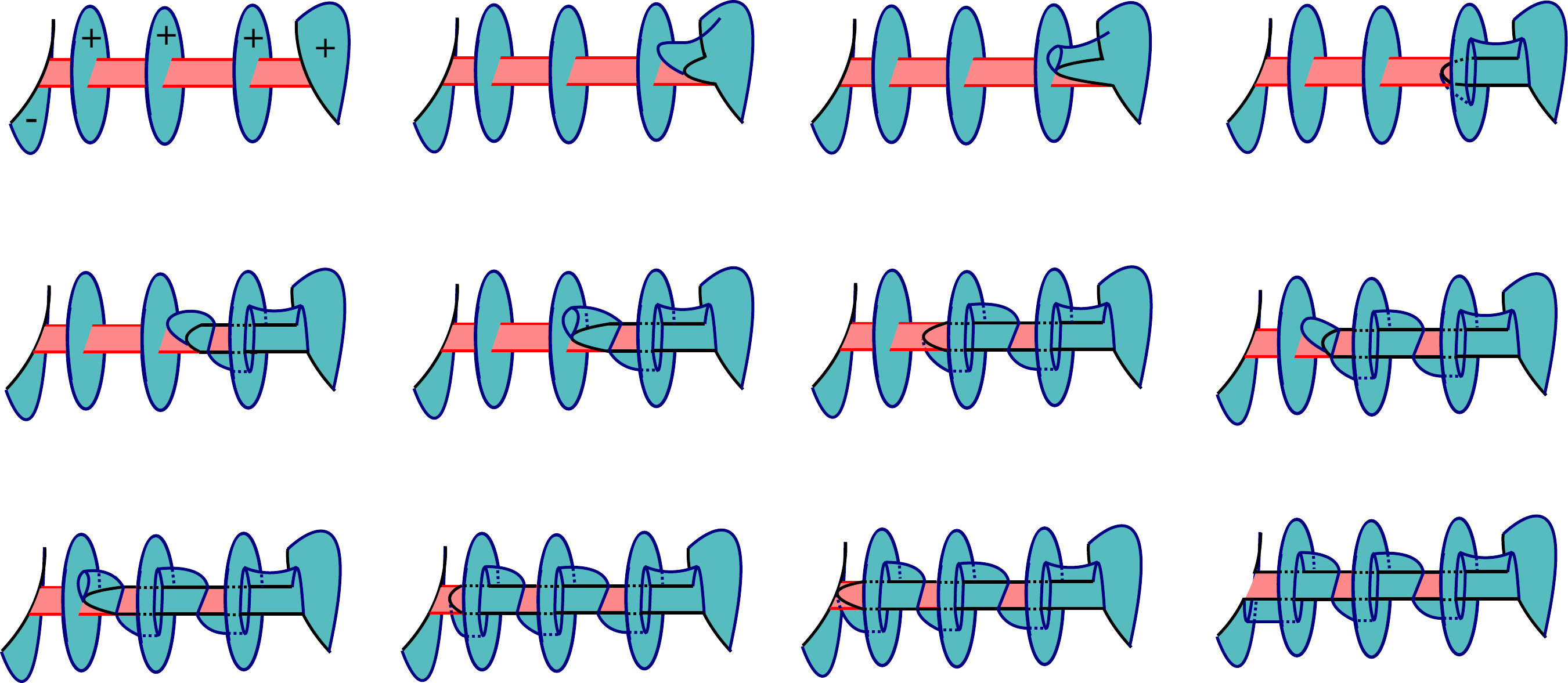}}
\caption[One leaf in a band movie.]{Left to right, top to bottom: One leaf in a band movie $\G_t$ as $t$ decreases from $1$ to $0$.}
\label{fig:bandmoveoneleaf}
\end{centering}\end{figure}

\begin{figure}\begin{centering}
\scalebox{0.95}{\includegraphics[width=\textwidth]{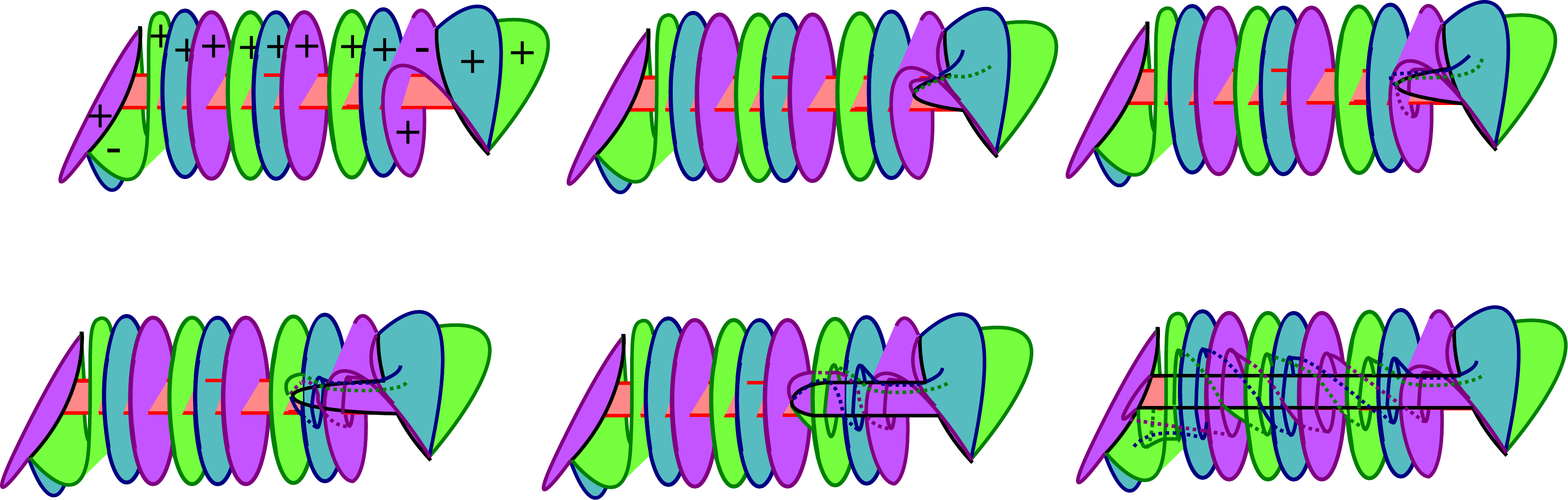}}
\caption[Three leaves in a band movie.]{Left to right, top to bottom: Three leaves in a band movie $\G_t$ as $t$ decreases from $1$ to $0$.}
\label{fig:bandmovethreeleaves}
\end{centering}\end{figure}

\begin{figure}\begin{centering}
\scalebox{0.9}{\includegraphics[width=\textwidth]{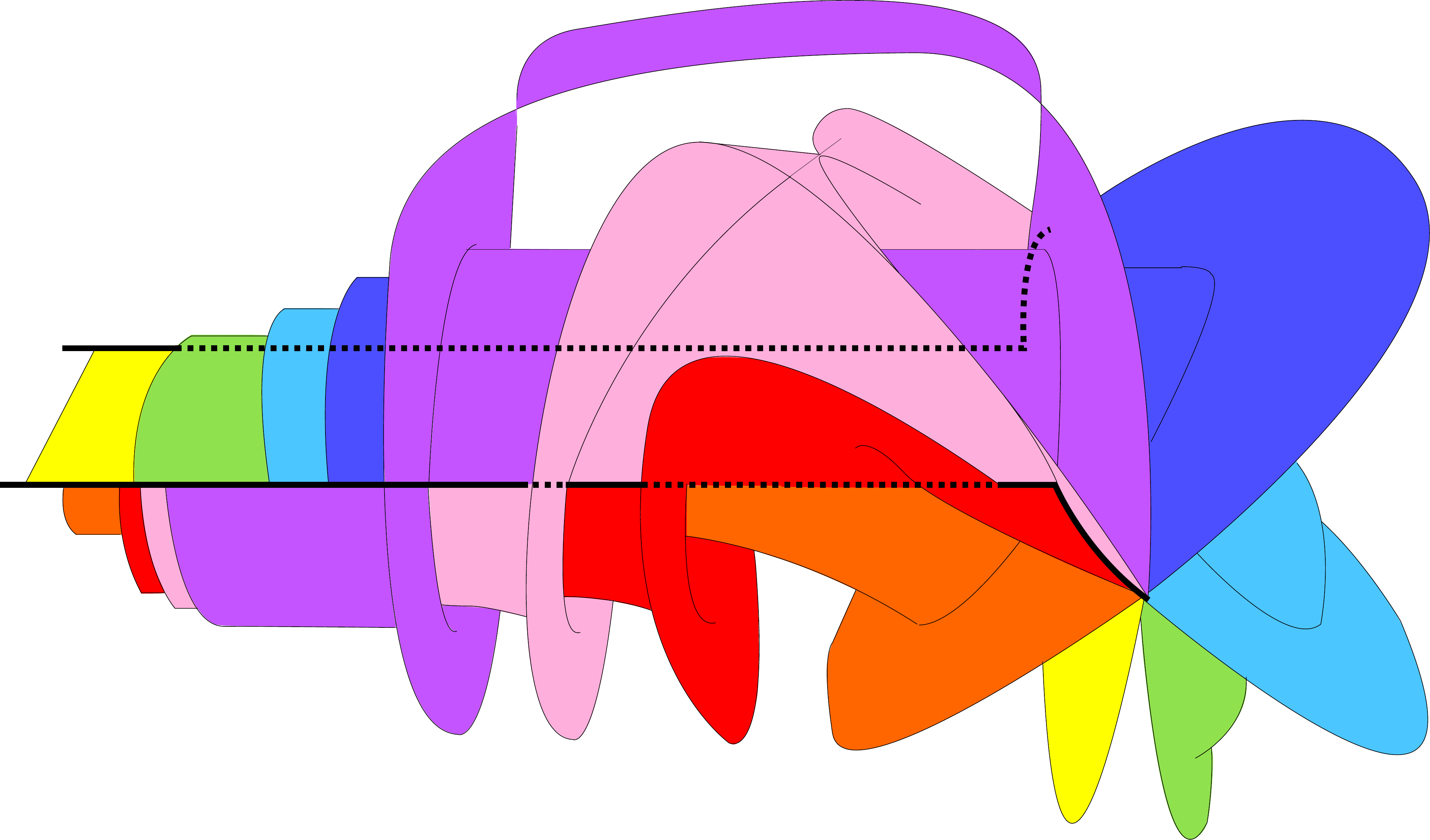}}
\caption[A detailed view of a cone singularity at the end of a band movie.]{A detailed view of a cone singularity in $\G_0$ near the end of band $b$ in a band movie. The two cones have identical neighborhoods, modulo rotation.}
\label{fig:bandmovesingularity}
\end{centering}\end{figure}

\end{move}

\begin{move}[Disk movie]\label{diskmove}

As in the minimum movie, let $h:B^3\times I\to I$ be projection onto the second factor. Let $R\subset B^3\times I$ be a disk with one minimum. That is, $R$ is a properly embedded disk with unknotted boundary in $B^3\times 1$ so that $h|_R$ is Morse with single index-$0$ critical point. Let $D$ be a disk in $B^3\times 1$ bounded by $\boundary R$. 

Restrict $h$ to $(B^3\times I)\setminus\nu(R)$. Let $\G_1$ be any singular fibration on $h^{-1}(1)$ with the following properties.
\begin{itemize}
\item All but finitely many leaves of $\G_1$ intersect $D$ transversely.
\item Where a leaf of $\G_1$ is tangent to $D$, that leaf locally looks like a minimum or maximum disk or saddle and locally intersects $D$ in a single point.
\item No singular points of $\G_1$ meet $\nu(D)$.
\item Near $\boundary \nu(R)$, the leafs of $\G_1$ intersect $D$ in circles parallel to $\boundary\nu(R)$.
\end{itemize}

By abuse of notation, take a copy of $D:=D\setminus\nu(R)$ to live in each $h^{-1}(t)$. Now consult Figure~\ref{fig:cancelminima}.

Note $\G_1$ induces a singular foliation $\F$ on $D$, whose singularities are dots $p_1,\ldots, p_k$ and crosses, for $k\ge 1$. Let $p_i\in D$ be a dot singularity in this foliation. Let $P_i$ be the set $\cup\{E\subset D$ a disk $\mid p_i\in E$, $\boundary E$ in one leaf of $\G_1, p_i$ is the unique singularity of $\F|_{\mathring{E}}\}$. Let $D_i$ be the closure of $\cup_{E\in P_i}E$. Then either $\boundary D_i$ meets a cross singularity of $\F$, or $D_i=D$. Note $D_i$ may not be a disk.

\begin{proposition}
For some choice of $i$, $D_i$ is a disk.
\end{proposition}
\begin{proof}
Suppose $D_1$ is not a disk. Let $E_1$ be the interior component of $D\setminus D_1$. $E_1$ is an open disk, so there is another $p_i$ in $E_1$ (reorder so $p_2\in E_1$). Now suppose $D_2$ is not a disk. Similarly let $E_2$ be the interior component of $D\setminus D_2$. Then $E_2$ is an open disk, so there is another $p_i$ in $E_2$. Note $E_2\subset E_1$, so $p_i\neq p_1$ or $p_2$ (reorder so $p_3\in E_2$).

We continue inductively. Since there are a finite number of $p_i$, eventually some $D_i$ must be a disk.
\end{proof}

So without loss of generality, assume $D_1$ is a disk.

From $t=1$ to $3/4$, play an interior $0$-,$1$- or $2$-,$3$- stabilization movie in a neighborhood of $p_1$ (choose according to Figure~\ref{fig:howtostabmin}; take the cone to be type II and the dot to be type III). Call the new cone singularity $q$; take $D$ to intersect $q$ so that $\G_{3/4}$ induces the same foliation as $\G_1$ on $D$, but one dot singularity in the foliation $\G_{3/4}|_D$ is at $q$. From $t=3/4$ to $t=1/2$, if $D\neq D_1$ then position the cone along the disk $D_1$. (If $D=D_1$, position the cone along a copy of $D_1$ shrunk slightly to lie in the interior of $h^{-1}(3/4)$.)

Now in $\G_{1/2}$, if $D_1=D$ then there is a leaf component of $\G_{1/2}$ which is a disk parallel to $D$. If $D_1\neq D$, then we can isotope $\G_{1/2}$ so $D$ is disjoint from all singularities of $\G_{1/2}$ and the foliation $\G_{1/2}$ induces on $D$ has $k-1$ dot singularities.

Repeat this process another $k-1$ times, from $t=1/2$ to $t=1/4$, so that there is a leaf component $E$ of $\G_{1/4}$ which is a disk parallel to $D$.

Then there is a subset $B\subset h^{-1}(1/4)$ containing $E$ so that $\G_{1/4}$ agrees with $\F_1$ of the minimum movie in $B$. Play the minimum movie in $B$ from $t=1/4$ to $0$.

We call the movie $\mathcal{G}_t$ a {\emph{disk movie}}.

\begin{figure}\begin{centering}
\scalebox{0.8}{
\begingroup%
  \makeatletter%
  \providecommand\color[2][]{%
    \errmessage{(Inkscape) Color is used for the text in Inkscape, but the package 'color.sty' is not loaded}%
    \renewcommand\color[2][]{}%
  }%
  \providecommand\transparent[1]{%
    \errmessage{(Inkscape) Transparency is used (non-zero) for the text in Inkscape, but the package 'transparent.sty' is not loaded}%
    \renewcommand\transparent[1]{}%
  }%
  \providecommand\rotatebox[2]{#2}%
  \newcommand*\fsize{\dimexpr\f@size pt\relax}%
  \newcommand*\lineheight[1]{\fontsize{\fsize}{#1\fsize}\selectfont}%
  \ifx\svgwidth\undefined%
    \setlength{\unitlength}{407.48580812bp}%
    \ifx\svgscale\undefined%
      \relax%
    \else%
      \setlength{\unitlength}{\unitlength * \real{\svgscale}}%
    \fi%
  \else%
    \setlength{\unitlength}{\svgwidth}%
  \fi%
  \global\let\svgwidth\undefined%
  \global\let\svgscale\undefined%
  \makeatother%
  \begin{picture}(1,1.19707296)%
    \lineheight{1}%
    \setlength\tabcolsep{0pt}%
    \put(0,0){\includegraphics[width=\unitlength,page=1]{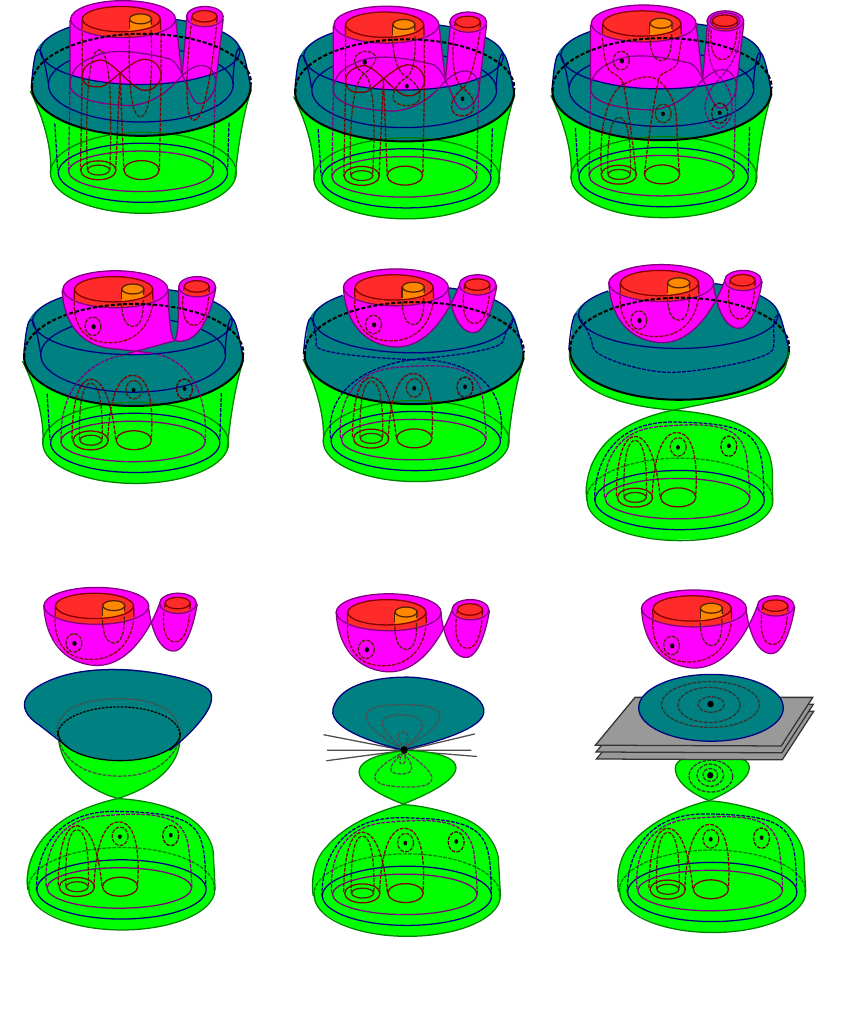}}%
    \put(0.96985516,0.78926941){\color[rgb]{0,0,0}\rotatebox{90}{\makebox(0,0)[t]{\lineheight{0}\smash{\begin{tabular}[t]{c}$G_t$  agrees with top of\\minimum movie near the disk\end{tabular}}}}}%
    \put(0,0){\includegraphics[width=\unitlength,page=2]{Fig35.pdf}}%
    \put(0.35975512,0.0005267){\color[rgb]{0,0,0}\makebox(0,0)[lt]{\lineheight{0}\smash{\begin{tabular}[t]{l}Minimum movie\end{tabular}}}}%
  \end{picture}%
\endgroup%
}
\caption[Some leaves in a disk movie.]{Left to right, top to bottom: Some leaves of $\G_t$ as $t$ decreases from $1$ to $0$ in a disk movie. As $t$ decreases, we 
perform interior stabilizations and position the cones along disks until in a set $B$ containing $\boundary D$ in $h^{-1}(1/4)$, $\G_{1/4}$ agrees with the top of the minimum movie. We play the minimum movie from $t=1/4$ to $0$.}
\label{fig:cancelminima}
\end{centering}\end{figure}

\begin{figure}\begin{centering}
\begingroup%
  \makeatletter%
  \providecommand\color[2][]{%
    \errmessage{(Inkscape) Color is used for the text in Inkscape, but the package 'color.sty' is not loaded}%
    \renewcommand\color[2][]{}%
  }%
  \providecommand\transparent[1]{%
    \errmessage{(Inkscape) Transparency is used (non-zero) for the text in Inkscape, but the package 'transparent.sty' is not loaded}%
    \renewcommand\transparent[1]{}%
  }%
  \providecommand\rotatebox[2]{#2}%
  \newcommand*\fsize{\dimexpr\f@size pt\relax}%
  \newcommand*\lineheight[1]{\fontsize{\fsize}{#1\fsize}\selectfont}%
  \ifx\svgwidth\undefined%
    \setlength{\unitlength}{327.15818871bp}%
    \ifx\svgscale\undefined%
      \relax%
    \else%
      \setlength{\unitlength}{\unitlength * \real{\svgscale}}%
    \fi%
  \else%
    \setlength{\unitlength}{\svgwidth}%
  \fi%
  \global\let\svgwidth\undefined%
  \global\let\svgscale\undefined%
  \makeatother%
  \begin{picture}(1,0.59187647)%
    \lineheight{1}%
    \setlength\tabcolsep{0pt}%
    \put(0,0){\includegraphics[width=\unitlength,page=1]{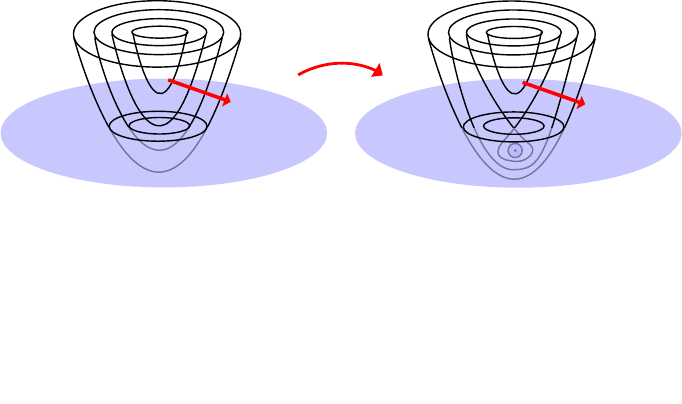}}%
    \put(0.40646512,0.53754252){\color[rgb]{0,0,0}\makebox(0,0)[lt]{\lineheight{0}\smash{\begin{tabular}[t]{c}0-,1-\\stabilization\end{tabular}}}}%
    \put(0,0){\includegraphics[width=\unitlength,page=2]{Fig36.pdf}}%
    \put(0.40673433,0.2210594){\color[rgb]{0,0,0}\makebox(0,0)[lt]{\lineheight{0}\smash{\begin{tabular}[t]{c}2-,3-\\stabilization\end{tabular}}}}%
  \end{picture}%
\endgroup%

\caption[Closeup of a dot singularity introduced during a disk movie.]{Leaves near a dot in the singularity $\G_t$ induces on $D$. The red arrow indicates the positive $S^1$ direction of $\G_t$. We stabilize, and choose the indices of the stabilization as indicated.}
\label{fig:howtostabmin}
\end{centering}\end{figure}

There are exactly $k$ more cone and $k+2$ more dot singularities in $\G_0$ than in $\G_1$, all of which are types II or III.
\end{move}

In Figure~\ref{fig:compositechart}, we give valid singularity charts for the simple and generalized cancellation movies, and the band and disk movies.
\begin{figure}\begin{centering}
\scalebox{0.95}{
\begingroup%
  \makeatletter%
  \providecommand\color[2][]{%
    \errmessage{(Inkscape) Color is used for the text in Inkscape, but the package 'color.sty' is not loaded}%
    \renewcommand\color[2][]{}%
  }%
  \providecommand\transparent[1]{%
    \errmessage{(Inkscape) Transparency is used (non-zero) for the text in Inkscape, but the package 'transparent.sty' is not loaded}%
    \renewcommand\transparent[1]{}%
  }%
  \providecommand\rotatebox[2]{#2}%
  \newcommand*\fsize{\dimexpr\f@size pt\relax}%
  \newcommand*\lineheight[1]{\fontsize{\fsize}{#1\fsize}\selectfont}%
  \ifx\svgwidth\undefined%
    \setlength{\unitlength}{368.87634758bp}%
    \ifx\svgscale\undefined%
      \relax%
    \else%
      \setlength{\unitlength}{\unitlength * \real{\svgscale}}%
    \fi%
  \else%
    \setlength{\unitlength}{\svgwidth}%
  \fi%
  \global\let\svgwidth\undefined%
  \global\let\svgscale\undefined%
  \makeatother%
  \begin{picture}(1,0.50305149)%
    \lineheight{1}%
    \setlength\tabcolsep{0pt}%
    \put(0,0){\includegraphics[width=\unitlength,page=1]{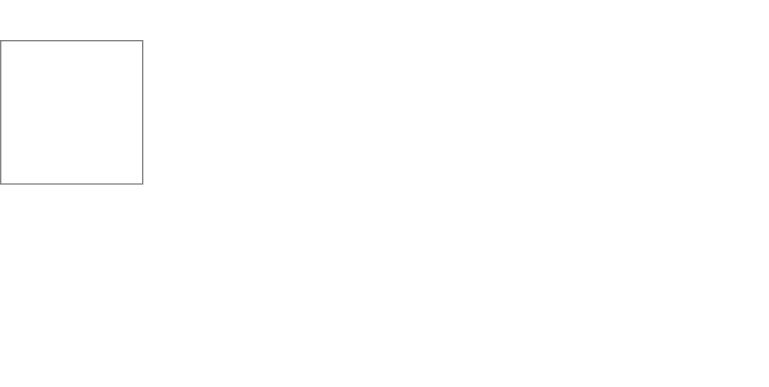}}%
    \put(0.05433059,0.46771728){\color[rgb]{0,0,0}\makebox(0,0)[lt]{\lineheight{0}\smash{\begin{tabular}[t]{l}Simple Cancellation\end{tabular}}}}%
    \put(0,0){\includegraphics[width=\unitlength,page=2]{Fig37.pdf}}%
    \put(0.84098126,0.47846953){\color[rgb]{0,0,0}\makebox(0,0)[lt]{\lineheight{0}\smash{\begin{tabular}[t]{l}Band\end{tabular}}}}%
    \put(0.55336968,0.41626916){\color[rgb]{0,0,0}\makebox(0,0)[lt]{\lineheight{0}\smash{\begin{tabular}[t]{l}Disk\end{tabular}}}}%
    \put(0,0){\includegraphics[width=\unitlength,page=3]{Fig37.pdf}}%
    \put(0.03169464,0.21211068){\color[rgb]{0,0,0}\makebox(0,0)[lt]{\lineheight{0}\smash{\begin{tabular}[t]{l}Generalized Cancellation\end{tabular}}}}%
    \put(0,0){\includegraphics[width=\unitlength,page=4]{Fig37.pdf}}%
  \end{picture}%
\endgroup%
}
\caption[Singularity charts for cancellation movies.]{Top left: two singularity charts for the simplified cancellation movie, depending on the indices of the cancelled singularities. Bottom left: two singularity charts for the generalized cancellation movie. Middle: A potential chart for the disk movie, where we included one $0$,$1$- stabilization and one $2$-,$3$-stabilization (recall these choices depend on orientations of leaves of $\F_1$ intersecting the disk). Right: two singularity charts for the band movie (recall the types of the cones created in the $1$-$2$- stabilization depends on the orientation of $\F$.}
\label{fig:compositechart}
\end{centering}\end{figure}

\subsection{A key lemma}\label{sec:usemovie}

In this section, we prove the following lemma.

\begin{lemma}\label{usemoviealt}
Let $Z^4$ be a compact $4$-manifold. Fix a Morse function $h:Z^4\to I$. Suppose $\F_t$ is a movie of singular fibrations on $Z^4$ (i.e. a family of smooth maps $\F_t:h^{-1}(t)\to S^1$) so that $\F_1$ and $\F_0$ are fibrations (with no singularities) and $\F_t$ is a concatenation of band, disk, and simple or generalized cancellation movies (to be defined in section~\ref{sec:blocks}) in order (decreasing $t$).

Then $\F_1\sqcup\F_0$ extends to a fibration on $Z^4$.
\end{lemma}

Our strategy for proving Theorem~\ref{maintheorem} is to construct a movie of singular fibrations on 
$B^4\setminus\nu(D)$ from $S^3$ to just above the lowest minimum of $D$. We choose this movie so that the top singular fibration on $S^3\setminus\nu(K)$ is the honest fibration of the fibered knot $K$, and that the bottom singular fibration is the fibration of a solid torus by meridian disks. 
Lemma~\ref{usemoviealt} ensures that this movie of singular fibrations actually yields a fibration of $B^4\setminus\nu(D)$. By~\cite{jeff}, this immediately completes the proof of Theorem~\ref{maintheorem} (i.e.~\cite{jeff} implies that the fibers are handlebodies), but we will separately analyze the constructed movie to directly show that the fibers of $B^4\setminus\nu(D)$ are handlebodies and even understand how they are embedded in $B^4$ (see Section~\ref{sec:conclusion}).

\begin{proof}
Let $\F_t$ be a singular movie on $Z^4$ with Morse function $h:Z^4\to I$ so that $\F_1$ and $\F_0$ are nonsingular. Assume $\F_t$ is built by composing some number of band, disk, and simple or generalized cancellation movies, in that order (decreasing $t$).

By Theorem~\ref{fibrationthm}, to conclude that $Z^4$ is fibered it is sufficient to show that $\F_t$ is valid (perhaps up to reparametrization).
In subsection~\ref{composite}, we found singular charts for the band, disk, and generalized cancellation movies.

In particular, the band movie admits a chart in which the two cones at the bottom are type II (opposite indices), and assumes nothing about existing cones and dots.

The disk movie admits a chart in which the cones and dots at the bottom are all type II or III, and assumes nothing about existing cones and dots.

The cancellation movies each admit a chart in which a type II and type III singularity die (the slopes of the corresponding arcs may be positive or negative), and all other arcs are nearly vertical (of arbitrarily-signed slope).

Thus, by stacking these charts 
we obtain a valid singularity chart for $\F_t$. We give two sample singularity charts of $\F_t$ in Figures~\ref{fig:fullchart1nofiber} and~\ref{fig:fullchart2nofiber}.
\end{proof}

\begin{figure}\begin{centering}
\scalebox{0.9}{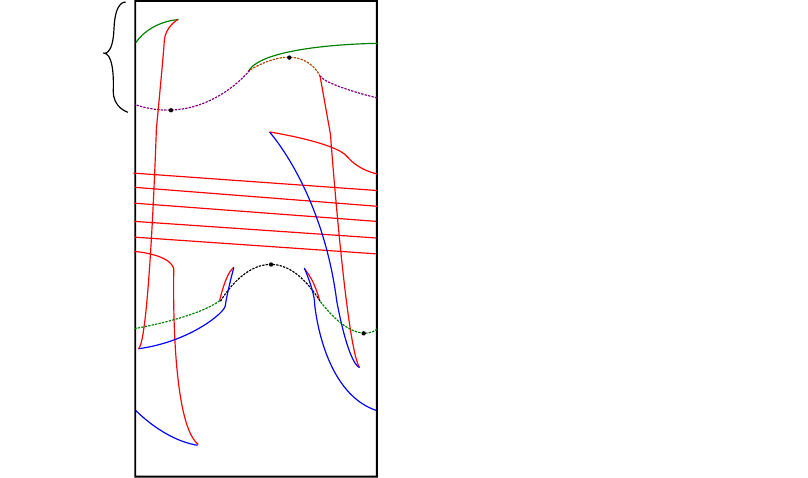}
\caption[An example valid singularity chart.]{A possible valid singularity chart for $\F_t$ when $\F_t$ is a composition of one band movie, one disk movie, and then three cancellation movies. Left: All cancellations are simple. Right: The cancellations are generalized.}
\label{fig:fullchart1nofiber}
\end{centering}\end{figure}

\begin{figure}\begin{centering}
\scalebox{0.9}{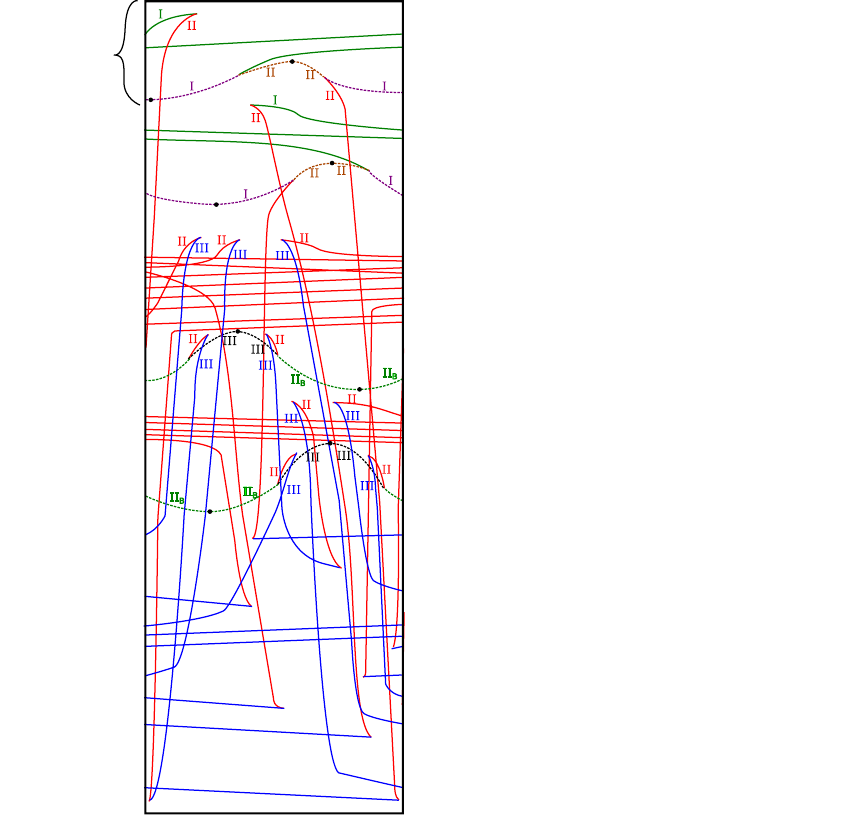}
\caption[A more complicated example of a valid singularity chart.]{A possible valid singularity chart for $\F_t$ when $\F_t$ is a composition of two band movies, two disk movie, and then nine cancellation movies. Left: All cancellations are simple. Right: The cancellations are generalized.}
\label{fig:fullchart2nofiber}
\end{centering}\end{figure}

\section{Extending the fibration on the knot complement to a fibration of the disk complement\label{construction}}

\subsection{Proof of Theorem~\ref{maintheorem}}

Recall the setup of the ribbon disk $D$ with boundary a fibered knot $K$. We have:
\begin{itemize}
\item $D\subset S^3\times[1,3]$.
\item For $t_1>t_2>\ldots>t_n\in(2,3)$, $D\cap (S^3\times t_i)=$ a link with a band attached (referred to as a ``band," ``ribbon band," or, ``fission band"). 
\item For $s_1>\ldots>s_{n+1}\in(1,2)$, $D\cap (S^3\times s_i)$ is an $(n+1-i)$-component unlink and a disjoint disk (referred to as a minimum disk).
\item For all other $u\in(s_{n+1},3]$, $D\cap (S^3\times u)$ is a nonsingular link. For all $u\in[0,s_{n+1})$, $D\cap (S^3\times u)=\emptyset$.
\end{itemize}

\begin{maintheorem}
Suppose the fission bands at $t=t_1,\ldots, t_n$ are disjoint and transverse to the fibration on $S^3\setminus\nu(K)$ (when projected to $S^3\times 3$). Then the fibration on $S^3\setminus\nu(K)$ extends to a fibration of handlebodies on $B^4\setminus\nu(D)$.
\end{maintheorem}

\begin{proof}
Let $\F_3$ be the fibration on $S^3\setminus \nu(K)=h^{-1}(3)\setminus\nu(D)$. For $i=1,\ldots, n$, from $t=t_i+\epsilon$ to $t_i-\epsilon$, play a band movie in a neighborhood of $b_i$. Extend $\F_{t_n-\epsilon}$ vertically to $\F_2:h^{-1}(2)\setminus\nu(D)$. The singular fibration $\F_2$ has exactly $2n$ cone singularities and zero dot singularities.

Let $d_1,\ldots, d_{n+1}$ be disjoint disks in $S^3\times 2$ with $d_i$ corresponding to the minimum at $s_i$. For $i=1,\ldots, n$, from $t=s_i+\epsilon$ to $s_i-\epsilon$, play a disk movie near each $d_i$. Extend $\F_t$ vertically to $\F_{s_n-\epsilon}$ (that is, just past the penultimate minimum). 
The singular fibration $\F_{s_n-\epsilon}$ has an equal number of cone and dot singularities. (Say $\F_{s_n-\epsilon}$ has $r$ cone singularities.)

Finally, we play $r$ generalized cancellation movies to extend $\F_t$ to a movie on $h^{-1}([s_{n+1}+\epsilon,3])\setminus\nu(D)$ so that $\F_{s_{n+1}+\epsilon}$ has no singularities.

By Lemma~\ref{usemoviealt}, $(B^4\setminus\nu(D))\cap h^{-1}([s_{n+1}+\epsilon,3])$ admits a fibration $\F$. Note $(B^4\setminus\nu(D))\cap h^{-1}([0,s_{n+1}])\cong(B^4\setminus($trivial disk$))\cong S^1\times B^3$. Glue $\F$ to the fibration of $S^1\times B^3$ by $3$-balls to obtain a fibration $\G$ of $B^4\setminus\nu(D)$.

It follows immediately from~\cite[Theorem 1.1]{jeff} that the fibers of $\G$ are handlebodies, completing the proof of Theorem~\ref{maintheorem}. 
\end{proof}

\begin{remark}
 It is essential that $D$ is a ribbon disk and that the fission bands defining $D$ are transverse to the fibration. These conditions ensure that after extending the movie of fibrations below the heights of all bands and all but one minimum disk, there are an equal number of cone and dot singularities in the cross-section $\{t=s_n-\epsilon\}$ (above the bottom minimum disk).
\end{remark}

In subsection~\ref{geomcancel}, we show explicitly that the fibers of $\G$ are handlebodies. In section~\ref{sec:conclusion}, we show how to obtain $2$-handle attaching circles for a handlebody fiber $H$ of $\G$ on $\boundary H\cap h^{-1}(3)=($a fiber for $K)$. This explicitly describes the embedding of $H$ into $B^4\setminus\nu(D)$ (up to isotopy rel boundary).

\subsection{Remark on the generalized cancellation movies}\label{sec:simplecancel}

In the proof of Theorem~\ref{maintheorem}, the generalized cancellation movies played from $t=s_{n}-\epsilon$ to $s_{n+1}+\epsilon$ may actually be taken to be simple cancellation movies. This is not strictly necessary for Theorem~\ref{maintheorem}, but is necessary to explicitly understand the embedding of a handlebody fiber of $D$ into $B^4\setminus\nu(D)$, which we will discuss in section~\ref{sec:conclusion}. We first give a more general statement.

\begin{lemma}\label{simplecancel}
Let $V$ be a connected $3$-manifold with torus boundary, and $\F_1$ be a singular fibration on $V$ with $r$ cone and $r$ dot singularities, which are all type II or III. Assume that $\F_1^{-1}(\theta)\cap\boundary V$ is a connected simple closed curve for each $\theta$. Then by playing only simple cancellation movies, $\F_1$ can be extended to a valid movie of singular fibrations $\F_t$ on $V\times I$ so that $\F_0$ is nonsingular.
\end{lemma}

\begin{proof}
Note that $\F_1$ can be extended to a movie $\F_t$ in which $\F_0$ is nonsingular via {\emph{generalized}} cancellation movies. Therefore, $V=V\times 0$ admits a fibration over $S^1$. Then $V$ is an irreducible $3$-manifold (with boundary), and in particular every $2$-sphere embedded in $V$ bounds a $3$-ball.

\begin{proposition}\label{notindex2}
If $\F:=\F_1$ has any singularities, then it has a dot singularity $p$ so that when flowing from $p$ outward perpendicular to the leaves of $\F$, 
the first singular leaf component has a cone $q$ so that the following are true:
\begin{itemize}
\item With respect to the circular Morse function $\F$, $p$ and $q$ are critical points of index (0,1) or (3,2).
\item Suppose $\F^{-1}(\theta)\cap\nu(q)$ is a two-sheeted hyperboloid. The two sheets are {\emph{not}} in the same component of $\F^{-1}(\theta)$.
\end{itemize}
\end{proposition}

In more words, $p$ and $q$ are critical points of $\F$ of cancelling indices. Suppose they are index $0$ and $1$, respectively. The second condition relates to a handle decomposition of $V$ induced by $\F$, in which $p$ and $q$ correspond to a $0$- and $1$-handle. We are requiring that the feet of the $1$-handle corresponding to $q$ cannot both be onthe $0$-handle corresponding to $p$. This would imply that at least the handles corresponding to $p$ and $q$ could be cancelled, so we might hope to achieve this cancellation by cancelling the singularities $p$ and $q$ of $\F$. 
\begin{proof}
Suppose $\F$ has $k_i$ critical points of index $i$. Then $0=\chi(V)=k_0-k_1+k_2-k_3$. But recall there are an equal number of dot and cone singularities in $V$, so $k_0+k_3=k_1+k_2$. Therefore, $k_0=k_1$ and $k_2=k_3$.

Let $p$ be a dot singularity of $\F$, and reorient $\F$ if necessary so that $p$ is an index-$0$ critical point of $\F$. 
Flow from $p$ in the positive $S^1$ direction until finding a singular leaf (with cone point $q$). This must happen eventually, since $\F$ is defined on a manifold with boundary.

\begin{claim}
If $q$ is a cone of index $1$, then $q$ satisfies the second condition of the proposition.
\end{claim}
\begin{innerproof}[Proof of claim]
We argue that every index-$1$ cone satisfies the second condition.

Recall $k_0=k_1$. Fix a nonsingular leaf $L$ of $\F$ which meets $\boundary V$. Reparametrize $\F$ in a neighborhood of the index-$0$ points so that $\F^{-1}(0)=L\cup\{$the $k_0$ index-$0$ critical points$\}$ and the cone singularities all lie in distinct leaves of $\F$. Let $V_\theta=\cup_{0\le x\le\theta}\F^{-1}(x)$. Suppose $V_{\theta}$ has $n$ components for $\theta_0-\epsilon<\theta<\theta_0$ while $V_{\theta}$ has $n-1$ components for $\theta_0\le\theta<\theta+\epsilon$, for some $\theta_0$. Then $\F^{-1}(\theta_0)$ must contain a cone satisfying the proposition. See Figure~\ref{fig:vtheta}. (Recall each $\F^{-1}(\theta)$ meets $\boundary V$ in one curve.)

Note $V_0=\F^{-1}(0)$ has at least $k_0+1$ components, while $V_{2\pi}=V$ has one component. Therefore, the number of components must decrease at least $k_0$ times. But there are $k_1=k_0$ cones, so every index-$1$ cone must satisfy the second condition of proposition. This completes the proof of the claim.
\end{innerproof}

\begin{figure}\begin{centering}
\scalebox{0.8}{
\begingroup%
  \makeatletter%
  \providecommand\color[2][]{%
    \errmessage{(Inkscape) Color is used for the text in Inkscape, but the package 'color.sty' is not loaded}%
    \renewcommand\color[2][]{}%
  }%
  \providecommand\transparent[1]{%
    \errmessage{(Inkscape) Transparency is used (non-zero) for the text in Inkscape, but the package 'transparent.sty' is not loaded}%
    \renewcommand\transparent[1]{}%
  }%
  \providecommand\rotatebox[2]{#2}%
  \newcommand*\fsize{\dimexpr\f@size pt\relax}%
  \newcommand*\lineheight[1]{\fontsize{\fsize}{#1\fsize}\selectfont}%
  \ifx\svgwidth\undefined%
    \setlength{\unitlength}{441.85467577bp}%
    \ifx\svgscale\undefined%
      \relax%
    \else%
      \setlength{\unitlength}{\unitlength * \real{\svgscale}}%
    \fi%
  \else%
    \setlength{\unitlength}{\svgwidth}%
  \fi%
  \global\let\svgwidth\undefined%
  \global\let\svgscale\undefined%
  \makeatother%
  \begin{picture}(1,0.46051666)%
    \lineheight{1}%
    \setlength\tabcolsep{0pt}%
    \put(0,0){\includegraphics[width=\unitlength,page=1]{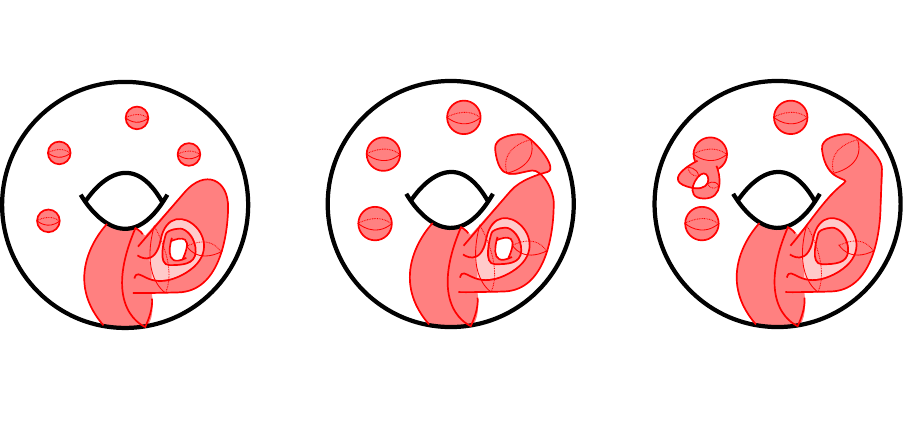}}%
    \put(0.09317116,0.17784553){\color[rgb]{0,0,0}\makebox(0,0)[lt]{\lineheight{0}\smash{\begin{tabular}[t]{l} \end{tabular}}}}%
    \put(0.37744745,0.41547825){\color[rgb]{0,0,0}\makebox(0,0)[lt]{\lineheight{0}\smash{\begin{tabular}[t]{l}increasing $\theta$\end{tabular}}}}%
    \put(0,0){\includegraphics[width=\unitlength,page=2]{Fig42.pdf}}%
    \put(0.41920768,0.06423994){\color[rgb]{0,0,0}\makebox(0,0)[lt]{\lineheight{1}\smash{\begin{tabular}[t]{c}number of\\components\\decreases\end{tabular}}}}%
    \put(0.68925525,0.063272){\color[rgb]{0,0,0}\makebox(0,0)[lt]{\lineheight{1}\smash{\begin{tabular}[t]{c}reach a cone singularity,\\but number of components\\does not decrease\end{tabular}}}}%
    \put(0.11668204,0.0604186){\color[rgb]{0,0,0}\makebox(0,0)[lt]{\lineheight{0}\smash{\begin{tabular}[t]{l}$V_\theta$\end{tabular}}}}%
    \put(0,0){\includegraphics[width=\unitlength,page=3]{Fig42.pdf}}%
  \end{picture}%
\endgroup%
}
\caption[Schematic of $V_\theta$ for increasing $\theta$.]{Schematic of $V_\theta$ for increasing $\theta$. Left: $V_\theta$ for $\theta>0$ small. Middle: The number of components of $V_\theta$ decreases. $\F^{-1}(\theta)$ contains a cone satisfying the second condition of the proposition. Right: We draw two cones. One is index-$2$, one is index-$1$ but does not satisfy the other condition. The number of components of $V_\theta$ does not decrease.}
\label{fig:vtheta}
\end{centering}\end{figure}

Now assume $q$ is an index-$2$ critical point. See Figure~\ref{fig:noindex2}. There is a leaf component of $\F$ near $q$ which is sphere separating $V$ into two pieces, one of which is a ball $B$ not containing $p$. There is some dot $p'$ inside $B$. Repeat the argument starting from the dot $p'$ instead of $B$. If $p'$ does not satisfy Proposition~\ref{notindex2}, then we find another spherical leaf component of $\F$ inside $B$, bouding a ball in $B'\subset B$ that does not contain $p'$. Then there is a dot point $p''$ inside $B'$. We continue inductively; since there are finitely many critical points in $\F$, eventually we must find a dot point as in Proposition~\ref{notindex2}, completing the proof of Proposition~\ref{notindex2}.
\end{proof}

\begin{figure}\begin{centering}
\begingroup%
  \makeatletter%
  \providecommand\color[2][]{%
    \errmessage{(Inkscape) Color is used for the text in Inkscape, but the package 'color.sty' is not loaded}%
    \renewcommand\color[2][]{}%
  }%
  \providecommand\transparent[1]{%
    \errmessage{(Inkscape) Transparency is used (non-zero) for the text in Inkscape, but the package 'transparent.sty' is not loaded}%
    \renewcommand\transparent[1]{}%
  }%
  \providecommand\rotatebox[2]{#2}%
  \newcommand*\fsize{\dimexpr\f@size pt\relax}%
  \newcommand*\lineheight[1]{\fontsize{\fsize}{#1\fsize}\selectfont}%
  \ifx\svgwidth\undefined%
    \setlength{\unitlength}{208.32100215bp}%
    \ifx\svgscale\undefined%
      \relax%
    \else%
      \setlength{\unitlength}{\unitlength * \real{\svgscale}}%
    \fi%
  \else%
    \setlength{\unitlength}{\svgwidth}%
  \fi%
  \global\let\svgwidth\undefined%
  \global\let\svgscale\undefined%
  \makeatother%
  \begin{picture}(1,0.83936143)%
    \lineheight{1}%
    \setlength\tabcolsep{0pt}%
    \put(0,0){\includegraphics[width=\unitlength,page=1]{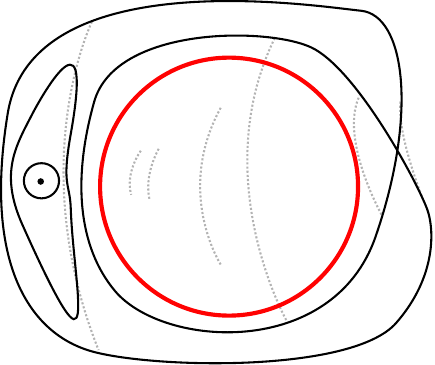}}%
    \put(0.26828896,0.64206108){\color[rgb]{1,0,0}\makebox(0,0)[lt]{\lineheight{0}\smash{\begin{tabular}[t]{l}$\partial B$\end{tabular}}}}%
    \put(0.06275359,0.48035431){\color[rgb]{0,0,0}\makebox(0,0)[lt]{\lineheight{0}\smash{\begin{tabular}[t]{l}$p$\end{tabular}}}}%
    \put(0.91386699,0.40104179){\color[rgb]{0,0,0}\makebox(0,0)[lt]{\lineheight{0}\smash{\begin{tabular}[t]{l}$q$\end{tabular}}}}%
    \put(0,0){\includegraphics[width=\unitlength,page=2]{Fig43.pdf}}%
    \put(0.73903253,0.43228118){\color[rgb]{0,0,0}\makebox(0,0)[lt]{\lineheight{0}\smash{\begin{tabular}[t]{l}$p'$\end{tabular}}}}%
    \put(0.39893502,0.44757173){\color[rgb]{1,0,0}\makebox(0,0)[lt]{\lineheight{0}\smash{\begin{tabular}[t]{l}$\partial B'$\end{tabular}}}}%
    \put(0,0){\includegraphics[width=\unitlength,page=3]{Fig43.pdf}}%
    \put(0.44990519,0.38767034){\color[rgb]{0,0,0}\makebox(0,0)[lt]{\lineheight{0}\smash{\begin{tabular}[t]{l}$p''$\end{tabular}}}}%
  \end{picture}%
\endgroup%

\caption[Finding a cone that can cancel a dot.]{Starting from dot $p$ and flowing perpendicular to the singular fibration $\F$, the next singular leaf we find has a cone $q$ corresponding to an index-$2$ critical point. A nearby leaf component is a sphere bounding ball $B$ containing dot $p'$. Starting from $p'$ and flowing perpendicular to the leaves of $\F$, again the first cone corresponds to a critical point of index-$2$. A nearby leaf component is a sphere bounding ball $B'\subset B$ containing dot $p''$.}
\label{fig:noindex2}
\end{centering}\end{figure}

Let $p,q$ satisfy Proposition~\ref{notindex2}. Let $\eta:[0,1]\to V$ be an arc from $q$ to $p$ intersecting the leaves of $\F_1$ 
transversely. Proposition~\ref{notindex2} exactly ensures that we may play a simple cancellation movie around $\eta$ 
to extend $\F_1$ to a valid movie on $t\in[1-\epsilon,1]$ so that $\F_{1-\epsilon}$ 
has $r-1$ cone and $r-1$ dot singularities. Proceed inductively to extend $\F_t$ to a valid movie on $t\in[0,1]$ so that $\F_0$ is nonsingular. This completes the proof of Lemma~\ref{simplecancel}.
\end{proof}

In the setting of Theorem~\ref{maintheorem}, $\F_{s_n-\epsilon}$ is a singular fibration of a solid torus satisfying the hypotheses of Lemma~\ref{simplecancel}. Therefore, we may take all the cancellation movies played from $t=s_n-\epsilon$ to $s_{n+1}+\epsilon$ to be simple.

\subsection{Understanding the fibers of $B^4\setminus\nu(D)$}\label{geomcancel}
In this subsection, we explicitly show that the fibers of $\G$ are handlebodies, via the construction of Theorem~\ref{maintheorem} and Theorem~\ref{fibrationthm}.

Let $H$ be the fiber of $B^4\setminus\nu(D)$ obtained from Theorem~\ref{maintheorem}. Then $F=H\cap\boundary B^4$ is a leaf in the fibration of $S^3\setminus\nu(K)$.
Explicitly, $H\cong F\times I\cup\{1$-,$2$-,$3$-handles$\}$, where the $1$-handles correspond to the type I cones which appear and vanish during a band movie. 

If a minimum disk $d_i$ intersects the interior of band $b_j$ (when both are projected to the same height), then when a disk movie is played near $d_i$, we obtain $2$-handles in $H$ geometrically cancelling all the $1$-handles corresponding to the band movie for band $b_j$. See Figure~\ref{fig:handlescancel}. We must now only be careful that a different minimum disk cancel the $1$-handles for each band.

\begin{figure}\begin{centering}
\includegraphics[width=.6\textwidth]{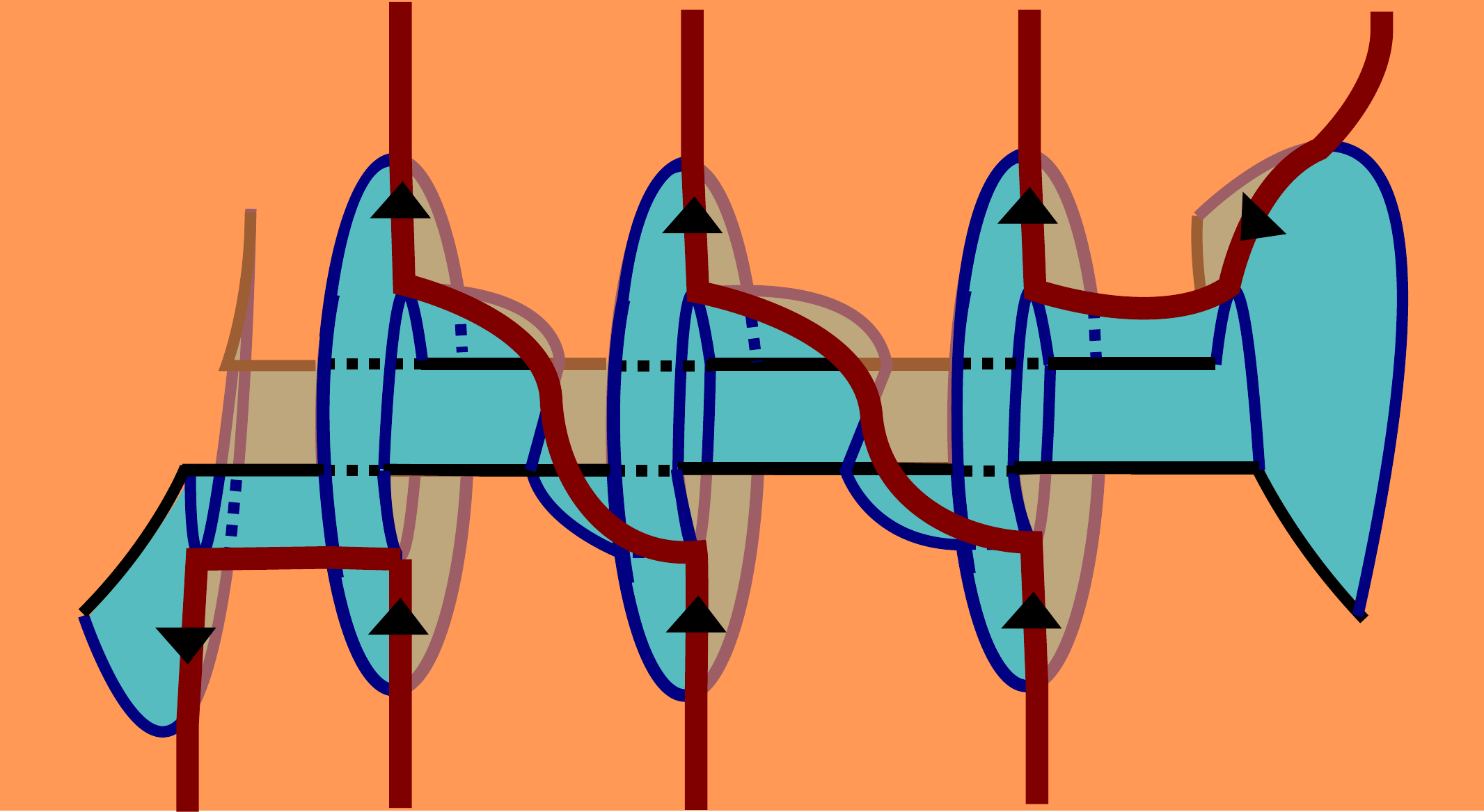}
\caption[All 1-handles in the 3-dimensional fiber of $B^4\setminus\nu(D)$ are geometrically cancelled.]{Rectangular sheet: part of the minimum disk $d_i$ that intersects a core of the band $b_j$. Surface: A leaf $L$ of $\F_t$ for $t=t_i-\epsilon$. Some curves in $L\cap d_i$ run along the cores of the $1$-handles corresponding to $b_j$ (attached to the $3$-dimensional fiber $H$). These curves are all distinct, since they locally have the same orientation in $d_i$. (Here, we are using the fact that $b_j$ is transverse to the fibration $\F_t$.) Therefore, $2$-handles attached to $H$ corresponding to the Type II cones arising during the disk movie on $d_i$ will geometrically cancel the $1$-handles corresponding to the type I cones arising during the band movie on $b_j$. }
\label{fig:handlescancel}
\end{centering}\end{figure}

 Let $d_1,\ldots, d_{n+1}$ be the minimum disks of $D$ (projected into $h^{-1}(3)=\boundary B^4=S^3$).
\begin{lemma}\label{matchingprop}
Up to reordering the minimum disks $\{d_i\}$, we may assume $d_i\cap \mathring{b_i}\neq\emptyset$.
\end{lemma}
\begin{proof}
Let $\Gamma$ be a bipartite graph with $2n+1$ vertices in $V\cup W=\{v_1,\ldots, v_{n+1}\}\cup\{w_1,\ldots, w_n\}$. Say $v_i$ and $w_j$ are adjacent if $d_i\cap \mathring b_j\neq\emptyset$. If $\boundary d_i\cap\boundary b_j\neq\emptyset$ but $d_i\cap\mathring{b_j}=\emptyset$, then we may isotope $D$ to cancel $d_i$ and $b_j$ (without moving the other bands). Assume we have cancelled all such pairs. 

Let $X\subset W$ be a set of $k\le n$ vertices. Let $N_\Gamma(X)$ denote the subset of $V$ which is adjacent to $X$. Suppose $|N_\Gamma(X)|<k$. Then some band $b_i$ with $w_i\in X$ is a fusion band, rather than fission, yielding a contradiction. Thus, by Hall's Marriage Theorem, there exists a $W$-saturated matching of $\Gamma$. That is, we can relabel the disks $d_i$ (and corresponding vertices $w_i$) so that $d_i\cap\mathring{b_i}\neq\emptyset$ (i.e.\ $v_i$ is adjacent to $w_i$).
\end{proof}

By Proposition~\ref{matchingprop}, a collection of $2$-handles in the fiber $H$ of $\G$ on $B^4\setminus \nu(D)$ (resulting from the minimum disk $d_i$) geometrically cancel all the $1$-handles corresponding to the band $b_i$. Therefore, $H\cong(\boundary H\times I)\cup\{2$-,$3$-handles$\}$, so $H$ is a handlebody.

\subsection{Proof of Theorem~\ref{secondtheorem}}

\begin{secondtheorem}
Let $K$ be a fibered knot in $S^3$. Let $J\subset S^3$ be a knot so that there is a {\emph{ribbon concordance}} from $K$ to $J$ (i.e.\ there is an annulus $A$ properly embedded in $S^3\times I$ with boundary $(K\times 1)\sqcup (\overline{J}\times 0)$ so that $\proj_I|_{A}$ is Morse with no local maxima). 
Say the index-$1$ critical points of $A$ correspond to fission bands attached to $K$. If the bands can be isotoped to be transverse to the fibration on $S^3\setminus\nu(K)$, then $(S^3\times I)\setminus\nu(A)$ is fibered by compression bodies.
\end{secondtheorem}

\begin{proof}
We proceed as in Theorem~\ref{maintheorem}. Let $A$ be an annulus in $S^3\times I$ with boundary $K\times 1\sqcup\overline{J}\times 0$, with no local maxima with respect to projection on $I$. Say the local minima of $A$ occur in $S^3\times s_i$ for $s_1>s_2>\ldots> s_n\in[2/7,3/7]$ and the index-$1$ critical points in $S^3\times t_i$ for $t_1>t_2>\ldots>t_n\in[5/7,6/7]$. By hypothesis, we assume that bands corresponding to the index-$1$ critical points are transverse to the fibration $\F_1$ on $(S^3\times 1)\setminus\nu(K)$ (when projected to $S^3\times 1$).

We build a movie of singular fibrations $\F_t$ on $(S^3\times I)\setminus\nu(A)$ as in Theorem~\ref{maintheorem}. Extend $\F_t$ to $(S^3\times[4/7,1])\setminus\nu(A)$ by playing a band movie near each band of $A$. Then extend to $(S^3\times[1/7,1])\setminus\nu(A)$ by attaching playing a disk movie near each minimum disk of $A$. There are an equal number of cone and dot critical points in $\F_{1/6}$, so we may play generalized cancellation movies to extend $\F_t$ to $(S^3\times I)\setminus\nu(A)$ with $\F_0$ a nonsingular fibration. (In fact, these generalized cancellations could be taken to be simple by Lemma~\ref{simplecancel}.) 
Thus, $J$ is fibered.

By Lemma~\ref{usemoviealt}, $(S^3\times I)\setminus\nu(A)$ admits a fibration $\G$ with fiber $H$. Via the argument of subsection~\ref{geomcancel}, $H\cong ([H\cap (S^3\times 1)]\times I )\cup\{2$- and $3$-handles$\}$. Therefore, $H$ is a compression body from the fiber for $K$ to the fiber for $J$.

Note $H\cong (\mathring{\Sigma}_{g(K)}\times I)\cup\{g(K)-g(J)$ $2$-handles$\}\cong\natural_{g(K)+g(J)} S^1\times D^2$.

\end{proof}

\section{Conclusion: explicitly finding the fiber of $B^4\setminus\nu(D)$\label{sec:conclusion}}

Let $F$ be a fiber for fibered knot $K$. Let $D\subset B^4$ be a ribbon disk for $K$ so that fission bands $b_i$ for $D$ are transverse to the fibration of $S^3\setminus\nu(K)$. (Recall from Section~\ref{ch:intro} that while $D$ is properly embedded in $B^4$, $D$ can be defined by fission bands attached to $K$ in $S^3$.) By Theorem~\ref{maintheorem}, the fibration extends to a fibration $\G$ on $B^4\setminus\nu(D)$. One fiber of $\G$ has boundary $F\cup D$; call this fiber $H$.

In the proof of Theorem~\ref{maintheorem}, we constructed $H$ by building a valid movie $\F_t, t\in[s_{n+1}+\epsilon,3]$ on $B^4\setminus\nu(D)$ so that $H=\cup_t \F_t^{-1}(\theta))\cup (3$-handle$)$ for some $\theta_0$ (in particular, $F=\F_3^{-1}(\theta_0)$). The $2$-handles of $H$ correspond to the type II cones and type \btwo bowls of $\F_t$ intersecting the vertical $\theta=\theta_0$ line in a valid chart $\mathcal{C}$ for $\F_t$. 

Recall $\F_t$ may be built from band, disk, and simple (see subsection~\ref{sec:simplecancel}) cancellation movies. Say that from $t=3$ to $2$, $\F_t$ is built from band movies; from $t=2$ to $1$, $\F_t$ is built from disk movies; from $t=1$ to $s_{n+1}+\epsilon$, $\F_t$ is built from simple cancellation movies.
In $\mathcal{C}$, every type II cone arises during a band movie or a disk movie. The type II cones arising in a band movie correspond to arcs which are nearly vertical until $t=1$. From $t=1$ to $t=s_{n+1}+\epsilon$, during simple cancellation movies all cones correspond to arcs which are nearly vertical. Therefore, we may choose $\theta_0$ so that the vertical $\theta=\theta_0$ line intersects type II arcs only in $1<t<2$, does not intersect the arcs of type II cone singularities arising from band movies, and does not meet any type \btwo singularities. (Perhaps after reparametrizing $\F_t$ so that type \btwo and II boundary arcs are short.) See Figure~\ref{fig:fullchart1} (left).

\begin{figure}\begin{centering}
\scalebox{0.9}{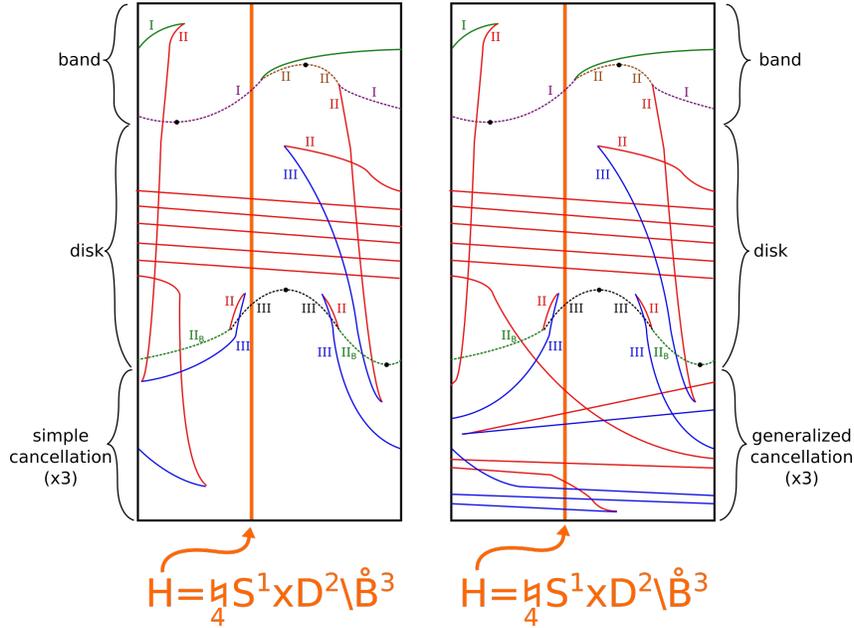}
\caption[Valid singularity chart of the fibration constructed on a ribbon disk complement.]{We draw potential valid singularity charts for $\F_t$, $t\in[s_{n+1},3]$. By Lemma~\ref{simplecancel}, we may assume that the cancellations of $\F_t$ are all simple, as on the left. Then all $2$-handles of $H$ arise from positioning a cone or playing a minimum movie during a disk movie. Therefore, the $2$-handles are attached exactly along the intersections of $\F_t^{-1}(\theta)$ with a minimum disk, projected to the same height.}
\label{fig:fullchart1}
\end{centering}\end{figure}

To recap, this means that every type II or \btwo singularity meeting the line $\theta=\theta_0$ in $\mathcal{C}$ arises during a disk movie. In $H\cap h^{-1}(t)$, as $t$ decreases, passing through these singularities corresponds to compressing $H\cap h^{-1}(t)$ along a closed curve of intersection with a minimum disk of $D$. We conclude that intersections of the minima of $D$ (projected to $S^3$) with $H\cap h^{-1}(3)=F$ give attaching circles of the $2$-handles of $H$. (This decomposition is not minimal if the line $\theta=\theta_0$ meets type III arcs in $\mathcal{C}$, but we may choose a set of linearly independent curves to give $2$-handles if we wish.)

Therefore, to find the $2$-handle attaching curves on $F\cup D$ defining $H$, we do the following:
\begin{itemize}
\item Find a disjoint collection of disks $d_1\sqcup \cdots\sqcup d_n$ in $S^3$ bounded by $K$ resolved along the $b_i$.
\item Isotope the $d_i$ so that near the interior of $\boundary d_i\cap\boundary F$, $d_i$ and $F$ coincide.
\item Draw a (maximal linearly independent) collection of closed curves $d_i\cap F$ on $F$. These are $2$-handle attaching curves on $F\cup D$ which yield $H$.
\end{itemize}

We could equivalently state this process as:
\begin{itemize}
\item Project $D$ to $S^3$ to find an immersed ribbon disk $D'$ with boundary $K$.
\item Isotope $D'$ so that near $\boundary D'=\boundary F$, $F$ coincides with the sheet of $D'$ containing the boundary.
\item Draw a (maximal linearly independent) collection of closed curves in $D'\cap F$ on $F$. These are $2$-handle attaching curves on $F\cup D$ which yield $H$.
\end{itemize}

In Figure~\ref{fig:820}, we give explicitly a fibered ribbon disk and use the above procedure to find attaching circles for $2$-handles in the handlebody fiber. In Figure~\ref{fig:11n74} we repeat this process for two fibered ribbon disks with the same boundary to find the fiber of a fibered $2$-knot in $S^4$.

\begin{figure}\begin{centering}
\includegraphics[width=.65\textwidth]{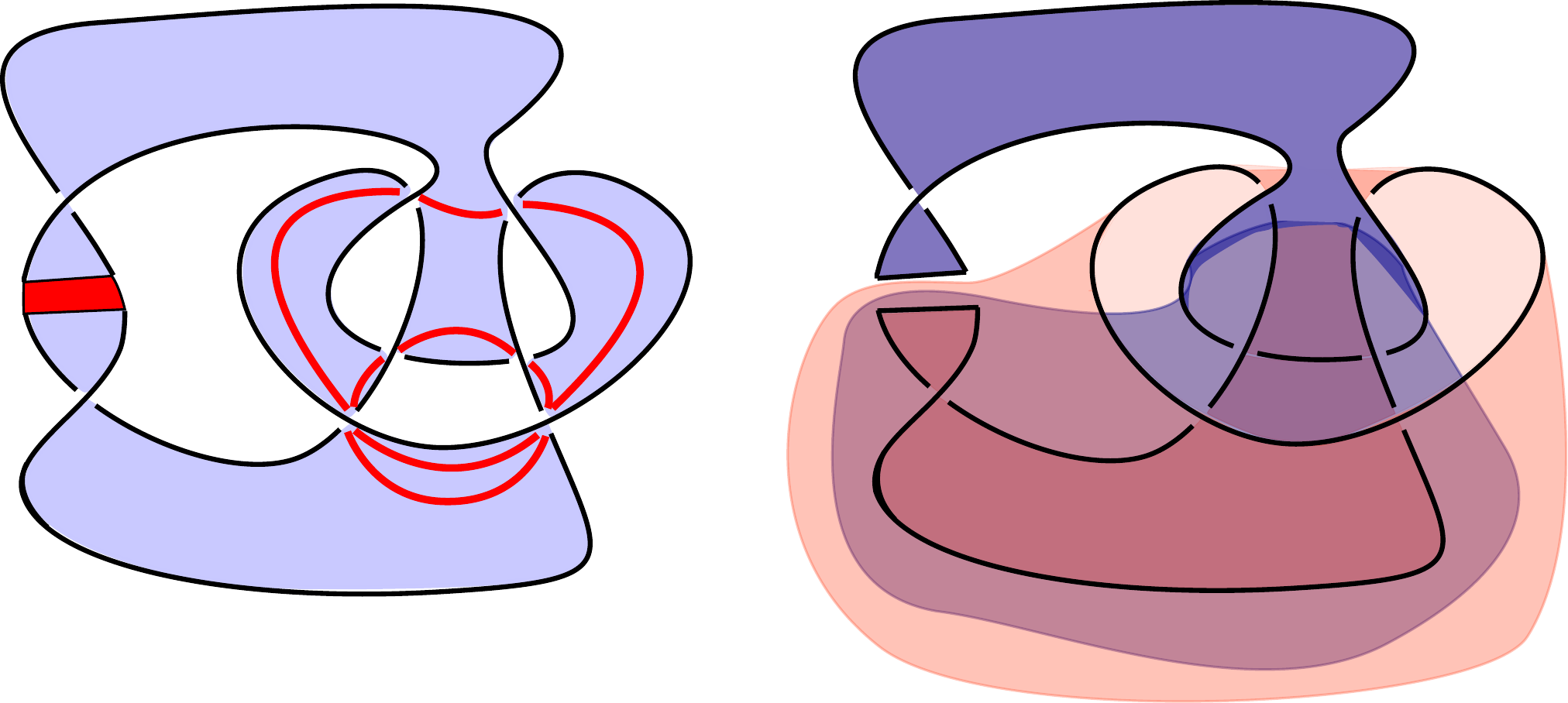}
\caption[Fibering the complement of a ribbon disk bounded by $\mathtt{8_{20}}$.]{Left: The fibered knot $K=\mathtt{8_{20}}$ (i.e.\ the pretzel knot $P(2,-3,3)$). We draw a fission band defining a ribbon disk $D$ with boundary $K$ and shade a genus-$2$ fiber surface for $K$. On the fiber, we draw two attaching circles defining a genus-$2$ handlebody fiber of $B^4\setminus\nu(D)$. Right: Two disjoint disks bounded by $K$ resolved along the fission band. These disks intersect the fiber in arcs (near the band) and two curves, which are the attaching circles drawn on the left.}
\label{fig:820}
\end{centering}\end{figure}

\begin{figure}\begin{centering}
\includegraphics[width=.85\textwidth]{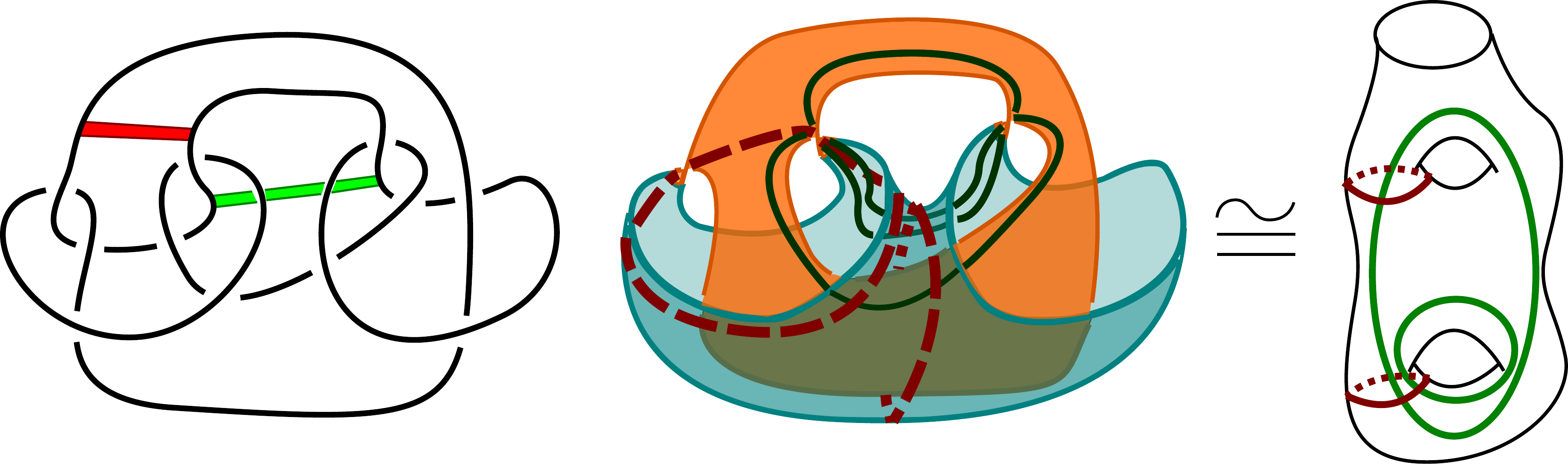}
\caption[Fibering the complements of two distinct ribbon disks bounded by $\mathtt{11n_{74}}$.]{Left: A diagram for the fibered ribbon knot $K=\mathtt{11n_{74}}$ (i.e.\ the pretzel knot $P(3,-3,3,-2)$). Each of the pictured bands specifies a $2$-minimum ribbon disk $D_i$ for $K$. (These examples of disks come from~\cite{doublyslice}, who produce many examples of doubly slice knots.) Using the construction of Theorem~\ref{maintheorem}, we find attaching circles for the handlebody fiber for each disk. Middle: The fiber $F$ for $K$ consists of two annuli glued along three half-twisted bands. The two sets of circles specify the handlebody fiber for the disks given by the bands. Right: On $F\cup D$, the attaching circles specify a Heegaard splitting of $S^3$. Thus, $D_1\cup\overline{D_2}\subset S^4$ is fibered by $3$-balls (so $D_1\cup\overline{D_2}$ is the unknotted sphere).}
\label{fig:11n74}
\end{centering}\end{figure}

\section{More examples\label{sec:examples}}
It remains open whether any/every ribbon disk for a fibered knot $K$ in $S^3$ can be defined by bands transverse to the fibration of $S^3\setminus\nu(K)$. However, for specific examples of $K$ we can check that this condition is satisfied by some ribbon disk $D$.

In particular, we find fibered ribbon disks for small fibered ribbon knots.
\begin{proposition}
Let $K$ be a prime fibered ribbon knot with crossing number at most $12$. Then $K$ bounds a ribbon disk $D\subset B^4$ so that $S^4\setminus\nu(D)$ is fibered by handlebodies.
\end{proposition}
\begin{proof}
By Corollary~\ref{diskcor}, if $K$ bounds a disk $D$ with two minima, then the fibration on $S^3\setminus\nu(K)$ extends to a fibration of $B^4\setminus\nu(D)$. In particular, every prime ribbon knot of at most $10$ crossings bounds a $2$-minimum disk~\cite{kawauchi}, as do most prime fibered knots of $11$ and $12$ crossings (observed through e.g. symmetric union diagrams~\cite{lamm}). Using the KnotInfo database~\cite{livingston} to list fibered slice knots, and symmetric union diagrams of these knots due to Lamm~\cite{lamm}, we explicitly produce a $3$-minimum disk $D$ for each $11$- or $12$-crossing prime fibered ribbon knot which does not bound an obvious-to-the-author $2$-minimum disk. The bands for each disk either lie in a fiber for the knot $K\subset S^3$, or run transverse to the fibration. See Figure~\ref{fig:verification}. Transversality of a band can be seen by observing the band core is given by taking an arc in a fiber and adding twists (of one sign) around the binding (knot).

\begin{remark}
In the first draft of this paper, the author was not aware of the following examples of ribbon disks due to Brendan Owens and Frank Swenton (as of this writing, they have not yet released a paper including these disks, but Brendan Owens kindly let this author know of their examples). They have constructed (in particular) examples of $2$-minimum disks bounded by each of the knots in Figure~\ref{fig:verification} except for $\mathtt{12a_{990}}$. (These are each potentially isotopic to the $3$-minimum examples.) We retain the $3$-minimum examples as illustrations of how one can verify the transversality condition of Theorem~\ref{maintheorem}. In Figure~\ref{fig:brendansdisks}, we provide the $2$-minimum disks of Owens and Swenton.
\end{remark}

\end{proof}

\begin{figure}\begin{centering}
\scalebox{0.9}{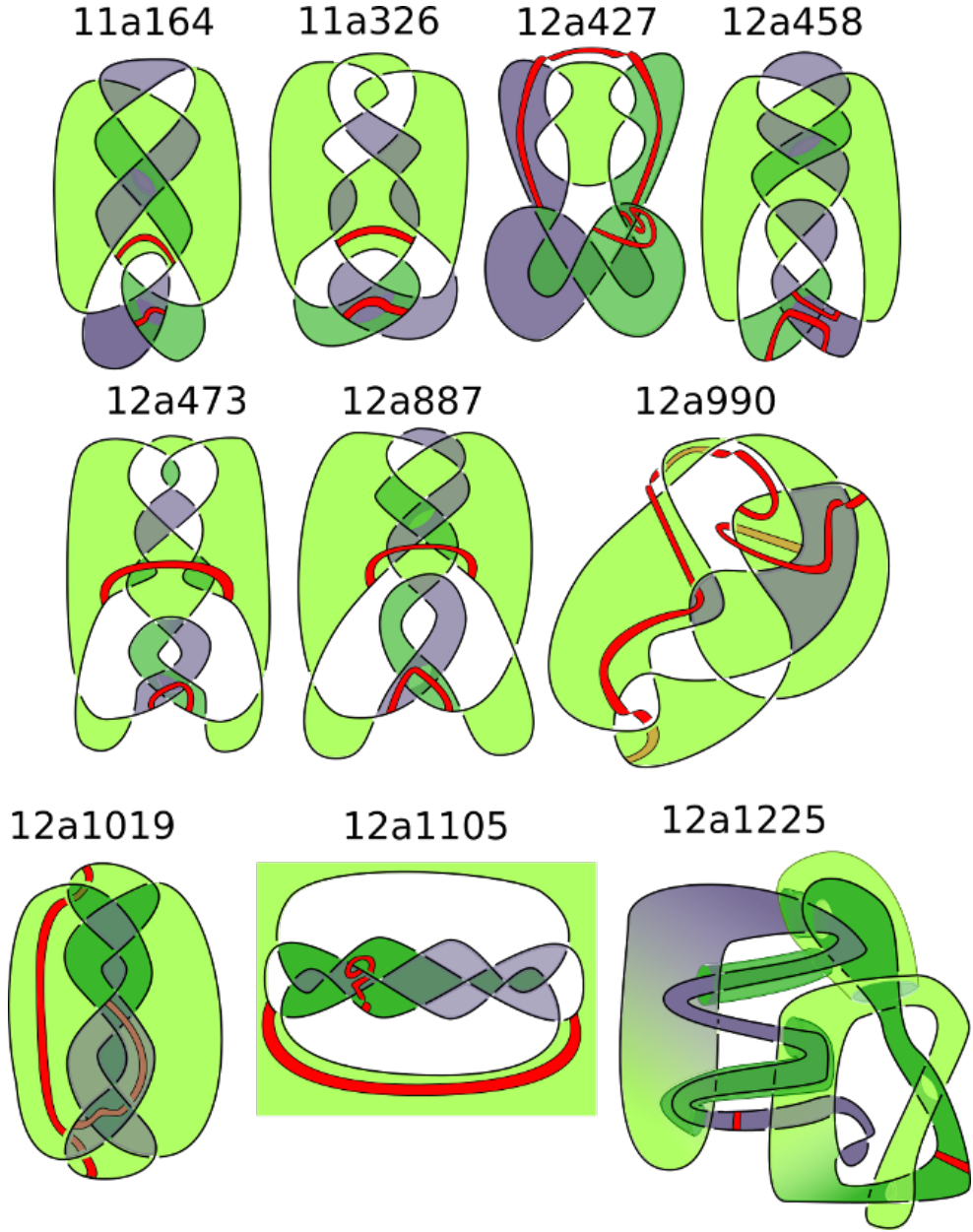}
\caption[Some 3-minimum ribbon disks satisfying the transversality condition of Theorem~\ref{maintheorem}.]{These prime fibered ribbon knots did not obviously (to the author) bound $2$-minimum disks in $B^4$. We produce here a $3$-minimum ribbon disk for each knot. The bands for $\mathtt{11a_{164}}, \mathtt{11a_{326}}, \mathtt{12a_{458}}, \mathtt{12a_{473}}, \mathtt{12a_{887}}, \mathtt{12a_{1225}}$ each lie in the illustrated fiber. Knots $\mathtt{12a_{427}}, \mathtt{12a_{101}}, \mathtt{12a_{1105}}$ are drawn with one band in the fiber and one transverse to the fibration. Both bands for $\mathtt{12a_{990}}$ run transverse to the fibration.}
\label{fig:verification}
\end{centering}\end{figure}

\begin{figure}\begin{centering}
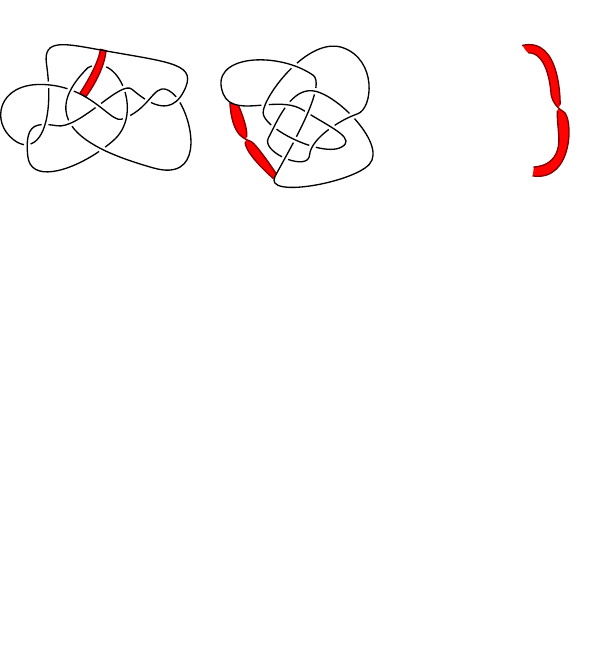
\caption[2-minimum disks of Owens--Swenton for all but one of the knots in Figure~\ref{fig:verification}.]{A band diagram of a $2$-minimum disk for each knot of Figure~\ref{fig:verification} except $\mathtt{12a_{990}}$, constructed by B. Owens and F. Swenton.}
\label{fig:brendansdisks}
\end{centering}\end{figure}

\begin{question}
Does $\mathtt{12a_{990}}$ bound a $2$-minimum ribbon disk?
\end{question}

Note here we defined ``small'' knots to be prime knots with fewer than 13 crossings. One might reasonably consider ``small'' knots to be $2$-bridge knots.
In Lisca's~\cite{lisca} classification of ribbon $2$-bridge knots, he produces an explicit ribbon disk for every ribbon $2$-bridge knot. Most of these disks have two minima, except for the examples in Figure~\ref{fig:2bridge} (left). By drawing the ribbon knots as symmetric unions (using diagrams of Cristoph Lamm{\footnote{This preprint is available online as C. Lamm, {\emph{Symmetric union presentations for $2$-bridge ribbon knots}}, arXiv:math/0602395 [math.GT], Feb. 2006}}), we find ribbon disks with two minima (Figure~\ref{fig:2bridge}, right). Therefore, by Corollary~\ref{diskcor}, every fibered ribbon $2$-bridge knot bounds a fibered ribbon disk.

\begin{figure}\begin{centering}
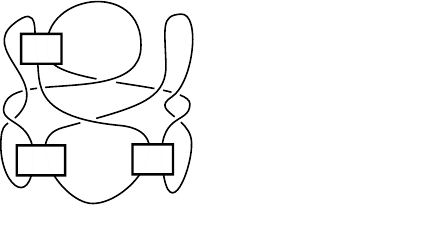
\caption[Ribbon 2-bridge knots admit ribbon disks with exactly two minima.]{Left: bands defining a ribbon disk for $K\subset S^3$, constructed by Lisca~\cite{lisca}. Right: A $2$-minimum ribbon disk with the same boundary knot. Note here that the integers $\pm2a,\pm2b$ refer to half-twists. The knot $K$ is fibered if and only if $|a|=|b|=1.$}
\label{fig:2bridge}
\end{centering}\end{figure}

However, we note that in this setting, the $2$-bridge notion of ``small'' is actually less interesting than low crossing-number. In fact, there are finitely many fibered ribbon $2$-bridge knots. We are not aware of this fact in the literature, so we include this as an interesting remark.

\begin{remark}
Up to mirroring, there are exactly five fibered ribbon $2$-bridge knots: $\mathtt{8_9, 9_{27}, 10_{42}, 11a_{96}, 12a_{477}}$.
\end{remark}
\begin{proof}
By Casson and Gordon~\cite{gordon2bridge}, every ribbon $2$-bridge knot (up to mirroring) has a continued fraction expansion in one of the following three families:
\begin{itemize}
\item Family 1: $[c_1, -c_2,  c_3, -c_4, \ldots,  c_{2n+1}, -1, -c_{2n+1}, \ldots, c_4, -c_3, c_2$, $-c_1]$ where each $c_i>0$, 
\item Family 2: $[2a, 2, 2b, -2, -2a, -2b]$,  $a, b\neq 0$, 
\item Family 3: $[2a, 2, 2b, -2a, -2, -2b]$,  $a, b\neq 0$.
\end{itemize}

As this is somewhat of a side note from the main thrust of this paper, we omit exposition of continued fraction expansion notation for $2$-bridge knots. For those interested, we recommend section 2 of~\cite{murasugi}.

\begin{proposition}
A knot $K$ in Family 1 is never fibered.
\end{proposition}
\begin{proof}
Let $[d_1,\ldots, d_m]$ be a {\emph{strict}} continued fraction for $[-c_{2n+1}, \ldots$, $c_4$, $-c_3$, $c_2$, $-c_1]$, as in~\cite{murasugi} (i.e.\, all $d_i$ are nonzero, $d_{2i+1}$ is even and if $|d_{2i+1}|=2$ then $d_{2i+1}d_{2i+2}<0$). Then $K$ has continued fraction expansion $[-d_m,\ldots$,$-d_1$, $-1$, $d_1,\ldots$, $d_m]$. For odd $m$, this fraction gives a $2$-component link, so $m$ must be even. We draw a minimum-genus Seifert surface $F$ for $K$ in Figure~\ref{fig:familyone}. From $F$, we can deplumb an annulus with $d_1-$sign$(d_2)+1$ half-twists and an annulus with $-(d_1-$sign$(d_2))+1$ half-twists. These cannot both be Hopf bands, so $F$ is not a fiber.
\end{proof}

\begin{figure}\begin{centering}
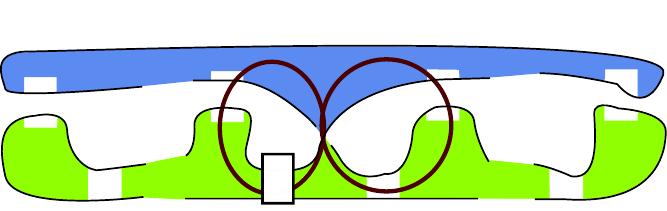
\caption[There are only 5 fibered, ribbon 2-bridge knots (Family 1).]{A minimum-genus Seifert surface for $2$-bridge ribbon knots in Family 1. From this surface, we can deplumb annuli with $d_1-$sign$(d_2)+1$ and $-d_1+$sign$(d_2)+1$ half-twists. These cannot both be Hopf bands, so the surface is not a fiber (and hence the knot is not a fibered knot).}
\label{fig:familyone}
\end{centering}\end{figure}

In Figure~\ref{fig:family23}, we draw a minimum-genus Seifert surface for a knot in Family 2 or 3. In either case, the surface decomposes as a Murasugi sum of two Hopf bands and 4 annuli with $2a,2b,-2a,-2b$ half-twists, respectively. Therefore, the knot is fibered if and only if $|a|=|b|=1$.

\begin{figure}\begin{centering}
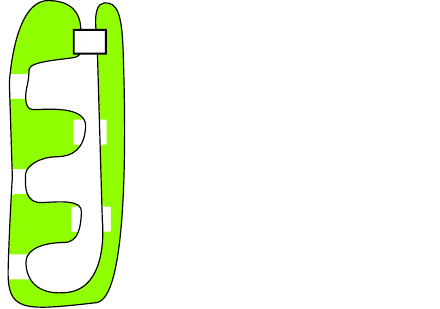
\caption[There are only 5 fibered, ribbon 2-bridge knots (Families 2 and 3).]{Minimum-genus surfaces for knots in Families 2 and 3. These are fibers if and only if $|a|=|b|=1$, so a knot in Family 2 or3 is fibered if and only if $|a|=|b|=1$.}
\label{fig:family23}
\end{centering}\end{figure}

The Family 2 knot with $a=1,b=-1$, and the Family 3 knots with $\{a,b\}=\{1,-1\}$ are all  $\mathtt{9_{27}}$. The Family $2$ knot with $a=1,b=1$ and the Family $3$ knot with $a=1,b=1$ are both $\mathtt{8_9}$. There are no other repeats, so we are left with five distinct $2$-bridge knots. Using KnotFinder~\cite{livingston}, we identify these knots as $\mathtt{8_9, 9_{27}, 10_{42}, 11a_{96}, 12a_{477}}$.
\end{proof}

\section{More questions and final comments\label{sec:questions}}
Recall Question~\ref{mainquestion}, the main motivation for this paper.

\begin{mainquestion}
Let $D$ be a homotopy-ribbon disk in $B^4$ so that $\boundary D\subset S^3$ is a fibered knot. Is $B^4\setminus\nu(D)$ fibered?
\end{mainquestion}

This naturally leads to the (potentially) easier question:
\begin{question}\label{constructive}
Let $K$ be a fibered homotopy-ribbon knot in $S^3$. Does $K$ bound a fibered homotopy-ribbon disk in $B^4$?
\end{question}

We remark that if the answer to Question~\ref{constructive} is ``no'' for any knot $K$, then the relative smooth $4$-dimensional Poincar\'e conjecture is false.
\begin{proof}
Assume $K$ is fibered and homotopy-ribbon in $S^3$. By Casson and Gordon~\cite{gordon}, $K$ bounds a disk $D$ fibered by handlebodies in a homotopy $4$-ball $V$. By Larson and Meier~\cite{jeff}, if $V\cong B^4$ then $D$ is strongly homotopy-ribbon.
\end{proof}
In Remark~\ref{strongremark}, we claimed that the converse of Larson and Meier's theorem holds even when ``strongly homotopy-ribbon'' is relaxed to ``homotopy-ribbon.'' That is, if a disk $D$ in $B^4$ is homotopy-ribbon and fibered, then the fibers are handlebodies. This follows easily from the proof of~\cite{jeff} and a paper by Cochran~\cite{cochran} cited therein, but we sketch the proof here as a matter of interest. This proof closely follows arguments in both of these papers; one should refer to~\cite{jeff} or~\cite{cochran} for more detail and exposition on disk-knots in $B^4$ or ribbon $2$-knots in $S^4$, respectively.

\begin{proof}
We make use of the following lemma of Cochran: Let $M$ be a closed $3$-manifold. Then the natural map $i^\#:M\to K(\pi_1(M),1)$ induces the zero map $i_*:H_3(M)\to H_3(\pi_1(M))$ if and only if $M=\#_h S^1\times S^2$ for some $h\ge 0$.

Let $D$ be a fibered homotopy-ribbon disk in $B^4$, with fiber $H$. Then $K=\boundary D$ is a fibered knot, with fiber $F$.

Let $W=B^4\setminus\nu(D)$ and $\widetilde{W}$ be the infinite cyclic cover of $W$. Then $\widetilde{W}\cong H\times\R, \boundary\widetilde{W}\cong F\times\R$. Thus, $\pi_1(H)\cong[\pi_1(W),\pi_1(W)]$ and $\pi_1(F)\cong[\pi_1(S^3_0(K)),$ $\pi_1(S^3_0(K))]$. Since $D$ is homotopy-ribbon, the inclusion map $\pi_1(S^3\setminus K)\to\pi_1(W)$ is surjective. Therefore, the inclusion $\pi_1(F)\to\pi_1(H)$ is surjective.

Let $G$ be the $2$-sphere $D\cup\overline{D}$ in $S^4=B^4\cup\overline{B^4}$. Then $S^4\setminus\nu(G)$ is fibered, with fiber $M'=H\cup_{F=\overline{F}}\overline{H}$. Let $M=M'\cup B^3$ be a closed $3$-manifold. By Seifert van-Kampen, the inclusion $H\into M$ induces an isomorphism $\pi_1(H)\cong\pi_1(M)$.

Let $X$ be obtained by surgering $S^4$ along $G$ (i.e. deleting $\nu(G)$ and gluing in $S^1\times B^3$). Let $A=(B^4\times I)\setminus(D\times I)$ be a $5$-manifold with $\boundary A\cong X$. The inclusion $\widetilde{X}\into\widetilde{A}$ induces an isomorphism $\pi_1(\widetilde{X})\cong\pi_1(\widetilde{A})$, since $\pi_1(\widetilde{A})=\pi_1(\widetilde{B^4\setminus D})\cong\pi_1(H)\cong\pi_1(M)\cong\pi_1(\widetilde{X})$. We obtain the following commutative diagram (left), which induces a commutative map of spaces (right).

\begin{center}
\begin{tabular}{cc}
\begin{tikzpicture}
  \matrix (m) [matrix of math nodes,row sep=3em,column sep=2.5em,minimum width=2em]
  {
     \pi_1(\widetilde{A}) &  \\
     \pi_1(\widetilde{X})& \left[\pi_1(X),\pi_1(X)\right] \\};
  \path[-stealth]
    (m-2-1) edge node [left] {$i_*$} (m-1-1)
            edge node [below] {id} (m-2-2)
    (m-1-1) edge [dashed] (m-2-2);
\end{tikzpicture}
&
\begin{tikzpicture}
  \matrix (m) [matrix of math nodes,row sep=3em,column sep=2.5em,minimum width=2em]
  {
     \widetilde{A} &  \\
     \widetilde{X}& K(\left[\pi_1(X),\pi_1(X)\right],1) \\};
  \path[-stealth]
    (m-2-1) edge node [left] {$i$} (m-1-1)
            edge node [below] {$i^{\#}$} (m-2-2)
    (m-1-1) edge [dashed] (m-2-2);
\end{tikzpicture}
\end{tabular}
\end{center}

Since $H_3(\widetilde{A})=H_3(\widetilde{B^4\setminus D})\cong H_3(H\times\mathbb{R})\cong H_3(H)=0$ and $i_*:H_3(\widetilde{X})\to H_3(\pi_1(X),\pi_1(X))$ factors through $H_3(\widetilde{A})$, $i_*$ is the zero map. But $\widetilde{X}\cong \R\times \widehat{M}$ and $[\pi_1(X),\pi_1(X)]\cong\pi_1(M)$, so $i_*:H_3(M)\to H_3(\pi_1(M))$ is the zero map. Therefore, $M\cong\#_g S^1\times S^2$ for some $g$. Since $\pi_1(H)\cong\pi_1(M)$, $g$ is the genus of $F$ and $H$ is a genus-$g$ handlebody.

\end{proof}

The following more restrictive version of Question~\ref{constructive} may be easier to answer.
\begin{question}\label{constructive2}
Let $K$ be a fibered ribbon knot in $S^3$. Does $K$ bound a fibered ribbon disk in $B^4$?
\end{question}
In section~\ref{sec:examples}, we gave an affirmative answer to Question~\ref{constructive2} when $K$ is prime with fewer than $13$ crossings.

Let $K$ be a ribbon knot. Let $S$ be a minimum-genus Seifert surface for $K$. Let $D$ be a ribbon disk for $K$, and project $D$ to $S^3$ to find an immersed ribbon disk $D'$ in $S^3$ bounded by $K$. Isotope $S$ so that $S$ agrees with one sheet of $D'$ near $K$. 
Without $K$ and the ribbon disk $D$ obeying the hypotheses of Theorem~\ref{maintheorem}, it is not true that $\{D'\cap S\}$ must contain $g(S)$ linearly independent curves in $H_1(S;\Z)$ (such a set of curves is called a ``ribbon derivative" for $K$ on $S$). Cochran and Davis~\cite{cochrandavis} constructed genus-1 ribbon knots whose genus-1 Seifert surfaces do not admit ribbon derivatives. Of course, these examples are not fibered (as the only genus-1 fibered knots are the two trefoils and the figure eight, none of which are slice).

\begin{question}
If $K$ is fibered (and hence $S$ is a fiber), then must $S\cap D'$ include $g(S)$ simple closed curves which are linearly independent in $H_1(S;\Z)$?
\end{question}

We used the transversality of $b_i$ to the fibration on $S^3\setminus\nu(K)$ to ensure the $1$-handles in the relative handlebody decomposition of $H$ are all geometrically cancelled. Without the transversality condition, we might amend the band movie to find a $3$-manifold bounded by $D$ (although not a fibration of $B^4\setminus\nu(D)$). However, we have not shown that a general ribbon disk $D$ bounds a {\emph{handlebody}} in $B^4$.

\begin{question}\label{lastquestion}
Let $D\subset B^4$ be a ribbon disk for a knot $K$. What conditions on $D$ and $K$ imply that $D$ bounds a handlebody in $B^4$?
\end{question}

In the case that $\boundary D$ is the unknot, proving that $D$ bounds 3-ball as in Question~\ref{lastquestion} would positively answer the following conjecture\footnote{The author strongly recommends the reference~\cite{suzuki} to anyone interested in $2$-dimensional knot theory.}.

\begin{suzuki}
Let $G$ be a $2$-knot in $S^4$ with unknotted equator. Then $G$ is smoothly unknotted.
\end{suzuki}
We remark that if a $2$-knot $G$ has unknotted equator, then $\pi_1(S^4\setminus G)\cong \Z$, so Suzuki's unknotting conjecture is a subcase of the following more general conjecture.

\begin{unknotting}
Let $G$ be a $2$-knot in $S^4$ with $\pi_1(S^4\setminus G)\cong\Z$. Then $G$ is smoothly unknotted.
\end{unknotting}
The $3$-dimensional unknotting conjecture (Dehn's lemma) was proved by Papakyriakopoulos~\cite{papa}. For $n>4$, the $n$-dimensional unknotting conjecture was proved in the topological case by Stallings~\cite{stallings} and later in the smooth case by Levine~\cite{lev1}~\cite{lev2}.

By work of Freedman and Quinn~\cite{freedman}, if $\pi_1(S^4\setminus G)\cong\Z$, then $G$ is topologically unknotted (i.e. $G$ bounds a $3$-ball topologically embedded into $S^4$). Whether $G$ is smoothly unknotted remains an open question.

\bibliographystyle{plain}
\bibliography{biblio}

\end{document}